\documentclass[10pt, final, journal]{IEEEtran} 
\usepackage{listings}
\usepackage{amsmath}
\usepackage{amsthm}
\usepackage{tikz}
\usepackage{caption}
\usepackage{array}
\usepackage{mdwmath}
\usepackage{multirow}
\usepackage{mdwtab}
\usepackage{eqparbox}
\usepackage{amsfonts}
\usepackage{tikz}
\usepackage{multirow,bigstrut,threeparttable}
\usepackage{amsthm}
\usepackage{array}
\usepackage{bbm}
\usepackage{subfigure}
\usepackage{epstopdf}
\usepackage{mdwmath}
\usepackage{mdwtab}
\usepackage{eqparbox}
\usepackage{tikz}
\usepackage{latexsym}
\usepackage{cite}
\usepackage{amssymb}
\usepackage{bm}
\usepackage{amssymb}
\usepackage{graphicx}
\usepackage{mathrsfs}
\usepackage{epsfig}
\usepackage{psfrag}
\usepackage{setspace}
\usepackage{hyperref}
\usepackage{algorithm}
\usepackage{algpseudocode}
\usepackage{stfloats}

\newtheorem{remark}{Remark}

\newtheorem{theorem}{Theorem}

\newtheorem{lemma}{Lemma} 
\newtheorem{corollary}{Corollary}

\newtheorem{construction}{Estimator Construction}

\def \bP {\mathbb{P}}
\def \bE {\mathbb{E}}

\def \bR {\mathbb{R}}

\def \cM {\mathcal{M}}

\def \spo {\mathsf{Poi}}
\def \poly {\mathsf{poly}}
\def \var {\mathsf{Var}}
\def \spo {\mathsf{Poi}}
\newcommand\argmin{\mathop{\mbox{{\rm argmin}}}\limits}

\begin{document}

\title{Minimax Estimation of the $L_1$ Distance}

\author{Jiantao~Jiao,~\IEEEmembership{Student Member,~IEEE},~Yanjun~Han,~\IEEEmembership{Student Member,~IEEE}, and Tsachy~Weissman,~\IEEEmembership{Fellow,~IEEE}
\thanks{Jiantao Jiao, Yanjun Han, and Tsachy Weissman are with the Department of Electrical Engineering, Stanford University, CA, USA. Email: \{jiantao,yjhan, tsachy\}@stanford.edu}. 
\thanks{This work was supported in part by the Center for Science of Information (CSoI), an NSF Science and Technology Center, under grant agreement CCF-0939370. The material in this paper was presented in part at the 2016 IEEE International Symposium on Information Theory, Barcelona, Spain. }
}
%
%

\date{\today}

\vspace{-10pt}

\maketitle

\begin{abstract}
We consider the problem of estimating the $L_1$ distance between two discrete probability measures $P$ and $Q$ from empirical data in a nonasymptotic and large alphabet setting. When $Q$ is known and one obtains $n$ samples from $P$, we show that for every $Q$, the minimax rate-optimal estimator with $n$ samples achieves performance comparable to that of the maximum likelihood estimator (MLE) with $n\ln n$ samples. When both $P$ and $Q$ are unknown, we construct minimax rate-optimal estimators whose worst case performance is essentially that of the known $Q$ case with $Q$ being uniform, implying that $Q$ being uniform is essentially the most difficult case. The \emph{effective sample size enlargement} phenomenon, identified in Jiao \emph{et al.} (2015), holds both in the known $Q$ case for every $Q$ and the $Q$ unknown case. However, the construction of optimal estimators for $\|P-Q\|_1$ requires new techniques and insights beyond the approximation-based method of functional estimation in Jiao \emph{et al.} (2015). 
\end{abstract}


\section{Introduction} \label{sec.intro}
\subsection{Problem formulation}
Statistical functionals are usually used to quantify the fundamental limits of data processing tasks such as data compression (e.g. Shannon entropy~\cite{Shannon1948}), data transmission (e.g. mutual information~\cite{Shannon1948}), estimation and testing~(e.g. Kullback--Leibler divergence~\cite[Thm. 11.8.3]{Cover--Thomas2006}, $L_1$ distance~\cite[Chap. 13]{Lehmann--Romano2005}), etc. They measure the difficulties of the corresponding data processing tasks and provide benchmarks for constructive algorithms. In this sense, it is of great value to obtain accurate estimates of these functionals in various problems.

In this paper, we consider estimating the $L_1$ distance between two discrete distributions $P = (p_1,p_2,\ldots,p_S), Q = (q_1,q_2,\ldots,q_S)$, which is defined as:
\begin{align}
\|P-Q\|_1 \triangleq \sum_{i = 1}^S |p_i - q_i|. 
\end{align}
Throughout we use the squared error loss, i.e., the risk function for an estimator $\hat{L}$ is defined as
\begin{equation}
R(P,Q; \hat{L}) \triangleq \bE |\hat{L}(X^n, Y^n)-\|P-Q\|_1|^2,
\end{equation}
where $(X_i,Y_i) \stackrel{\text{i.i.d.}}{\sim}P \times Q$. The maximum risk of an estimator $\hat{L}$, and the minimax risk in estimating $\|P-Q\|_1$ are defined as
\begin{align}
R_{\text{maximum}}(\mathcal{P},\mathcal{Q};\hat{L}) & \triangleq \sup_{P\in \mathcal{P},Q\in \mathcal{Q}} R(P,Q; \hat{L}), \\
R_{\text{minimax}}(\mathcal{P},\mathcal{Q}) & \triangleq \inf_{\hat{L}} \sup_{P \in \mathcal{P},Q\in \mathcal{Q}} R(P,Q; \hat{L}),
\end{align}
respectively, where $\mathcal{P},\mathcal{Q}$ are given collections (uncertainty sets) of probability measures $P$ and $Q$, respectively, and the infimum is taken over all estimators $\hat{L}$ that are functions of the empirical observations.

The $L_1$ distance is closely related to the Bayes error, i.e., the fundamental limit, in classification problems. Specifically, for a two-class classification problem, if the prior probabilities for each class are equal, then the minimum probability of error achieved using the optimal classifier is given by
\begin{align} \label{eqn.bayeserrorxy}
	L^* = \frac{1}{2} - \frac{1}{4} \| P_{X|Y = 1}- P_{X|Y = 0}\|_1,
\end{align}  
where $Y\in \{0,1\}$ indicates the class, and $P_{X|Y}$ are the class-conditional distributions. Hence, the problem of estimating $L^*$ in this classification problem is reduced to estimating the $L_1$ distance between the two class-conditional distributions $P_{X|Y = 1},P_{X|Y = 0}$ from the empirical data. In the statistical learning theory literature, most work on Bayes classification error estimation deals with the case that $P_{X|Y = 1}$ and $P_{X|Y = 0}$ are continuous distributions, and it turns out that it is very difficult to estimate this quantity in the general continuous case. Indeed, we know from~\cite[Section 8.5]{Devroye--Gyorfi--Lugosi1996probabilistic} the negative result that for every sample size $n$, any estimate of the Bayes error $\hat{L}_n$, and any $\epsilon>0$, there exist some class-conditional distributions such that $\mathbb{E}|\hat{L}_n - L^*|\geq \frac{1}{4}-\epsilon$.

This negative result shows that one needs to look at special classes of the class-conditional distributions in order to obtain meaningful and consistent estimates. In the discrete setting, the seminal work of Valiant and Valiant~\cite{Valiant--Valiant2011power} deserves special mention. They constructed an estimator for $\|P-Q\|_1$ and showed that when $S/\ln S \lesssim n \lesssim S$, it achieves $L_1$ error $\sqrt{S/(n\ln n)}$, and it takes at least $n\gg \frac{S}{\ln S}$ samples to achieve consistent estimation of $\|P-Q\|_1$. Valiant and Valiant~\cite{Valiant--Valiant2013estimating} constructed another estimator of $\|P-Q\|_1$ using linear programming which achieves the $L_1$ error $\sqrt{\frac{S}{n\ln n}}$ when $n \asymp \frac{S}{\ln S}$. We argue in this paper that the simplest estimator for $\|P-Q\|_1$, namely plugging in the empirical distribution $P_n,Q_n$ and obtaining $\|P_n-Q_n\|_1$ achieves $L_1$ error rate $\sqrt{S/n}$ for $n\gtrsim S$. In this sense, the optimal estimator seems to enlarge the sample size $n$ to $n\ln n$ in the error rate expression. This phenomenon was termed the \emph{effective sample size enlargement} in~\cite{Jiao--Venkat--Han--Weissman2015minimax}. 

\subsection{Approximation-based method}
We emphasize that the observed \emph{effective sample size enlargement} here is another manifestation of the recently discovered phenomenon in functional estimation of high dimensional objects. There has been a recent wave of study on functional estimation of high dimensional parameters~\cite{Valiant--Valiant2013estimating,Jiao--Venkat--Han--Weissman2015minimax,Wu--Yang2014minimax,Acharya--Orlitsky--Suresh--Tyagi2014complexity}, and it was shown in Jiao \emph{et al.}~\cite{Jiao--Venkat--Han--Weissman2015minimax} that for a wide class of functional estimation problems (including Shannon entropy $H(P)=\sum_{i=1}^S -p_i\ln p_i$, $F_\alpha \triangleq \sum_{i = 1}^S p_i^\alpha$, and mutual information), there exists a general approximation-based method that can be applied to design minimax rate-optimal estimators whose performance with $n$ samples is essentially that of the MLE (maximum likelihood estimator, or the plug-in estimator) with $n\ln n$ samples. 

The general approximation-based method in~\cite{Jiao--Venkat--Han--Weissman2015minimax} is as follows. Consider estimating $G(\theta)$ of a parameter $\theta \in \Theta \subset \mathbb{R}^p$ for an experiment $\{P_\theta: \theta \in \Theta\}$, with a consistent estimator $\hat{\theta}_n$ for $\theta$, where $n$ is the number of observations. Suppose the functional $G(\theta)$ is analytic\footnote{A function $f$ is analytic at a point $x_0$ if and only if its Taylor series about $x_0$ converges to $f$ in some neighborhood of $x_0$.} everywhere except at $\theta \in \Theta_0$. A natural estimator for $G(\theta)$ is $G(\hat{\theta}_n)$. In the estimation of functionals of discrete distributions, $\Theta$ is the $S$-dimensional probability simplex, and a natural candidate for $\hat{\theta}_n$ is the empirical distribution, which is unbiased for any $\theta\in\Theta$.

We propose to conduct the following two-step procedure in estimating $G(\theta)$.

\begin{enumerate}
	
	\item \textbf{Classify the Regime}: Compute $\hat{\theta}_n$, and declare that we are in the ``non-smooth'' regime if $\hat{\theta}_n$ is ``close'' enough to $\Theta_0$. Otherwise declare we are in the ``smooth'' regime;
	\item {\bf Estimate}:
	\begin{itemize}
		\item If $\hat{\theta}_n$ falls in the ``smooth'' regime, use an estimator ``similar'' to $G(\hat{\theta}_n)$ to estimate $G(\theta)$;
		\item If $\hat{\theta}_n$ falls in the ``non-smooth'' regime, replace the functional $G(\theta)$ in the ``non-smooth'' regime by an approximation $G_{\text{appr}}(\theta)$ (another functional) which can be estimated without bias, then apply an unbiased estimator for the functional $G_{\text{appr}}(\theta)$.
	\end{itemize}
\end{enumerate}

Approaches of this nature appeared before \cite{Jiao--Venkat--Han--Weissman2015minimax} in Lepski, Nemirovski, and Spokoiny~\cite{Lepski--Nemirovski--Spokoiny1999estimation}, Cai and Low~\cite{Cai--Low2011}, Vinck \emph{et al.}~\cite{Vinck--Battaglia--Balakirsky--Vinck--Pennartz2012estimation}, Valiant and Valiant~\cite{Valiant--Valiant2011power}. It was developed independently for entropy estimation by Wu and Yang~\cite{Wu--Yang2014minimax}, and the ideas proved to be very fruitful in Acharya \emph{et al.}~\cite{Acharya--Orlitsky--Suresh--Tyagi2014complexity}, Wu and Yang~\cite{wu2015chebyshev}, Orlitsky, Suresh, and Wu~\cite{orlitsky2016optimal}, Wu and Yang~\cite{wu2016sample}. However, we emphasize that in all the examples above except for the $L_1$ distance estimator in Valiant and Valiant~\cite{Valiant--Valiant2011power}, the functionals considered all take the form $G(\sum_{i = 1}^p f(\theta_i))$ or $G(\int f(p(x))dx )$, where $p(x)$ is a \emph{univariate} density or function, and each $\theta_i \in \mathbb{R}$. In particular, the functions $f(\cdot)$ considered are everywhere analytic except at zero, e.g., $x^\alpha, |x|^\alpha$ for $\alpha>0$ and $x\ln x$. Most of these features are violated in the $L_1$ distance estimation problem. If we write 
$
	\|P-Q\|_1 = \sum_{i  =1}^S f(p_i,q_i)
$
with $f(x,y) = |x-y| \in C([0,1]^2)$, then we have:
\begin{enumerate}
	\item a \emph{bivariate} function $f(x,y)$ in the sum;
	\item a function $f(x,y)$ which is analytic except on a \emph{segment} $x=y\in [0,1]$. 
\end{enumerate}

As discussed in Jiao \emph{et al.}~\cite{Jiao--Venkat--Han--Weissman2015minimax}, approximation of multivariate functions is much more involved than that of univariate functions, and the fact that the ``non-smooth'' regime is around a line segment here makes the application of the approximation-based method quite difficult: what shape should we use to specify the ``non-smooth'' regime? We provide a comprehensive answer to this problem in this paper, thereby substantially generalizing the applicability of the approximation-based method and demonstrate the intricacy of functional estimation problems in high dimensions. Our recent work~\cite{Han--Jiao--Weissman2016minimaxdivergence} presents the most up-to-date version of the general approximation-based method, which is applied to construct minimax rate-optimal estimators for the KL divergence (also see Bu \emph{et al.}\cite{bu2016estimation}), $\chi^2$-divergence, and the squared Hellinger distance. The \emph{effective sample size enlargement} phenomenon holds in all these cases as well. 

We emphasize that the complications triggered by the bivariate function $f(x,y) = |x-y|$ make the $L_1$ distance estimation problem highly challenging. Indeed, prior to our work, the only known estimators that require sublinear samples were in~\cite{Valiant--Valiant2011power,Valiant--Valiant2013estimating}, which achieved $L_1$ error $\sqrt{\frac{S}{n\ln n}}$ in the regime of $\frac{S}{\ln S} \lesssim n \lesssim S$ but not the regime $n \gg S$, and the lower bound was proved for the regime $n \asymp \frac{S}{\ln S}$, i.e., when the optimal error is a constant. The complete characterization of the minimax rates and the estimator that achieves the minimax rates were unknown prior to this work. 

Our main contributions in this paper are the following:
\begin{enumerate}
\item We apply the approximation-based method to construct minimax rate-optimal estimators with computational complexity $O(n \ln n)$ for $\|P-Q\|_1$ when $Q$ is known, and show that for any fixed $Q$, our estimator performs with $n$ samples at least as well as the plug-in estimator with $n\ln n$ samples. Precisely, the performance of the plug-in estimator for any fixed $Q$ is dictated by the functional $\sum_{i = 1}^S q_i \wedge \sqrt{\frac{q_i}{n}}$, while that of the minimax rate-optimal estimator is dictated by the functional $\sum_{i = 1}^S q_i \wedge \sqrt{\frac{q_i}{n \ln n}}$. Furthermore, we show that \emph{any} plug-in estimator would not achieve the same performance as our algorithm does. As we argue in Lemma~\ref{lemma.separateplugin}, for estimating $\|P-Q\|_1$ with known $Q$, for any distribution estimate $\hat{P}$ constructed from the samples from $P$, the estimator $\|\hat{P}-Q\|_1$ does not achieve the minimax rates in the worst case if $\hat{P}$ does not depend on $Q$. Concretely, the performance of any plug-in rule $\hat{P}$ behaves essentially as the MLE in the worst case.

\item We generalize the approximation-based method in~\cite{Jiao--Venkat--Han--Weissman2015minimax} to construct a minimax rate-optimal estimator for $\|P-Q\|_1$ when both $P$ and $Q$ are unknown with computational complexity $O(n \ln^2 n)$. We illustrate the novelty of our scheme via the following results: 
\begin{enumerate}
\item The performance of our estimator with $n$ samples is essentially that of the MLE with $n\ln n$ samples. 
\item Any algorithm that only conducts approximation around the origin does not achieve the minimax rates. Indeed, as we argue in Lemma~\ref{lemma.zeroapproximationinsufficient}, for any algorithm that employs the MLE when $\hat{p}\gtrsim \frac{\ln n}{n},\hat{q}\gtrsim \frac{\ln n}{n}$ cannot achieve the minimax rates when $n\gg S$. The reason why the estimator of Valiant and Valiant~\cite{Valiant--Valiant2011power} cannot achieve the minimax rates when $n\gg S$ is that~\cite{Valiant--Valiant2011power} did not conduct approximation when $p$ and $q$ are large. One of our key contributions is to figure out how to conduct approximation when $\hat{p}\gtrsim \frac{\ln n}{n}, \hat{q} \gtrsim \frac{\ln n}{n}$ and achieve the minimax rates when $n\gg S$.  
\item Best polynomial approximation is not sufficient for achieving minimax rate-optimality in this problem. As we argue in Lemma~\ref{lemma.pointwise}, any one-dimensional polynomial that achieves the best approximation error rate cannot be used in constructing the optimal estimator, and it is necessary to use a multivariate polynomial with certain pointwise error guarantees. One of our key contributions is to construct a proper \emph{multivariate} polynomial with desired pointwise approximation error. 
\item Approximation over the union of the ``nonsmooth'' regime may not work. As we show in Lemma~\ref{lemma.stripelowerbound}, there does not exist a single multivariate polynomial that achieves the desired approximation error over the whole ``nonsmooth'' regime. Instead, in our approach, we construct polynomial approximations of the function $f(p,q) = |p-q|$ over a \emph{random} regime that is determined by empirical data. To our knowledge, it is the first time that a \emph{random} approximation regime approach appears in the functional estimation literature. 
\item Our estimator is agnostic to the potentially \emph{unknown} support size $S$, but behaves as well as the minimax rate-optimal estimator that knows the support size $S$. 
\end{enumerate}
\end{enumerate}

The rest of the paper is organized as follows. In Section \ref{sec.knownq} and \ref{sec.unknownq}, we present a thorough performance analysis of the MLE and explicitly construct the minimax rate-optimal estimators, where Section \ref{sec.knownq} covers the known $Q$ case and Section \ref{sec.unknownq} generalizes to the case of unknown $Q$. Discussions in Section \ref{sec.comparisionwithother} highlight the significance and novelty of our approaches by reviewing several other approaches which are shown to be suboptimal. Section \ref{sec.exp} presents the experimental results comparing our schemes with existing approaches. The auxiliary lemmas used throughout this paper are collected in Appendix~\ref{sec.auxiliarylemmas}. Appendix~\ref{sec.proofofmaintheorems} contains proofs of the main theorems. Proofs of all the lemmas in the main text and that used in the proofs of the main theorems can be found in Appendix~\ref{sec.proofofmainlemmas}, where proofs of all the auxiliary lemmas are collected in Appendix~\ref{sec.proofofauxiliarylemmas}. 

\emph{Notation:} for non-negative sequences $a_\gamma,b_\gamma$, we use the notation $a_\gamma \lesssim_{\alpha}  b_\gamma$ to denote that there exists a constant $C$ that only depends on $\alpha$ such that $\sup_{\gamma } \frac{a_\gamma}{b_\gamma} \leq C$, and $a_\gamma \gtrsim b_\gamma$ is equivalent to $b_\gamma \lesssim a_\gamma$. When the constant $C$ is universal we do not write subscripts for $\lesssim$ and $\gtrsim$. Notation $a_\gamma \asymp b_\gamma$ is equivalent to $a_\gamma \lesssim  b_\gamma$ and $b_\gamma \lesssim  a_\gamma$. Notation $a_\gamma \gg b_\gamma$ means that $\liminf_\gamma \frac{a_\gamma}{b_\gamma} = \infty$, and $a_\gamma \ll b_\gamma$ is equivalent to $b_\gamma \gg a_\gamma$. We write $a\wedge b=\min\{a,b\}$ and $a\vee b=\max\{a,b\}$. Moreover, $\poly_n^d$ denotes the set of all $d$-variate polynomials of degree of each variable no more than $n$, and $E_n[f;I]$ denotes the distance of the function $f$ to the space $\poly_n^d$ in the uniform norm $\|\cdot\|_{\infty,I}$ on $I\subset \bR^d$. The space $\poly_n^1$ is also abbreviated as $\poly_n$. All logarithms are in the natural base. The notation $x\geq \mathcal{Y}$, where $x$ is a real number and $\mathcal{Y}$ is a set of real numbers, is equivalent to $x\geq y$ for all $y\in \mathcal{Y}$. 

Throughout this paper, we utilize the Poisson sampling model instead of the binomial model, whose minimax risks can be shown to be closely related, as in~\cite[Lemma 16]{Jiao--Venkat--Han--Weissman2015minimax}. 

\section{Divergence Estimation with Known $Q$} \label{sec.knownq}
First we consider the case where $Q=(q_1,\cdots,q_S)$ is known while $P$ is an unknown distribution with support $\mathcal{S} = \{1,2,\cdots,S\}$. In other words, $\mathcal{P}=\cM_S$ and $\mathcal{Q}=\{Q\}$. We analyze the performance of the MLE in this case, and construct the approximation-based minimax rate-optimal estimator.

We utilize the Poisson sampling model, in which we observe a Poisson random vector 
\begin{align}
\mathbf{X} = [X_1,X_2,\ldots,X_S],
\end{align}
where the coordinates of $\mathbf{X}$ are mutually independent, and $X_i \sim \spo(n p_i)$. We define $\hat{p}_i = \frac{X_i}{n}$ as the empirical probabilities. 

\subsection{Performance of the MLE}
The MLE serves as a natural estimator for the $L_1$ distance which can be expressed as
$
\|P_n-Q\|_1 = \sum_{i=1}^S |\hat{p}_i - q_i|
$, where $P_n = \mathbf{X}/n =(\hat{p}_1, \hat{p}_2, \cdots,\hat{p}_S)$ is the empirical distribution. Since we are using the Poisson sampling mode, we have $n\hat{p}_i \sim \mathsf{Poi}(np_i)$.


We obtain the upper and lower bounds for the mean squared error of $\|P_n-Q\|_1$ in the following theorem. 
\begin{theorem}\label{Thm.mle}
The maximum likelihood estimator $\|P_n-Q\|_1$ satisfies 
\begin{align}
& \sup_{P\in \mathcal{M}_S}\bE_P|\|P_n-Q\|_1 - \|P-Q\|_1|^2 \nonumber \\
& \quad \leq 4 \left( \sum_{i = 1}^S q_i \wedge \sqrt{\frac{q_i}{n}} \right)^2 + \frac{1}{n}. 
\end{align}
We can also lower bound the worst case mean squared error as
\begin{align}
\sup_{P\in \mathcal{M}_S}\bE_P|\|P_n-Q\|_1 - \|P-Q\|_1|^2 \ge \frac{1}{2} \left(  \sum_{i = 1}^S q_i \wedge \sqrt{\frac{q_i}{n}} \right)^2. 
\end{align}
\end{theorem}

The following corollary is straightforward since $\sup_{Q \in \mathcal{M}_S}  \sum_{i = 1}^S q_i \wedge \sqrt{\frac{q_i}{n}} \asymp \sqrt{\frac{S}{n}}$ when $n \gtrsim S$. 
\begin{corollary}\label{Cor.mle}
If $n\gtrsim S$, we have
\begin{align}
\sup_{P,Q\in \mathcal{M}_S}\bE_P|\|P_n-Q\|_1 - \|P-Q\|_1|^2 \asymp \frac{S}{n}.
\end{align}
\end{corollary}

Hence, it is necessary and sufficient for the MLE to have $n\gg S$ samples to be consistent in terms of the worst case mean squared error.

\subsection{Construction of the optimal estimator}

We apply our general recipe to construct the minimax rate-optimal estimator. For simplicity of analysis, we conduct the classical ``splitting'' operation~\cite{Tsybakov2013aggregation} on the Poisson random vector $\mathbf{X}$, and obtain two independent identically distributed random vectors
 $\mathbf{X}_j = [X_{1,j},X_{2,j},\ldots,X_{S,j}]^T, j\in\{1,2\}$, such that each component $X_{i,j}$ in $\mathbf{X}_j$ has distribution $\spo(np_i/2)$, and all coordinates in $\mathbf{X}_j$ are independent. For each coordinate $i$, the splitting process generates a random sequence $\{T_{ik}\}_{k=1}^{X_i}$, $T_{ik}\in \{1,2\}$, such that $\{T_{ik}\}_{k=1}^{X_i}|\mathbf{X} \sim \mathsf{multinomial}(X_i; (1/2,1/2))$, and assign $X_{i,j}=\sum_{k=1}^{X_i} \mathbbm{1}(T_{ik}=j)$ for $j\in\{1,2\}$. All the random variables $\{ \{T_{ik}\}_{k=1}^{X_i}:1\leq i\leq S\}$ are conditionally independent given our observation $\mathbf{X}$. The ``splitted'' empirical probabilities are defined as $\hat{p}_{i,j} = X_{i,j} / (n/2)$. To simplify notation, we redefine $n/2$ as $n$ to ensure that $n\hat{p}_{i,j} \sim \mathsf{Poi}(np_i), j = 1,2$. We emphasize that the sampling splitting approach is not conducted in the implementation of the estimator. 

We construct two set functions with variable $q$ as input defined as:
\begin{align}
U(q;c_1) & = \begin{cases}
[0,\frac{2c_1\ln n}{n}], & q\le \frac{c_1\ln n}{n}\\
[q-\sqrt{\frac{c_1q\ln n}{n}}, q+ \sqrt{\frac{c_1q\ln n}{n}}], & \frac{c_1\ln n}{n}<q\leq 1. 
\end{cases} \label{eq.uncertain_set} \\
U_1(q) & = \begin{cases}
[0,\frac{(c_1 + c_3)\ln n}{n}], & q\le \frac{c_1\ln n}{n}\\
[q-\sqrt{\frac{c_3q\ln n}{n}}, q+ \sqrt{\frac{c_3q\ln n}{n}}], & \frac{c_1\ln n}{n}<q\leq 1. 
\end{cases} \label{eq.uncertain_set2}
\end{align}
Here $c_1>0, c_1>c_3>0$ are constants that will be determined later. The set $U(q;c_1)$ is constructed to satisfy the following property:
\begin{lemma} \label{lemma.individualqcover}
Suppose $n\hat{q} \sim \spo(nq)$. Then, 
\begin{align}
\bP(\hat{q} \notin U(q;c_1)) \leq \frac{2}{n^{c_1/3}},
\end{align}
where the set function $U(q;c_1)$ is defined in~(\ref{eq.uncertain_set}). 
\end{lemma}

It is clear that for any $q\in [0,1], U_1(q) \subset U(q;c_1)$. The constants $c_1>0, c_1>c_3>0$ will be chosen later to make sure that the following three ``good'' events have overwhelming probability:
\begin{align}
E_1 & = \bigcap_{i = 1}^S \left \{ \hat{p}_{i,1} > U_1(q_i) \Rightarrow p_i \geq q_i \right \} \label{eqn.error1} \\
E_2 & = \bigcap_{i = 1}^S \left \{ \hat{p}_{i,1} < U_1(q_i) \Rightarrow p_i \leq q_i \right \} \label{eqn.error2} \\
E_3 & = \bigcap_{i = 1}^S \left \{ \hat{p}_{i,1} \in U_1(q_i) \Rightarrow p_i \in U(q_i;c_1) \right \}. \label{eqn.error3}
\end{align}

Here $A \Rightarrow B$ represents the logical implication operation that is equivalent to $A^c \cup B$. The intuitions behind the constructions of these ``good'' events are as follows. Since we use the first half of the samples $\hat{p}_{i,1}$ to classify regime, and would later use three different estimators depending on whether $\hat{p}_{i,1}$ lies to the left, to the right, or inside $U_1(q_i)$, it is desirable that we can infer the relationship between $p_i$ and $q_i$ based on the location of $\hat{p}_{i,1}$. The reason why these events can be controlled to have high probabilities is that we have specifically designed $U_1(q)$ to make it a strict subset of the set $U(q;c_1)$, and the sets $U(q;c_1)$ are designed to satisfy Lemma~\ref{lemma.individualqcover}, which ensures that the size of $U(q;c_1)$ is essentially the length of the confidence interval when the empirical probability $\hat{q}$ is observed.

We have the following lemma controlling the probability of these probabilities. 
\begin{lemma}\label{lemma.goodeventswin}
Denote the overall ``good'' event $E = E_1 \cap E_2 \cap E_3$, where $E_1,E_2,E_3$ are defined in~(\ref{eqn.error1}),(\ref{eqn.error2}),(\ref{eqn.error3}). Then, 
\begin{align}
\bP(E^c) \leq \frac{3S}{n^\beta},
\end{align}
where
\begin{align}\label{eqn.betadefinition}
\beta = \min\left \{\frac{c_3^2}{3c_1}, \frac{(c_1-c_3)^2}{4c_1}, \frac{(\sqrt{c_1}-\sqrt{c_3})^2}{3} \right \}.
\end{align}
\end{lemma}

%
%


Now we construct the estimator. In the ``smooth" regime, i.e., $\hat{p}\notin U_1(q)$, we simply employ the plug-in estimator to estimate $f(p,q)$. In the ``non-smooth" regime, i.e., $\hat{p} \in U_1(q)$, we need to approximate $f(p,q)$ by another functional which can be estimated without bias. We consider the best polynomial approximation of $f(x,q)$ on $U(q;c_1) \supset U_1(q)$, which is defined as
\begin{align}\label{eqn.knownqpoly}
P_K(x;q) = \argmin_{P\in \poly_K} \max_{z\in U(q;c_1)} |f(z,q)-P(z)|
\end{align}
where $\poly_K$ denotes the set of polynomials with degree no more than $K$. Once we obtain $P_K(x;q)$, we can use an unbiased estimate $\tilde{P}_K(\hat{p};q)$ such that $\bE\tilde{P}_K(\hat{p};q) = P_K(p;q)$ for $n\hat{p}\sim \mathsf{Poi}(np)$. As a result, the absolute value of the bias of the estimator $\tilde{P}_K(\hat{p};q)$ in the ``non-smooth" regime is exactly the approximation error of $P_K(x;q)$ in approximating $f(x,q)=|x-q|$ on $U(q;c_1)$, which can be significantly smaller than that of the MLE.


\begin{construction}\label{construction.estimator1}
We use the first half samples to classify regimes and the second half samples for estimation. Denote
\begin{align}
\tilde{L}_1 & = \sum_{i=1}^S [(\hat{p}_{i,2}-q_i) \mathbbm{1}(\hat{p}_{i,1}>U_1(q_i)) \nonumber \\
& \qquad \quad + (q_i - \hat{p}_{i,2}) \mathbbm{1}(\hat{p}_{i,1}<U_1(q_i))  \nonumber \\
& \qquad  \quad + 
\tilde{P}_K(\hat{p}_{i,2};q_i)\mathbbm{1}(\hat{p}_{i,1}\in U_1(q_i))]
\end{align}
and define
\begin{align}
\hat{L}^{(1)} =  0 \vee \left(\tilde{L}_1 \wedge 2 \right), 
\end{align}
where $U(q_i;c_1)$ and $U_1(q_i)$ are given by (\ref{eq.uncertain_set}), (\ref{eq.uncertain_set2}), $K=c_2\ln n$, and $c_1,c_2>0,c_3>0$ are properly chosen constants. 
\end{construction}

The performance of this estimator is presented in the following theorem. 
\begin{theorem}\label{Thm.opt1}
Suppose there exist two constants $c,C$ such that $c \ln S \leq \ln n \leq C \ln \left(  \sum_{i = 1}^S \sqrt{q_i} \wedge q_i \sqrt{n\ln n}  \right)$. Then, there exists constants $c_1,c_2,c_3$ depending only on $c,C$ in Construction~\ref{construction.estimator1} such that 
\begin{align}
\sup_{P\in\cM_S} \bE_P|\hat{L}^{(1)} - \|P-Q\|_1|^2 \lesssim_{c,C} \left( \sum_{i = 1}^S q_i \wedge \sqrt{\frac{q_i}{n\ln n}} \right)^2. 
\end{align}
In particular, if $\ln n \leq C \ln S$, we have
\begin{align}
\sup_{P,Q\in\cM_S} \bE_P|\hat{L}^{(1)} - \|P-Q\|_1|^2 \lesssim_{C} \frac{S}{n\ln n}.
\end{align}
\end{theorem}

\begin{remark}
When we consider the worst case of $Q$, Theorem~\ref{Thm.opt1} assumes that the sample size cannot be too big ($\ln n \leq C \ln S$). It is obvious that an upper bound on the sample size is needed for the statement to be valid: indeed, if no upper bound on the sample size is imposed then in the asymptotic regime ($S$ fixed, $n\to \infty$) the convergence rate is faster than the parametric rate $\frac{1}{n}$, which is impossible. However, we are not sure that the current upper bound is tight. The reason why we introduced this upper bound is that it is needed to control the variance of our estimator, but the variance bound we have may not be tight. 

Compared to existing literature, the schemes by Valiant and Valiant~\cite{Valiant--Valiant2011power,Valiant--Valiant2013estimating} achieved mean squared error $\frac{S}{n\ln n}$ only in the regime of $\frac{S}{\ln S} \lesssim n \lesssim S$ but not the regime $n \gg S$. The main reason is that~\cite{Valiant--Valiant2011power,Valiant--Valiant2013estimating} did not conduct approximation when $p\geq \frac{\ln n}{n}$. As our work shows, the key reason behind whether one should conduct approximation or not is not whether the probability $p$ is close to zero or not, but whether the functional has a non-analytic point or not. As we show in Lemma~\ref{lemma.zeroapproximationinsufficient} in Section~\ref{sec.comparisionwithother}, any approach that only conducts approximation when $p$ is small cannot achieve the minimax rates for $n\gg S$ in general. 
\end{remark}

\subsection{Minimax lower bound}
It was shown in Valiant and Valiant~\cite{Valiant--Valiant2011power} that if $Q$ is the uniform distribution, when $n \asymp \frac{S}{\ln S}$, the minimax risk of estimating $\|P-Q\|_1$ is a constant. We prove a minimax lower bound for \emph{every} $Q$, and show that the performance achieved by our estimator in Theorem~\ref{Thm.opt1} is minimax rate-optimal for every fixed $Q$.

\begin{theorem}\label{Thm.lowerbound}
Suppose there exists a constant $C>0$ such that $\ln n \geq C \ln S, S\geq 2$. Then, there exists a constant $C'>0$ that only depends on $C$ such that if $ \sum_{j = 1}^S q_j \wedge \sqrt{\frac{q_j}{n\ln n}} \geq C'\left( \sqrt{\frac{\ln n}{n}} + \frac{\sqrt{S}\ln n}{n} \right)$, then
\begin{align}
\inf_{\hat{L}}\sup_{P\in\cM_S} \mathbb{E}_P |\hat{L}-\|P-Q\|_1|^2 \gtrsim_C \left( \sum_{i = 1}^S q_i \wedge \sqrt{\frac{q_i}{n\ln n}} \right)^2,
\end{align}
where the infimum is taken over all possible estimators. 

In particular, if there exist constant $c>0,C>0$ such that $n \geq c \frac{S}{\ln S}, \ln n \leq C \ln S$, then
\begin{align}
\sup_{Q \in \mathcal{M}_S} \inf_{\hat{L}} \sup_{P\in \mathcal{M}_S} \bE_P|\hat{L} - \|P-Q\|_1|^2 \gtrsim_{c,C} \frac{S}{n\ln n}.
\end{align}
\end{theorem}

Combining Theorem~\ref{Thm.opt1} and Theorem \ref{Thm.lowerbound}, we have the following theorem.
\begin{theorem}\label{thm.matchingknownqrate}
Suppose there exist constants $c>0,C>0$ such that $c \ln S \leq  \ln n \leq C \ln\left( \sum_{i = 1}^S \sqrt{q_i} \wedge q_i \sqrt{n\ln n} \right), S\geq 2$. Then, 
\begin{align}
\inf_{\hat{L}} \sup_{P\in \mathcal{M}_S} \mathbb{E}_P |\hat{L} - \|P-Q\|_1 |^2 & \asymp_{c,C} \left( \sum_{i = 1}^S q_i \wedge \sqrt{\frac{q_i}{n\ln n}} \right)^2. 
\end{align}
In particular, if $n\geq c\frac{S}{\ln S}, \ln n \leq C \ln S$, then
\begin{align}
\sup_{Q\in \mathcal{M}_S} \inf_{\hat{L}} \sup_{P\in \mathcal{M}_S} \mathbb{E}_P| \hat{L} - \|P-Q\|_1 |^2 & \asymp_{c,C} \frac{S}{n\ln n}. 
\end{align}
The estimator in Construction~\ref{construction.estimator1} achieves the minimax rates for every fixed $Q$. 
\end{theorem}

\section{Divergence Estimation with Unknown $Q$} \label{sec.unknownq}
Now we consider the general case where both $P$ and $Q$ are unknown to us, i.e., $\mathcal{P}=\mathcal{Q}=\cM_S$. 

We utilize the Poisson sampling model, in which we observe two Poisson random vectors
\begin{align}
\mathbf{X} & = [X_1,X_2,\ldots,X_S],\\
\mathbf{Y} & = [Y_1,Y_2,\ldots,Y_S],
\end{align}
where all the coordinates of $\mathbf{X}$ and $\mathbf{Y}$ are mutually independent, and $X_i \sim \spo(n p_i), Y_i \sim \spo(nq_i)$. We introduce the empirical probabilities $\hat{p}_i = \frac{X_i}{n}, \hat{q}_i = \frac{Y_i}{n}$. 

\subsection{Performance of the MLE}
In this case, the MLE is expressed as $\|P_n-Q_n\|_1=\sum_{i=1}^{S}|\hat{p}_i-\hat{q}_i|$. Since $|\|P_n-Q_n\|_1-\|P-Q\|_1|\le \|P_n-P\|_1+\|Q_n-Q\|_1$ by the triangle inequality, and $\mathbb{E}|\hat{p}_i - \hat{q}_i| \geq \mathbb{E}|\hat{p}_i - q_i|$ by the conditional Jensen's inequality, Theorem~\ref{Thm.mle} can again be applied here to give the performance of the MLE.
\begin{theorem}
If $n\gtrsim S$, the MLE satisfies
\begin{align}
\sup_{P,Q\in\cM_S}\mathbb{E}|\|P_n-Q_n\|_1-\|P-Q\|_1|^2 \asymp \frac{S}{n}.
\end{align}
\end{theorem}

Hence, the MLE achieves the mean squared error $S/n$, and requires $n\gg S$ samples to be consistent.
\subsection{Construction of the optimal estimator}

Again we apply our general recipe to construct the optimal estimator, but encounter several new difficulties: $f(x,y)=|x-y|$ is non-analytic on a segment, and both the uncertainty set and the polynomial approximation need to be generalized to the 2D case. We will overcome these obstacles step by step.

For simplicity of analysis, we conduct the classical ``splitting'' operation~\cite{Tsybakov2013aggregation} on the Poisson random vector $\mathbf{X}$, and obtain two independent identically distributed random vectors
 $\mathbf{X}_j = [X_{1,j},X_{2,j},\ldots,X_{S,j}]^T, j\in\{1,2\}$, such that each component $X_{i,j}$ in $\mathbf{X}_j$ has distribution $\spo(np_i/2)$, and all coordinates in $\mathbf{X}_j$ are independent. For each coordinate $i$, the splitting process generates a random sequence $\{T_{ik}\}_{k=1}^{X_i}$ such that $\{T_{ik}\}_{k=1}^{X_i}|\mathbf{X} \sim \mathsf{multinomial}(X_i; (1/2,1/2))$, and assign $X_{i,j}=\sum_{k=1}^{X_i} \mathbbm{1}(T_{ik}=j)$ for $j\in\{1,2\}$. All the random variables $\{ \{T_{ik}\}_{k=1}^{X_i}:1\leq i\leq S\}$ are conditionally independent given our observation $\mathbf{X}$. The splitting operation is similarly conducted for the Poisson random vector $\mathbf{Y}$ independently. The ``splitted'' empirical probabilities are defined as $\hat{p}_{i,j} = X_{i,j} / (n/2), \hat{q}_{i,j} = Y_{i,j} / (n/2)$. To simplify notation, we redefine $n/2$ as $n$ to ensure that $n\hat{p}_{i,j} \sim \mathsf{Poi}(np_i), j = 1,2$. We emphasize that the sampling splitting approach is not needed for the actual estimator construction.

As usual, first we classify ``smooth" and ``non-smooth" regimes. Since the function $f(x,y)=|x-y|\in C([0,1]^2)$ is non-analytic on the segment $x=y\in [0,1]$, we are looking for the ``uncertainty set" $U$ containing this segment such that any $(p,q)\in U$ can be ``localized" in the previous sense. We have the following lemma.
\begin{lemma}\label{lemma.stripeshape}
The two-dimensional set $U \subset [0,1]^2$ defined as 
\begin{align}
U & = \Bigg \{(p,q): |p-q| \leq \sqrt{\frac{2c_1 \ln n}{n}} (\sqrt{p}+\sqrt{q}), \nonumber \\
& \qquad  \qquad p\in [0,1], q\in [0,1] \Bigg \}\label{eq.stripe}
\end{align}
satisfies 
\begin{align}
U \supset \cup_{x\in [0,1]} U(x;c_1)\times U(x;c_1),
\end{align}
where $U(x;c_1)$ is given by (\ref{eq.uncertain_set}). 
\end{lemma}

We design another set $U_1$ as follows:
\begin{align}\label{eqn.2du1define}
U_1 = \{(p,q): |p-q| \leq \sqrt{\frac{(c_1+c_3) \ln n}{n}} (\sqrt{p} + \sqrt{q})\},
\end{align}
where $0<c_3<c_1$. Clearly $U_1 \subset U$. We choose the constants $c_1$ and $c_3$ later to ensure that the following four events happen with high probability: 

\begin{align}
E_1 & = \bigcap_{i = 1}^S \Bigg \{ \hat{p}_{i,1} - \hat{q}_{i,1} > \sqrt{\frac{(c_1 + c_3)\ln n}{n}}(\sqrt{\hat{p}_{i,1}} + \sqrt{\hat{q}_{i,1}}) \nonumber \\
& \qquad \qquad \Rightarrow p_i \geq q_i \Bigg \} \label{eqn.2derror1} \\
E_2 & = \bigcap_{i = 1}^S \Bigg \{ \hat{p}_{i,1} - \hat{q}_{i,1} < - \sqrt{\frac{(c_1 + c_3)\ln n}{n}}(\sqrt{\hat{p}_{i,1}} + \sqrt{\hat{q}_{i,1}}) \nonumber \\
& \qquad \qquad \Rightarrow p_i \leq q_i \Bigg \} \label{eqn.2derror2} \\
E_3 & = \bigcap_{i = 1}^S \left \{ \hat{p}_{i,1}+\hat{q}_{i,1} < \frac{c_1 \ln n}{n} \Rightarrow (p_i, q_i)\in  \left[ 0, \frac{2c_1 \ln n}{n} \right]^2\right \} \label{eqn.2derror3} \\
E_4 & = \bigcap_{i = 1}^S \Bigg \{ (\hat{p}_{i,1},\hat{q}_{i,1} ) \in U_1, \hat{p}_{i,1} + \hat{q}_{i,1} \geq \frac{c_1 \ln n}{n} \nonumber \\
& \qquad \qquad \Rightarrow (p_i,q_i) \in U, p_i+q_i\geq \frac{c_1 \ln n}{2n}, \nonumber \\
& \qquad \qquad \qquad \hat{p}_{i,1} +\hat{q}_{i,1} \geq \frac{p_i + q_i}{2}\Bigg \}. \label{eqn.2derror4}  
\end{align}


We have the following lemma controlling the probability of these events happening simultaneously. 
\begin{lemma}\label{lemma.2dgoodeventswin}
Denote the overall ``good'' event $E = E_1 \cap E_2 \cap E_3 \cap E_4$, where $E_1,E_2,E_3, E_4$ are defined in~(\ref{eqn.2derror1}),(\ref{eqn.2derror2}),(\ref{eqn.2derror3}), (\ref{eqn.2derror4}). Then, assuming $\frac{c_3}{c_1} < \frac{8}{(\sqrt{2}+1)^2} -1 \approx 0.373$, 
\begin{align}
\bP(E^c) \leq \frac{15S}{n^\beta},
\end{align}
where the constant $\beta$ is given by
\begin{align}\label{eqn.2dbetadefinition}
\beta = \min\left \{ \frac{c_1}{6}, \frac{(c_1-c_3)^2}{96 c_1}, \frac{1}{3} \left( \sqrt{2c_1} - \frac{\sqrt{2}+1}{2}\sqrt{c_1 + c_3} \right)^2 \right \}.
\end{align}
\end{lemma}
It is evident that we can make $\beta$ in~(\ref{eqn.2dbetadefinition}) arbitrarily large by taking $c_1$ large and keeping $c_3/c_1$ a small constant. Clearly, if the true parameters $(p,q) \notin U$, the MLE would be a decent estimator. It suffices to construct estimators when the true parameters $(p,q) \in U$. The known $Q$ case seems to suggest that we consider the best polynomial approximation of $f(x,y)=|x-y|$ on $U$. However, this will not work for two reasons:
\begin{enumerate}
\item the entire 2D stripe $U$ is too large for the polynomial approximation error to vanish at the correct rate;
\item best polynomial approximation in the 2D case is not unique, and may not achieve the desired pointwise error.
\end{enumerate}

We will explore these reasons in details in Section IV. To solve the first problem, we remark that although $U$ is the set such that its element can be localized within $U$, a specific element $(x,y)\in U$ can be localized in a much smaller subset
$U(x;c_1) \times U(y;c_1) \subset U$, where $U(x;c_1)$ is given by (\ref{eq.uncertain_set}). Hence, the approximation regime should be \emph{dependent} on the empirical observations to fully utilize the available information. 

For the second problem, we need to design a specific polynomial with satisfactory pointwise approximation properties. Our approximation recipe is the following. Take $K = c_2\ln n$. 

\begin{enumerate}
\item Over the square $\left[0, \frac{2c_1 \ln n}{n} \right]^2$: we consider the decomposition $|x-y|=(\sqrt{x}+\sqrt{y})|\sqrt{x}-\sqrt{y}|$ and introduce the following two bivariate polynomials $u_K(x,y)$ and $v_K(x,y)$ to uniformly approximate $\sqrt{x} + \sqrt{y}$ and $|\sqrt{x} - \sqrt{y}|$, respectively. Specifically, we have
\begin{align}\label{eqn.defuv}
& \sup_{(x,y)\in [0,1]^2} |u_K(x,y) - (\sqrt{x} + \sqrt{y})| \nonumber \\
& \qquad = \inf_{P \in \poly_K^2} \sup_{(x,y) \in [0,1]^2}|P(x,y) - (\sqrt{x} + \sqrt{y})| \\
& \sup_{(x,y)\in [0,1]^2} |v_K(x,y) - |\sqrt{x} - \sqrt{y}|| \nonumber \\
& \qquad = \inf_{P \in \poly_K^2} \sup_{(x,y) \in [0,1]^2}|P(x,y) - |\sqrt{x} - \sqrt{y}||. 
\end{align}
Then, denote $h_{2K}(x,y) = u_K(x,y)v_K(x,y) - u_K(0,0)v_K(0,0)$, we use the polynomial
\begin{align}\label{eqn.polynomial2dsquare}
P_K^{(1)}(x,y)= \frac{2c_1 \ln n}{n} h_{2K} \left( \frac{x n}{2c_1 \ln n}, \frac{y n}{2c_1 \ln n} \right)
\end{align}
to approximate $|x-y|$ over the square $\left[ 0, \frac{2c_1 \ln n}{n} \right]^2$. The polynomial $P_K^{(1)}(x,y)$ satisfies $P_K^{(1)}(0,0) = 0$. We remove the constant term in the definition of $P_K^{(1)}$ to guarantee that the estimator we construct is agnostic to the unknown support size $S$. In practice, $u_K$ and $v_K$ can be replaced by the efficiently implementable lowpass filtered Chebyshev expansion~\cite{Mhaskar--Nevai--Shvarts2013applications}, which achieves the same uniform error rate as the best polynomial approximation. 

\begin{remark}\label{remark.2dapprox}
We would like to discuss the intuitions behind our construction of the polynomials $u_K, v_K$. One observation is that \emph{best approximation}, which aims at approximating the bivariate function $|p-q|$ over the square $\left[0, \frac{2c_1 \ln n}{n}\right]^2$ under the supremum norm, may not work. Indeed, consider the segment $p+q = \frac{2c_1 \ln n}{n}$ over $\left[0, \frac{2c_1 \ln n}{n}\right]^2$, and the function $|p-q|$ over this segment can be viewed as a univariate function, whose best approximation error using degree $K = c_2 \ln n$ is lower bounded by $\frac{1}{K} \frac{2c_1 \ln n}{n}$ within a constant factor~\cite[Chap. 9, Thm. 3.3]{Devore--Lorentz1993}, which is of order $\frac{1}{n}$. Hence, the accumulated bias is at least $\frac{S}{n}$, which results in a worse critical scaling $n\gg S$ rather than the $n\gg \frac{S}{\ln S}$ critical scaling we aim for. The key idea that enabled us to achieve worst case accumulated bias $\sqrt{\frac{S}{n\ln n}}$ is the $P$ and $Q$ are probability measures satisfying $\sum_i p_i = \sum_i q_i =1$. Hence, it suffices to prove a pointwise bound for each individual symbol $\sqrt{\frac{p_i +q_i}{n\ln n}} + \frac{1}{n\ln n}$. However, to our knowledge, the study of pointwise bounds for multivariate approximation theory has been limited. The decomposition $|x-y| = |\sqrt{x}-\sqrt{y}|(\sqrt{x}+\sqrt{y})$ is translating the problem of obtaining pointwise bounds to the problem of obtaining uniform bounds. Indeed, the uniform error of approximating $|\sqrt{x}-\sqrt{y}|$ and $\sqrt{x}+\sqrt{y}$ over $\left[ 0, \frac{2c_1 \ln n}{n} \right]^2$ with degree $K = c_2 \ln n$ are both of order $\frac{1}{\sqrt{n\ln n}}$ (Lemma~\ref{lemma.firstorderdtmoduindependentofa}), and the finite-difference formula $\Delta(ab) = a\Delta b + b \Delta a + (\Delta a) (\Delta b)$ precisely gives us the desired pointwise bound. 
\end{remark}

\item Once we can assert with high probability $(p,q) \in U, p+q \geq \frac{c_1 \ln n}{2n}$, we utilize the best approximation polynomial of $|t|$ on $[-1,1]$ with order $K$. Denote it as 
\begin{align}
R_K(t) & = \argmin_{P \in \poly_K} \sup_{t\in [-1,1]}| P(t) - |t|| \\
& = \sum_{j = 0}^K r_j t^j,
\end{align}
we have
\begin{align}\label{eqn.polynomial2dstripe}
P_K^{(2)}(x,y; \hat{p}_{i,1}, \hat{q}_{i,1}) & = \sum_{j = 0}^K r_j W^{-j+1} (x-y)^j,
\end{align}
where $W = \sqrt{\frac{8c_1 \ln n}{n}} (\sqrt{(\hat{p}_{i,1}+\hat{q}_{i,1}) \vee \frac{1}{n}})$. It is the best approximation polynomial of $|t|$ over interval $[-W,W]$. 

\begin{remark}
We discuss the reason why we cannot apply the best approximation polynomial of $|t|$ over the square $\left[ 0, \frac{2c_1 \ln n}{n}\right]^2$. Note that the approximation width $W$ is at least of order $\frac{\ln n}{n}$ since $\hat{p}_{i,1} + \hat{q}_{i,1}\gtrsim \frac{\ln n}{n}$. However, for the square $\left[ 0, \frac{2c_1 \ln n}{n}\right]^2$, we easily have $W \ll \frac{\ln n}{n}$, but $\frac{\ln n}{n}$ is the minimum width which ensures concentration properties (Lemma~\ref{lemma.stripeshape}). Indeed, as we show in Lemma~\ref{lemma.pointwise}, any 1D approximation polynomial fails to achieve the pointwise error bound we discussed in Remark~\ref{remark.2dapprox} over $\left[ 0, \frac{2c_1 \ln n}{n}\right]^2$. 
\end{remark}
\end{enumerate}

Finally, we use the second part of the samples to construct the unbiased estimators for $P_K^{(1)}(x,y)$ defined in~(\ref{eqn.polynomial2dsquare}) and $P_K^{(2)}(x,y; \hat{p}_{i,1}, \hat{q}_{i,1})$ defined in~(\ref{eqn.polynomial2dstripe}). Concretely, we introduce the estimators $\tilde{P}_K^{(1)}(\hat{p}_{i,2},\hat{q}_{i,2})$ and $\tilde{P}_K^{(2)}(\hat{p}_{i,2},\hat{q}_{i,2}; \hat{p}_{i,1}, \hat{q}_{i,1})$ such that
\begin{align}
\mathbb{E}\left[ \tilde{P}_K^{(1)}(\hat{p}_{i,2},\hat{q}_{i,2})  \right ] & = P_K^{(1)}(p,q) \label{eqn.2dpk1} \\
\mathbb{E} \left[ \tilde{P}_K^{(2)}(\hat{p}_{i,2},\hat{q}_{i,2}; \hat{p}_{i,1}, \hat{q}_{i,1}) \bigg | \hat{p}_{i,1}, \hat{q}_{i,1} \right ] & = P_K^{(2)}(p,q; \hat{p}_{i,1}, \hat{q}_{i,1}). \label{eqn.2dpk2}
\end{align}
These unbiased estimators are easy to construct since for any $r,s\geq 1, r,s\in \mathbb{Z}, (n\hat{p},n\hat{q}) \sim \spo(np)\times \spo(nq)$, we have~\cite[Ex. 2.8]{Withers1987}
\begin{align}
\mathbb{E} \left[ \prod_{i = 0}^{r-1} \left( \hat{p} - \frac{i}{n} \right) \prod_{j = 0}^{s-1} \left( \hat{q} - \frac{j}{n} \right) \right] = p^r q^s.
\end{align}

The final estimator is presented as follows. 
\begin{construction}\label{construction.2d}
As before, use sample splitting to obtain $(\hat{p}_{i,1},\hat{q}_{i,1})$ and $(\hat{p}_{i,2},\hat{q}_{i,2})$. Denote
\begin{align}
& \tilde{L}_2 \nonumber \\
& = \left( \hat{p}_{i,2} - \hat{q}_{i,2} \right) \mathbbm{1}\left( \hat{p}_{i,1} - \hat{q}_{i,1} > \sqrt{\frac{(c_1 + c_3)\ln n}{n}}(\sqrt{\hat{p}_{i,1}} + \sqrt{\hat{q}_{i,1}}) \right) \nonumber \\
& + \left( \hat{q}_{i,2} - \hat{p}_{i,2} \right) \mathbbm{1} \left( \hat{p}_{i,1} - \hat{q}_{i,1} < - \sqrt{\frac{(c_1 + c_3)\ln n}{n}}(\sqrt{\hat{p}_{i,1}} + \sqrt{\hat{q}_{i,1}}) \right) \nonumber \\
& + \tilde{P}_K^{(1)}(\hat{p}_{i,2},\hat{q}_{i,2}) \mathbbm{1}\left( \hat{p}_{i,1}+\hat{q}_{i,1} < \frac{c_1 \ln n}{n} \right) \nonumber \\
& + \tilde{P}_K^{(2)}(\hat{p}_{i,2},\hat{p}_{i,2}; \hat{p}_{i,1},\hat{q}_{i,1})  \mathbbm{1}\left( (\hat{p}_{i,1},\hat{q}_{i,1} ) \in U_1, \hat{p}_{i,1} + \hat{q}_{i,1} \geq \frac{c_1 \ln n}{n} \right).
\end{align}
and define
\begin{align}
\hat{L}^{(2)} =  0 \vee \left(\tilde{L}_2 \wedge 2 \right). 
\end{align}
Here $U$ is given by (\ref{eq.stripe}), $U_1$ is defined in~(\ref{eqn.2du1define}), the estimators $\tilde{P}_K^{(1)}$ and $\tilde{P}_K^{(2)}$ are defined in~(\ref{eqn.2dpk1}) and~(\ref{eqn.2dpk2}) $K=c_2\ln n$, and $c_1>c_3> c_2>0$ are properly chosen constants, $\frac{c_3}{c_1} < \frac{8}{(\sqrt{2}+1)^2} -1 \approx 0.373$. 
\end{construction}


A pictorial explanation of the estimator construction is given in Fig~\ref{fig.2destimator}. Concretely, we use the first sample to classify into four regimes, and in each regime we do the following operations:
\begin{enumerate}
\item Regime I: 	plug-in: $\hat{p}_{2}-\hat{q}_{2}$ 
\item Regime II: 	plug-in: $\hat{q}_{2}-\hat{p}_{2}$
\item Regime III: 	2D polynomial approximation of $|p-q|$
\item Regime IV: 1D polynomial approximation of $|t|$ where $t = p-q$ with width  $\sqrt{\frac{8c_1 \ln n}{n}} \sqrt{ \left( \hat{p}_1 +\hat{q}_1 \right) \vee \frac{1}{n} }$
	\end{enumerate}

	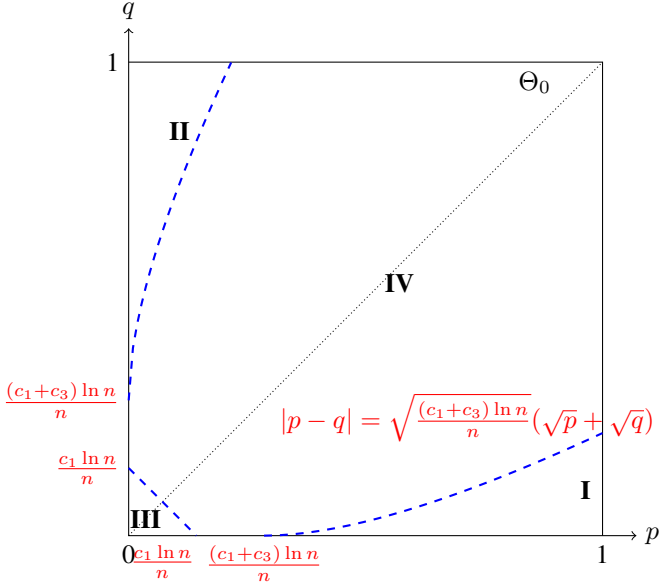
\begin{figure}[!htbp]
	\centering
	\begin{tikzpicture}[xscale=4.5, yscale=4.5]
	\draw [<->] (0,1.5) -- (0,0) -- (1.5,0);
	\draw (1.4,0) -- (1.4,1.4) -- (0,1.4);
	\draw [black, densely dotted] (0,0) -- (1.4,1.4);
	\node [black, below] at (1.2,1.4) {$\Theta_0$};
	\node [below] at (0,0) {0};
	\node [below] at (1.4,0) {1};
	\node [left] at (0,1.4) {1};
	\node [right] at (1.5,0) {$p$};
	\node [above] at (0,1.5) {$q$};
	
	\draw [blue, dashed, thick, domain = 0:0.303] plot (\x, {(sqrt(\x)+sqrt(0.4)) *(sqrt(\x)+sqrt(0.4)) });
	\draw [blue, dashed, thick, domain = 0.4:1.4] plot (\x, {(sqrt(\x)-sqrt(0.4)) *(sqrt(\x)-sqrt(0.4)) });	
	
	\draw [blue, dashed, thick] (0,0.2) -- (0.2,0);
	
	\node [red, below] at (0.4,0) {$\frac{(c_1+c_3)\ln n}{n}$};
	\node [red, left] at (0,0.4) {$\frac{(c_1+c_3) \ln n}{n}$};
	\node [red] at (1,0.35) {$|p-q|=\sqrt{\frac{(c_1+c_3) \ln n}{n}}(\sqrt{p}+\sqrt{q})$};
	
	\node [red, left] at (0, 0.2) {$\frac{c_1 \ln n}{n}$};
	\node [red, below] at (0.1, 0) {$\frac{c_1 \ln n}{n}$};	
	

	
	\node [below] at (0.8,0.8) {{\bf IV}};

	
	
	\node at (0.05, 0.05) {{\bf III}}; 
	
	
	\node [above] at (1.35, 0.08) {{\bf I}};

	\node [below] at (0.15,1.25) {{\bf II}};
	\end{tikzpicture}
	\caption{Pictorial explanation of the minimax rate-optimal estimator in the unknown $Q$ case. Note that we use $(\hat{p}_1,\hat{q}_1)$ to determine which one of the four estimators to use, and then apply $(\hat{p}_2, \hat{q}_2)$ obtained from the second independent sample to estimate. We classify the 2D unit square into four regimes using $(\hat{p}_1, \hat{q}_1)$. The dashed diagonal line denotes the points where the function $f(p,q) = |p-q|$ is not analytic.   }
	\label{fig.2destimator}
	\end{figure}

%
%

The next theorem presents the performance of $\hat{L}^{(2)}$. 
\begin{theorem}\label{Thm.opt2}
Suppose there exists a constant $C>0$ such that $\ln n \leq C \ln S$. Then, there exists $c_1,c_2,c_3$ that only depend on $C$ in Construction~\ref{construction.2d} such that 
\begin{align}
\sup_{P,Q\in\cM_S} \bE |\hat{L}^{(2)} - \|P-Q\|_1|^2 \lesssim_C \frac{S}{n\ln n}.
\end{align}
\end{theorem}

We note that the lower bound for the known $Q$ case also serves as a lower bound for the unknown $Q$ case. Indeed, when $Q$ is known, we can then produce $n$ i.i.d. samples from $Q$ and feed it into any algorithm that handles the unknown $Q$ case. Hence, Theorem \ref{Thm.lowerbound} and Theorem \ref{Thm.opt2} yield that $\hat{L}^{(2)}$ is minimax rate-optimal. Note that $\hat{L}^{(2)}$ achieves the minimax rate without knowing the support size $S$ a priori. Moreover, the \emph{effective sample size enlargement} effect holds again: the performance of the optimal estimator with $n$ samples is essentially that of the MLE with $n\ln n$ samples.
\section{Comparison with Other Approaches}\label{sec.comparisionwithother}
In this section, we review some other possible approaches in estimating the $L_1$ distance, and apply approximation theory to argue the strict suboptimality of some approaches.
\subsection{Approximation only around the origin}
In the previous papers~\cite{Valiant--Valiant2011power,Valiant--Valiant2013estimating,Jiao--Venkat--Han--Weissman2015minimax,Wu--Yang2014minimax,Acharya--Orlitsky--Suresh--Tyagi2014complexity} in estimating entropy, power sum, mutual information, etc, approximation is conducted only around the origin. However, we remark that this is insufficient in estimating the $L_1$ distance. We have the following result. 
\begin{lemma}\label{lemma.zeroapproximationinsufficient}
Let $\hat{L}$ denote an estimator of $\|P-Q\|_1$ that satisfies the following:
\begin{align}
\hat{L} & = \sum_{i = 1}^S g(\hat{p}_i, \hat{q}_i),
\end{align}
where the estimator $g(\hat{p}_i,\hat{q}_i) \in [-B,B]$ is a bounded function that satisfies $g(\hat{p}_i, \hat{q}_i) = |\hat{p}_i - \hat{q}_i|$ when $(\hat{p}_i, \hat{q}_i) \notin \left[ 0, \frac{2c_1\ln n}{n} \right]^2$, $g(0,0) = 0$. Suppose $n\gg S$. Then, 
\begin{align}
\sup_{P,Q \in \mathcal{M}_S} \bE | \hat{L} - \|P-Q\|_1 |^2 \gg \frac{S}{n\ln n}. 
\end{align}
\end{lemma}

Lemma~\ref{lemma.zeroapproximationinsufficient} explains the reason why the estimator of Valiant and Valiant~\cite{Valiant--Valiant2011power} can only achieve the optimal error rate when $n\lesssim S\lesssim n\ln n$, but ours achieves the optimal error rate for a much large set of parameter configurations. 

\subsection{One-dimensional approximation in the 2D case}
In the construction of $\hat{L}^{(2)}$, we split into two cases when $(\hat{p},\hat{q})\in U_1$, i.e., 1D approximation of $|t|$ via the substitution $t=x-y$ if $\hat{p}+\hat{q}>c_1\ln n/n$, and the decomposition of $|x-y|$ into  $(\sqrt{x}+\sqrt{y})|\sqrt{x}-\sqrt{y}|$ otherwise. Can we always do 1D approximation of $|t|$ with $t=x-y$ to achieve the desired approximation error, i.e., propose some $P(t)\in\mathsf{poly}_K$ with $K\asymp \ln n$ and $|P(t)-|t||\lesssim \sqrt{|t|/(n\ln n)} + \frac{1}{n\ln n}$ for any $|t|\le c_1\ln n/n$? We have the following lemma regarding the approximation of $|t|$.
\begin{lemma}\label{lemma.pointwise}
If $Q_K(t)\in \mathsf{poly}_K$ is even with $Q_K(0)=0$, and achieves the best uniform error rate $\max_{t\in[-1,1]}|Q_K(t)-|t||\lesssim 1/K$, we have
\begin{align}
\limsup_{K\to\infty}\frac{1}{K}\sup_{0<|t|\le 1} \frac{\left | |t|-|Q_K(t)-|t||  \right |}{t^2} < \infty.
\end{align}
\end{lemma}

Now we apply Lemma~\ref{lemma.pointwise} to the hypothetical polynomial $P(t)$. Doing parameter substitution $t = \frac{c_1 \ln n}{n} y, y\in [-1,1]$, by assumption we have for any $y\in [-1,1]$, 
\begin{align}
\left | \frac{n}{c_1 \ln n} P\left( \frac{c_1 \ln n}{n} y \right) - |y| \right | \lesssim \frac{\sqrt{|y|}}{K} + \frac{1}{K^2},
\end{align}
where $K \asymp \ln n$. It follows from Jensen's inequality that
\begin{align}
& \left | \frac{n}{c_1 \ln n} \left( P\left( \frac{c_1 \ln n}{n} y \right) + P\left(- \frac{c_1 \ln n}{n} y \right) \right)/2 - |y| \right | \nonumber \\
& \quad \lesssim \frac{\sqrt{|y|}}{K} + \frac{1}{K^2}. \label{eqn.uppercontra}
\end{align}
Define $Q(y) = \frac{n}{c_1 \ln n} \left( P\left( \frac{c_1 \ln n}{n} y \right) + P\left(- \frac{c_1 \ln n}{n} y \right) \right)/2$. It is clear that $Q(y)$ satisfies the assumptions in Lemma~\ref{lemma.pointwise}. Hence,
\begin{align}
|Q(y) - |y|| \geq |y| - C K y^2. 
\end{align}
However, it contradicts the upper bound~(\ref{eqn.uppercontra}) when $\frac{1}{K^2} \ll |y| \ll \frac{1}{K}$. Hence, any 1D approximation does not achieve the error rate that is achieved by our 2D approximation approach. 
\subsection{Approximation on the entire 2D stripe}
In the unknown $Q$ case we have decomposed the stripe $U$ into subsets where polynomial approximations take place. Is it possible that we use a single polynomial $P(x,y) \in \poly_K^2$ of degree $K\asymp \ln n$ to approximate $|x-y|$ such that $|P(x,y)-|x-y||\lesssim\sqrt{(x+y)/(n\ln n)}$ for any $(x,y)\in U$? We prove that the answer is negative even for $U'=\cup_{x\in[c_1\ln n/n,t_n]}U(x;c_1)\times U(x;c_1)\subset U$ and any $t_n\gg (\ln n)^3/n$. 

\begin{lemma}\label{lemma.stripelowerbound}
	If $(\ln n)^3/n\ll t_n\le 1/2$, $K\asymp \ln n$, we have
	\begin{align*}
		\liminf_{n\to\infty}\sqrt{n\ln n}\cdot\inf_{P\in\mathsf{poly}_K^2} \sup_{(x,y)\in U'} \frac{|P(x,y)-|x-y||}{\sqrt{x+y}} = +\infty.
	\end{align*}
\end{lemma}

Lemma~\ref{lemma.stripelowerbound} shows that for a too large set $U'$ (e.g., $U'=U$), every polynomial fails to achieve the desired approximation error bound $\sqrt{(x+y)/(n\ln n)}$. Hence, it is necessary to make the approximation regime be random and dependent on the empirical observations. 



\subsection{The failure of any plug-in estimator}
It is evident that the optimal $L_1$ distance estimators we constructed heavily exploit the interactions of $P$ and $Q$. For example, in the known $Q$ case, the estimator for $\|P-Q\|_1$ is not of the form $\|g(P_n)-Q\|_1$, where $g(\cdot)$ is an arbitrary function of the empirical distribution of $P$ that is independent of $Q$. 

We show that for any estimator $g(P_n)$ of the distribution $P$, the plug-in estimator $\|g(P_n)-Q\|_1$ does not achieve the minimax rates in estimating $\|P-Q\|_1$ when one considers the worst cases among all $P,Q\in \mathcal{M}_S$. 

\begin{lemma}\label{lemma.separateplugin}
Consider the known $Q$ case. Suppose $g(P_n) \in \mathbb{R}^S$ is an arbitrary function of the empirical distribution $P_n$, and $g(\cdot)$ does not depend on $Q$. Then, if $n\gtrsim S$, 
\begin{align}
\sup_{P,Q\in \mathcal{M}_S} \bE_P \left( \|g(P_n)-Q\|_1 - \|P-Q\|_1 \right)^2 \gtrsim \frac{S}{n}.
\end{align}
\end{lemma} 

Lemma~\ref{lemma.separateplugin} shows that since the plug-in estimator $\|g(P_n)-Q\|_1$ does not explicitly exploit the nonsmoothness of the function $\|P-Q\|_1$, in the worst case it behaves essentially like the maximum likelihood estimator as shown in Corollary~\ref{Cor.mle}.

\section{Experimental Results} \label{sec.exp}

In this section, we compare the empirical performances of our algorithms with the following approaches:
\begin{itemize}
\item maximum likelihood estimator (MLE): it is the approach of plugging-in the empirical distributions obtained through samples into the functional. As shown in Section \ref{sec.knownq} and \ref{sec.unknownq}, it does not achieve the minimax rates in estimating $\| P - Q\|_1$ in general in both the known $Q$ and unknown $Q$ cases. 
\item Valiant--Valiant estimator \cite{Valiant--Valiant2013estimating}: \cite{Valiant--Valiant2013estimating} released Matlab code corresponding to their estimator of $\| P - U_S\|_1$, which is proved to achieve the minimax rates when $n \asymp \frac{S}{\ln S}$, i.e., when the optimal error is a constant. Here $U_S$ denotes the uniform distribution with support size $S$. 
\item approximate profile maximum likelihood estimator (APML) \cite{pavlichin2017approximate}: the APML estimator is an approximate solution of the profile maximum likelihood estimator \cite{acharya2017unified}, which can be applied to estimate $\|P - U_S\|_1$, and $\|P-Q\|_1$ when both $P$ and $Q$ are unknown. It was shown in \cite{pavlichin2017approximate} that the APML estimator exhibits generally good empirical performances, albeit its theoretical properties are not yet understood well. 
\end{itemize}

In the sequel, for each true distribution pair $(P,Q)$, we fix the parameters in our estimators and vary the sample sizes to compare the estimation performances. We use the root mean squared error (RMSE) as the evaluation criterion.

Figure \ref{qknown} compares the four approaches mentioned above in estimating $\|P-U_S\|_1$, which is also called ``distance to uniformity''. We see that our algorithm is consistently better than the maximum likelihood estimator, and is competitive with the VV estimator \cite{Valiant--Valiant2013estimating} and APML estimator \cite{pavlichin2017approximate}. Our estimator has computational complexity $O(n \ln n)$. Indeed, in the worst case, we may need to evaluate a polynomial with degree $\ln n$ for each sample, which results in an overall $O(n\ln n)$ computational complexity. 

\begin{figure}
\centering
\includegraphics[width=0.5\textwidth]{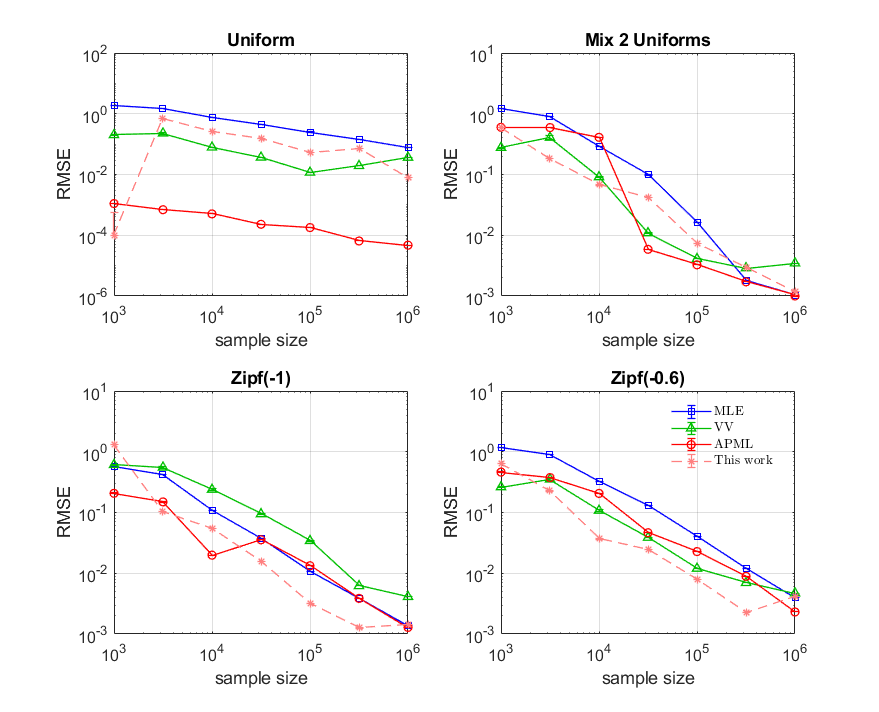}

\caption{Root mean squared error for estimation of $L_1$ distance to uniformity with fixed support size $S = 10^4$. The parameters in our algorithms are set to be $c_1 = 1.4, c_2 = 1.1, c_3 = 0.1$. Here ``Uniform'' refers to the uniform distribution with support size $S$, and ``Mix 2 Uniforms'' is a mixture of two uniform distributions, with half the probability mass on the first $S/5$ symbols, and the other half on the remaining $4S/5$ symbols, and $Zipf(\alpha) \sim 1/i^\alpha$ with $i \in \{1, . . . , S\}$.  Each data point represents 100 random trials, with 2 standard error bars smaller than the plot marker for most points. }
\label{qknown}
\end{figure}

Figure \ref{qunknown} compares the performances of the maximum likelihood estimator (MLE), our estimator, and APML in estimating $\|P-Q\|_1$ when both $P$ and $Q$ are unknown. Note that we did not choose to compare with \cite{Valiant--Valiant2013estimating} since there is no code available for their algorithm in the unknown $Q$ setting. We find our algorithm to perform consistently better than the maximum likelihood estimator, and is particularly competitive when the distributions $P$ and $Q$ are quite different from each other. Our estimator has computational complexity $O(n \ln^2 n)$ in the $Q$ unknown setting. In the worst case, we may need to evaluate a bivariate polynomial with degree $\ln n$ in each variable for each sample, which results in an overall $O(n\ln^2 n)$ computational complexity. 

\begin{figure}
\centering
\includegraphics[width=0.5\textwidth]{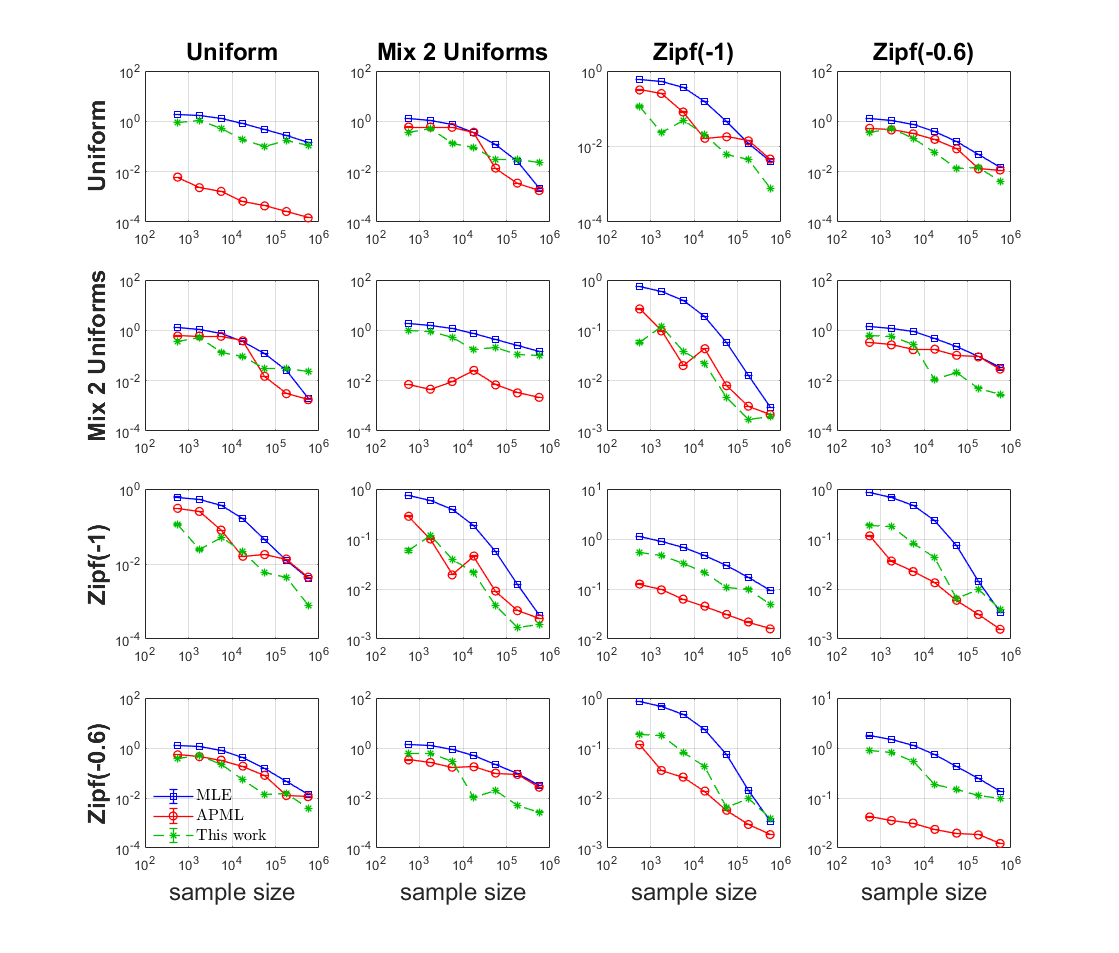}

\caption{ Root mean squared error for estimation of $\|P-Q\|_1$ where $P$ and $Q$ are indicated in the row and column names, respectively. The meanings of each specific $P,Q$ are the same as in Figure \ref{qknown} with fixed support size $S = 10^4$. The parameters of our estimator are set to be $c_1 = 3.6, c_2 = 1.2, c_3 = 1.3$. Each data point represents 100 random trials, with 2 standard error bars smaller than the plot marker for most points. Since $\|P-Q\|_1$ is symmetric in $P$ and $Q$, the plot exhibits symmetric behavior.
}
\label{qunknown}
\end{figure}

\section{Acknowledgements}

We are grateful to Vilmos Totik for discussing multivariate approximation theory, and for the insights that motivated the proof of Lemma~\ref{lemma.stripelowerbound}. We would like to thank Gregory Valiant for discussing the estimator in~\cite{Valiant--Valiant2011power}. We are grateful to the associated editor and the anonymous reviewers for constructive comments that have helped significantly improved the presentation of the paper. We thank Irena Fischer-Hwang and Banghua Zhu for the help in preparing the experimental results in Section \ref{sec.exp}.  
\appendices

\section{Auxiliary Lemmas}\label{sec.auxiliarylemmas}

The first-order symmetric difference of a function $f$ is given by
\begin{align}
\Delta_h^1 f(x) & = f\left( x + \frac{h}{2} \right) - f \left( x- \frac{h}{2} \right),
\end{align}
while the second order symmetric difference is given by
\begin{align}
\Delta_h^2 f(x) & = \Delta_h \left( \Delta_h^1 f(x) \right) \\
& = f(x+h) - 2f(x) + f(x-h). 
\end{align}
Analogously, the $r$-th order symmetric difference can be defined, and it is zero when $[x,x+rh]$ or $[x-rh,x]$ are not inside the domain of $f$. 

For function $f(x)$ with domain $[0,1]$, $\varphi(x) = \sqrt{x(1-x)}$, the first-order Ditzian--Totik modulus of smoothness is defined as
\begin{equation}\label{eqn.1stdtmodulus}
\omega^1_\varphi(f,t) \triangleq \sup_{0<h\leq t} \| \Delta_{h\varphi}^1 f(x) \|_\infty,
\end{equation}
and the second-order Ditzian--Totik modulus of smoothness is defined as
\begin{align}\label{eqn.2nddtmodulus}
\omega_\varphi^2(f,t) &  \triangleq \sup_{0<h\leq t} \| \Delta_{h\varphi}^2 f(x) \|_\infty. 
\end{align}

Similarly, we can also define the $r$-th order Ditzian--Totik modulus of smoothness for a function $f(x)$ with domain $[0,1]^2$:
\begin{align}
\omega_{[0,1]^2}^r(f,t) & = \sup_{1\leq i\leq 2, 0<h\leq t, x\in [0,1]^2} | \Delta^r_{i, h\varphi(x_i)} f(x) |, 
\end{align}
where $\Delta_{i,h}$ denotes the symmetric difference with respect to the $i$-th coordinate. 

The next lemma upper bounds the best polynomial approximation error by the Ditzian-Totik moduli. 
\begin{lemma}\cite[Thm. 7.2.1, Thm. 12.1.1.]{Ditzian--Totik1987}\label{lemma.ditziantotikgeneral}
There exists a constant $M(r)>0$ such that for any function $f\in C[0,1]$, 
\begin{align}
E_n(f;[0,1])\leq M(r)  \omega_\varphi^r(f,n^{-1}), \quad n >r,
\end{align}
where $E_n[f;I]$ denotes the distance of the function $f$ to the space $\poly_n$ in the uniform norm $\|\cdot\|_{\infty,I}$ on $I\subset \bR$. Moreover, if $f(x): [0,1]^2 \mapsto \mathbb{R}$, we have
\begin{align}
E_n[f; [0,1]^2] \leq M \omega^r_{[0,1]^2} (f, n^{-1}),
\end{align}
for any $r<n$, where $M$ is independent of $f$ and $n$, and $E_n[f; [0,1]^2]$ denotes the distance of the function $f$ to the space $\poly_n^2$ in the uniform norm on $[0,1]^2$. 
\end{lemma}

The modulus $\omega_\varphi^2(f,t)$ is computed for a variety of functions in the following lemma.
\begin{lemma}\cite[Chap. 3.4]{Ditzian--Totik1987}\label{lemma.dtcomputationchap34}
Suppose $f(x) = x^\delta, 0<\delta<1, x\in [0,1]$. Then,
\begin{align}
\omega_\varphi^2(f,t) & \asymp t^{2\delta} \\
\omega_\varphi^1(f,t) & \asymp \max\{t^{2\delta},t\}
\end{align}
where $\omega_\varphi^1(f,t)$ is defined in~(\ref{eqn.1stdtmodulus}), $\omega_\varphi^2(f,t)$ is defined in~(\ref{eqn.2nddtmodulus}). 
\end{lemma}

\begin{lemma}\label{lemma.firstorderdtmoduindependentofa}
Suppose $f(x;a) = |\sqrt{x} - \sqrt{a}|,x\in [0,1]$, and $a\in [0,1]$ is a parameter. Then, 
\begin{align}
\omega_\varphi^1(f,t) \leq \frac{t}{\sqrt{2}}. 
\end{align}
\end{lemma}

Next lemma computes the Ditzian--Totik modulus for function $f(x) = |2x\Delta -q|, x\in [0,1]$.  
 \begin{lemma}\label{lemma.dtmodulusknownq}
Suppose $f(x) = |2x\Delta -q|, \Delta>0, 0\leq q\leq 2\Delta, x\in [0,1]$. Then, for any integer $K\geq 1$, 
\begin{align}
\omega_\varphi^2(f, K^{-1}) & = \begin{cases}  2q & q \leq \frac{2\Delta}{1+K^2} \\ \frac{2\sqrt{q(2\Delta-q)}}{K} & \frac{2\Delta}{1+K^2} \leq q \leq \frac{2\Delta K^2}{1+K^2} \\ 2(2\Delta-q) & \frac{2\Delta K^2}{ 1+K^2}  \leq q \leq 2\Delta
 \end{cases}  \\
& \lesssim \min \left \{ q, \frac{\sqrt{q(2\Delta-q)}}{K}, (2\Delta-q) \right \}, 
\end{align}
where $\omega_\varphi^2(f,t)$ is defined in~(\ref{eqn.2nddtmodulus}). 
\end{lemma}


\begin{lemma}[Markov's inequality]\cite[Chap 4, Thm 1.4]{Devore--Lorentz1993} \label{lemma.markovinequality}
Suppose $P_n \in \poly_n$ is defined on $[-1,1]$. Then,
\begin{align}
\sup_{x\in [-1,1]} | P_n'(x) | \leq n^2 \sup_{x\in [-1,1]} |P_n(x)|
\end{align}
\end{lemma}

\begin{lemma}\cite[Thm. 7.3.1.]{Ditzian--Totik1987} \label{lemma.derivapproximationbound}
For $P_n$ the best $n$-th degree polynomial approximation to $f$ in $[0,1]$ and an integer $r\in \{1,2\}$ we have
\begin{align}
\sup_{x\in [0,1]} |\varphi^r P_n^{(r)}| \leq M n^r \omega_\varphi^r(f,n^{-1}),
\end{align}
where $\varphi(x) = \sqrt{x(1-x)}$ and $M$ is independent of $n$ and $f$. 
\end{lemma}


The next lemma shows that a polynomial on $[-1,1]$ nearly attains its supremum norm in a slightly smaller interval contained in $[-1,1]$. 
\begin{lemma}\cite[Thm. 8.4.8.]{Ditzian--Totik1987} \label{lemma.nsquareshrink}
Suppose $c>0$ is a constant, $P_n \in \poly_n$ defined on $[-1,1]$, $n^2 >c$. Then, there exists a constant $M(c)>0$ that does not depend on $n$ and $P_n$ such that
\begin{align}
\sup_{x\in [-1,1]} |P_n(x)| \leq M(c) \sup_{x\in [-1+cn^{-2}, 1-cn^{-2}]} |P_n(x) |. 
\end{align}
\end{lemma}

\begin{lemma}\label{lemma.evenpolyapproximation}
Suppose $P_K(x)$ is the best approximation polynomial with order $K$ of function $f(x)\in C[0,1]$ defined as
\begin{align}
P_K(x) & = \argmin_{P\in \poly_K} \max_{x\in [0,1]}|f(x) - P(x)|. 
\end{align}

Then, the best approximation polynomial with order $2K$ of function $f(z^2), z\in [-1,1]$ is given by $P_K(z^2)$. 
\end{lemma}

The following lemma characterizes the upper bounds of the coefficients of a bounded real polynomial.

\begin{lemma}\label{lem.polycoeff}\cite{Han--Jiao--Weissman2016minimaxdivergence}
	Let $p_n(x) = \sum_{\nu=0}^n a_\nu x^\nu$ be a polynomial of degree at most $n$ such that $|p_n(x)|\leq A$ for $x\in [a,b]$. Then
\begin{enumerate}
\item If $a+b\neq 0$, then
\begin{align}
|a_\nu| \le 2^{7n/2}A\left|\frac{a+b}{2}\right|^{-\nu}\left(\left|\frac{b+a}{b-a}\right|^n +1 \right), \qquad \nu=0,\cdots,n.
\end{align}
\item If $a+b = 0$, then
\begin{align}
|a_\nu| \leq A b^{-\nu} (\sqrt{2}+1)^n, \qquad \nu=0,\cdots,n.
\end{align}
\end{enumerate}
\end{lemma}

The following lemma gives an upper bound for the second moment of the unbiased estimate of $(p-q)^j$ in Poisson model.

\begin{lemma}\label{lemma.middlevariancebound}
Suppose $nX \sim \mathsf{Poi}(np),p\geq 0, q\geq 0$. Then, the estimator 
\begin{align}
g_{j,q}(X) \triangleq \sum_{k = 0}^j {j \choose k} (-q)^{j-k} \prod_{h = 0}^{k-1} \left( X - \frac{h}{n} \right)  
\end{align}
is the unique uniformly minimum variance unbiased estimator for $(p-q)^j,j\geq 0,j\in \mathbb{N}$, and its second moment is given by
\begin{align}
\mathbb{E}[(g_{j,q}(X))^2] & = \sum_{k = 0}^j {j \choose k}^2 (p-q)^{2(j-k)} \frac{p^k k!}{n^k} \\
& =  j! \left( \frac{p}{n} \right)^j L_j\left( - \frac{n(p-q)^2}{p} \right)\quad \text{Assuming }p>0,
\end{align}
where $L_m(x)$ stands for the Laguerre polynomial with order $m$, which is defined as:
\begin{align}
L_m(x) = \sum_{k = 0}^m {m \choose k} \frac{(-x)^k}{k!}
\end{align}
If $M \geq \max\left \{ \frac{n(p-q)^2}{p}, j  \right \}$, we have 
\begin{align}
\mathbb{E}[(g_{j,q}(X))^2] & \leq \left( \frac{2Mp}{n}\right)^j.
\end{align} 
When $k = 0, \prod_{h = 0}^{k-1} \left( X - \frac{h}{n} \right) \triangleq 1.$ When $p = 0$, $g_{j,q}(X) \equiv (-q)^j, \mathbb{E}[g_{j,q}(X)]^2 \equiv q^{2j}$. 
\end{lemma}

We construct the unbiased estimator of $(p-q)^j, j\geq 0$ when both $p$ and $q$ are unknown as in the following lemma. 
\begin{lemma}\label{lemma.middleconstraintvariancework}
Suppose $(n\hat{p},n\hat{q})\sim \mathsf{Poi}(np)\times \spo(nq)$. Then, the following estimator using $(\hat{p},\hat{q})$ is the unique uniformly minimum variance unbiased estimator for $(p-q)^j, j\geq 0, j\in \mathbb{Z}$:
\begin{align}
\hat{A}_j(\hat{p},\hat{q}) & = \sum_{k = 0}^j {j\choose k} \prod_{i = 0}^{k-1} (\hat{p}-\frac{i}{n}) (-1)^{j-k} \prod_{m = 0}^{j-k} (\hat{q}-\frac{m}{n}). 
\end{align} 
Furthermore, 
\begin{align}
\mathbb{E}\hat{A}_j^2 \leq  \left(  2(p-q)^2 \vee \frac{8j (p\vee q)}{n} \right)^j. 
\end{align}
\end{lemma}

The following lemma characterizes the behavior of the central moments of Poisson distributions. 
\begin{lemma}\label{lemma.poissoncentralmomenttight}
Suppose $n\hat{p} \sim \mathsf{Poi}(np)$. Then, for any integer $s\geq 2$, there exist $\lfloor s/2 \rfloor$ constants $h_{j,s}$ that are independent of $n$, such that
\begin{align}\label{eqn.poissoncentralmomentexpansion}
\mathbb{E}(\hat{p}-p)^s & = \frac{1}{n^s} \sum_{j = 1}^{\lfloor s/2 \rfloor} h_{j,s} (np)^j.
\end{align}
Furthermore,
\begin{align}
|h_{j,s}| & \leq \frac{2^j j^s}{j!}  \\
& \leq (2e^{e/(e-1)})^s \left( \frac{s}{\ln s} \right)^s. 
\end{align}
Consequently, there exists a constant $C_s>0$ depending only on $s$ satisfying $(C_s)^{1/s} \lesssim \frac{s}{\ln s}$ such that for any $s\geq 2$ an even integer,
\begin{align}
\mathbb{E}|\hat{p} - p|^s & \leq C_s \frac{(np)^{s/2} \vee (np)}{n^{s}}. 
\end{align}
For $s\geq 1$ odd integer, we have
\begin{align}
\mathbb{E}|\hat{p} - p|^s & \leq C_s \frac{(np)^{s/2} \vee (np)^{1/2}}{n^{s}}. 
\end{align}
\end{lemma}

We emphasize that the scaling $\left( \frac{s}{\ln s} \right)^s$ is consistent with the general moment bounds in~\cite{latala1997estimation}. However, the results in~\cite{latala1997estimation} do not directly apply here. Furthermore, Lemma~\ref{lemma.poissoncentralmomenttight} provides bounds on each individual $h_{j,s}$, which is not obtainable from a general moment bound. 


The next lemma controls the moments of $\frac{1}{\hat{p} \vee 1/n}$, where $n\hat{p} \sim \spo(np)$. 

\begin{lemma}\label{lemma.inversepoissonmoment}
Suppose $n\hat{p}\sim \spo(np)$. Then, for any integer $j\geq 0$, there exists a constant $B_j$ depending only on $j$ such that
\begin{align}
\mathbb{E}  \frac{1}{(\hat{p} \vee \frac{1}{n} )^j} & \leq  \frac{B_j}{p^j}. 
\end{align}
One may take $B_j = j \left( \frac{j}{e} \right)^j + 1 + j 2^{j+1} + j \left( \frac{16(j+1)}{e} \right)^{j+1}$. 
\end{lemma}

 The following lemma gives well-known tail bounds for Poisson and binomial random variables.
 \begin{lemma}\cite[Exercise 4.7]{mitzenmacher2005probability}\label{lemma.poissontail}
 	If $X\sim \spo(\lambda)$ or $X\sim \mathsf{B}(n,\frac{\lambda}{n})$, then for any $\delta>0$, we have
 	\begin{align}
 	\bP(X \geq (1+\delta) \lambda) & \leq \left( \frac{e^\delta}{(1+\delta)^{1+\delta}} \right)^\lambda \le e^{-\delta^2\lambda/3}\vee e^{-\delta\lambda/3} \\
 	\bP(X \leq (1-\delta)\lambda) & \leq  \left( \frac{e^{-\delta}}{(1-\delta)^{1-\delta}} \right)^\lambda \leq e^{-\delta^2 \lambda/2}.
 	\end{align}
 \end{lemma}
 
The following lemma presents the Hoeffding bound.
\begin{lemma}\label{lem_hoeffding}
  \cite{Hoeffding1963probability} 
Let $X_1,X_2,\ldots,X_n$ be independent random variables such that $X_i$ takes its value in $[a_i,b_i]$ almost surely for all $i\leq n$. Let $S_n=\sum_{i=1}^n X_i$, we have for any $t>0$,
  \begin{align}
    P\left\{|S_n-\bE[S_n]|\ge t\right\} \le 2\exp\left(-\frac{2t^2}{\sum_{i = 1}^n (b_i-a_i)^2}\right).
  \end{align}
\end{lemma}

The following lemma provides sharp estimates of $\bE |\hat{q}-q|$, where $n\hat{q}\sim \mathsf{Poi}(nq)$, which can be viewed as an analog of the binomial case studied in~\cite{Berend2013sharp}. 
\begin{lemma}\label{lemma.poissonmlebias}
Suppose $n\hat{q} \sim \mathsf{Poi}(nq)$. Then,
\begin{align}
\bE |\hat{q} - q| \in  \begin{cases} \{2q e^{-nq}\} & 0\leq q\leq \frac{1}{n} \\ [\sqrt{\frac{q}{2n}}, \sqrt{\frac{q}{n}}] & q \geq \frac{1}{n} \end{cases}. 
\end{align}
Hence, 
\begin{align}
\frac{1}{\sqrt{2}} \left( q \wedge \sqrt{\frac{q}{n}} \right) &  \leq  \bE |\hat{q} - q| \leq 2 \left( q \wedge \sqrt{\frac{q}{n}} \right). 
\end{align}
\end{lemma}

\begin{lemma}\label{lemma.dtxabsolutevaluefunction}
Suppose $n\hat{q} \sim \mathsf{Poi}(nq)$. Then, for any $p\geq 0$,
\begin{align}
\left| \mathbb{E}|\hat{q}-p| - |q-p| \right | \leq 2 \cdot \min\left \{ p, q, \sqrt{\frac{p}{n}}, \sqrt{\frac{q}{n}}   \right \}. 
\end{align}
Further, 
\begin{align}
\sup_{q\geq 0}\left| \mathbb{E}|\hat{q}-p| - |q-p| \right | \geq \frac{1}{\sqrt{2}} \left( p \wedge \sqrt{\frac{p}{n}} \right). 
\end{align}
\end{lemma}

The next lemma upper bounds the variance of $|\hat{q}-p|, n\hat{q} \sim \mathsf{Poi}(nq)$. 

\begin{lemma}\label{lemma.poissonmlevariance}
Suppose $n\hat{q}\sim \mathsf{Poi}(nq)$. Then, for any $p\geq 0$,
\begin{align}
\mathsf{Var}(|\hat{q}-p|) \leq \frac{q}{n}. 
\end{align}
\end{lemma}

%

\section{Proofs of main theorems}\label{sec.proofofmaintheorems}

\subsection{Proof of Theorem~\ref{Thm.mle}}

We have
\begin{align}
& \bE_P|\|P_n-Q\|_1 - \|P-Q\|_1|^2 \nonumber \\
&\quad = \left( \sum_{i = 1}^S \mathbb{E}|\hat{p}_i - q_i| - |p_i - q_i| \right)^2 + \mathsf{Var}( \|P_n-Q\|_1). 
\end{align}

Hence, 
\begin{align}
\left | \sum_{i = 1}^S \mathbb{E}|\hat{p}_i - q_i| - |p_i - q_i| \right | & \leq \sum_{i = 1}^S \mathbb{E} | |\hat{p}_i - q_i| - |p_i - q_i|| \\
& \leq \sum_{i = 1}^S 2 \left( q_i \wedge \sqrt{\frac{q_i}{n}} \right), 
\end{align}
where we applied Lemma~\ref{lemma.dtxabsolutevaluefunction}. 

To analyze the variance, due to the mutual independence of $\{\hat{p}_i, 1\leq i\leq S\}$, we have
\begin{align}
\mathsf{Var}(\|P_n-Q\|_1) & =  \sum_{i = 1}^S \mathsf{Var}(|\hat{p}_i -q_i|) \\
& \leq \sum_{i = 1}^S \frac{p_i}{n} \\
& = \frac{1}{n},
\end{align}
where we used Lemma~\ref{lemma.poissonmlevariance} in the second step. 

The proof of the upper bound is complete. Regarding the lower bound, setting $P = Q$, we have
\begin{align}
\left | \sum_{i = 1}^S \mathbb{E}|\hat{p}_i - q_i| - |p_i - q_i| \right |  & \geq \left | \sum_{i = 1}^S \mathbb{E}|\hat{p}_i - p_i| \right | \\
& =   \sum_{i = 1}^S \mathbb{E}|\hat{p}_i - p_i| \\
& \geq \frac{1}{\sqrt{2}} \sum_{i = 1}^S q_i \wedge \sqrt{\frac{q_i}{n}}. 
\end{align}

\subsection{Proof of Theorem~\ref{Thm.opt1}}

The following lemma gives the bias and variance bound of $\tilde{P}_K(\hat{p};q)$.
\begin{lemma}\label{lem.opt_1}
For $n\hat{p}\sim \mathsf{Poi}(np)$ with $p\in U(q;c_1)$, there exists a universal constant $B>0$ such that 
\begin{align}
|\bE \tilde{P}_K(\hat{p};q) - |p-q|| &\lesssim q\wedge \frac{1}{K}\sqrt{\frac{q c_1 \ln n}{n}}\\
\mathsf{Var}(\tilde{P}_K(\hat{p};q)) & \lesssim \frac{B^K c_1 \ln n}{n}(p + q)
\end{align}
where $\tilde{P}_K(\hat{p};q)$ is the unique uniformly minimum variance unbiased estimate of $P_K(x;q)$ defined in~(\ref{eqn.knownqpoly}), $U(q;c_1)$ is defined in~(\ref{eq.uncertain_set}) and $K = c_2 \ln n, c_2<c_1$. 

\end{lemma}

\begin{proof}
Recall the ``good'' events $E_1,E_2,E_3$ defined in~(\ref{eqn.error1}),(\ref{eqn.error2}),(\ref{eqn.error3}) and define $E = E_1 \cap E_2 \cap E_3$. We have
\begin{align}
& \bE (\hat{L}^{(1)} - \|P-Q\|_1)^2 \nonumber \\
& \quad = \mathbb{E} \left[ (\hat{L}^{(1)} - \|P-Q\|_1)^2 \mathbbm{1}(E) \right] \nonumber \\
& \qquad+ \mathbb{E} \left[ (\hat{L}^{(1)} - \|P-Q\|_1)^2\mathbbm{1}(E^c) \right ] \\
& \quad\leq \mathbb{E} \left[ (\tilde{L}_1 - \|P-Q\|_1)^2 \mathbbm{1}(E) \right] + 4 \bP(E^c) \\
& \quad \leq \mathbb{E} \left[ (\tilde{L}_1 - \|P-Q\|_1)^2 \mathbbm{1}(E) \right]  + \frac{12S}{n^\beta},
\end{align}
where we have applied Lemma~\ref{lemma.goodeventswin}. 

Define the random variables
\begin{align}
\mathcal{E}_1 & = \sum_{i \in I_1} (\hat{p}_{i,2} - q_i) - |p_i - q_i| \\
\mathcal{E}_2 & = \sum_{i \in I_2} (q_i - \hat{p}_{i,2}) -|p_i - q_i| \\
\mathcal{E}_3 & = \sum_{i  \in I_3} \tilde{P}_K(\hat{p}_{i,2};q_i) - |p_i - q_i|,
\end{align}
where the random index sets $I_1,I_2,I_3$ are defined as
\begin{align}
I_1 & = \{i: \hat{p}_{i,1} > U_1(q_i) , p_i \geq q_i\} \\
I_2 & = \{i: \hat{p}_{i,1} < U_1(q_i) , p_i \leq q_i\}\\
I_3 & = \{i: \hat{p}_{i,1} \in U_1(q_i) , p_i \in U(q_i;c_1)\}. 
\end{align}
The indices $I_1,I_2,I_3$ are independent of the random variables $\{\hat{p}_{i,2}: 1\leq i\leq S\}$. Since
\begin{align}
(\tilde{L}_1 - \|P-Q\|_1) \mathbbm{1}(E) & = \mathcal{E}_1 \mathbbm{1}(E) + \mathcal{E}_2 \mathbbm{1}(E) + \mathcal{E}_3 \mathbbm{1}(E),
\end{align}
it follows from Cauchy's inequality that
\begin{align}
\bE (\hat{L}^{(1)} - \|P-Q\|_1)^2  & \leq 3\left( \mathbb{E}\mathcal{E}_1^2 + \bE \mathcal{E}_2^2 + \bE \mathcal{E}_3^2 \right) + \frac{12S}{n^\beta},
\end{align}
where $\beta$ is defined in~(\ref{eqn.betadefinition}). 

It follows from the law of total variance that
\begin{align}
\bE \mathcal{E}_1^2 & = \mathbb{E}\left( \var(\mathcal{E}_1|I_1) + (\bE [\mathcal{E}_1|I_1])^2 \right) \\
& = \bE \var(\mathcal{E}_1|I_1) \\
& \leq \sum_{i =1}^S \frac{p_i}{n} \\
& = \frac{1}{n},
\end{align}
where we have used the fact that $\bE [\mathcal{E}_1|I_1] = 0$ with probability one and Lemma~\ref{lemma.poissonmlevariance}. Similarly we have $\bE \mathcal{E}_2^2 \leq n^{-1}$. 

Regarding $\bE \mathcal{E}_3^2$, it follows from Lemma~\ref{lem.opt_1} and the mutual independence of $\{\hat{p}_{i,2}: 1\leq i\leq S\}$ that
\begin{align}
\bE \mathcal{E}_3^2 & \lesssim \sum_{i = 1}^S \frac{B^K c_1 \ln n}{n} (p_i + q_i) + \left( \sum_{i = 1}^S q_i \wedge \sqrt{\frac{ c_1 q_i}{c_2^2 n\ln n}} \right)^2 \\
& \lesssim \left( \sum_{i = 1}^S q_i \wedge \sqrt{\frac{c_1 q_i}{c_2^2 n\ln n}} \right)^2 + \frac{c_1 \ln n}{n^{1-\epsilon}},
\end{align}
where $\epsilon = c_2 \ln B$. 

Hence, 
\begin{align}
& \bE (\hat{L}^{(1)} - \|P-Q\|_1)^2 \nonumber \\
& \quad \lesssim \left( \sum_{i = 1}^S q_i \wedge \sqrt{\frac{c_1 q_i}{c_2^2 n\ln n}} \right)^2 + \frac{c_1 \ln n}{n^{1-\epsilon}} + \frac{S}{n^\beta},
\end{align}
where $\beta$ is defined in~(\ref{eqn.betadefinition}) and $\epsilon = c_2 \ln B$. 

If $\ln n \gtrsim \ln S$, one may choose $c_1$ large enough and $c_3 = c_1/2$ to ensure that $\frac{S}{n^\beta} \lesssim \frac{\ln n }{n^{1-\epsilon}}$. When $\ln n \lesssim \ln \left(  \sum_{i = 1}^S \sqrt{q_i} \wedge q_i \sqrt{n\ln n}  \right)$, one may choose $c_2$ small enough to ensure that $\frac{\ln n }{n^{1-\epsilon}} \lesssim \left( \sum_{i = 1}^S q_i \wedge \sqrt{\frac{q_i}{n\ln n}} \right)^2$. The worst case of $Q$ result is proved upon noting that 
\begin{align}
\sum_{i = 1}^S q_i \wedge \sqrt{\frac{q_i}{n\ln n}}  & \leq \sum_{i =1}^S \sqrt{\frac{q_i}{n\ln n}} \\
& \leq \sqrt{\frac{S}{n\ln n}}. 
\end{align}
In the worst case of $Q$ we no longer need the condition $\ln n \gtrsim \ln S$ since we can ensure $\frac{S}{n^\beta} \lesssim \frac{S}{n\ln n}$ if we take $c_1$ large enough and $c_3 = c_1/2$. 

\end{proof}

\subsection{Proof of Theorem~\ref{Thm.lowerbound}}

The main tool we employ is the so-called method of two fuzzy hypotheses presented in Tsybakov \cite{Tsybakov2008}. Suppose we observe a random vector ${\bf Z} \in (\mathcal{Z},\mathcal{A})$ which has distribution $P_\theta$ where $\theta \in \Theta$. Let $\sigma_0$ and $\sigma_1$ be two prior distributions supported on $\Theta$. Write $F_i$ for the marginal distribution of $\mathbf{Z}$ when the prior is $\sigma_i$ for $i = 0,1$. 
Let $\hat{T} = \hat{T}({\bf Z})$ be an arbitrary estimator of a function $T(\theta)$ based on $\bf Z$. We have the following general minimax lower bound.

\begin{lemma}\cite[Thm. 2.15]{Tsybakov2008} \label{lemma.tsybakov}
	Given the setting above, suppose there exist $\zeta\in \mathbb{R}, s>0, 0\leq \beta_0,\beta_1 <1$ such that
	\begin{align}
	\sigma_0(\theta: T(\theta) \leq \zeta -s) & \geq 1-\beta_0 \\
	\sigma_1(\theta: T(\theta) \geq \zeta + s) & \geq 1-\beta_1.
	\end{align}
	If $\mathsf{TV}(F_1,F_0) \leq \eta<1$, then
	\begin{equation}
	\inf_{\hat{T}} \sup_{\theta \in \Theta} \bP_\theta\left( |\hat{T} - T(\theta)| \geq s \right) \geq \frac{1-\eta - \beta_0 - \beta_1}{2},
	\end{equation}
	where $F_i,i = 0,1$ are the marginal distributions of $\mathbf{Z}$ when the priors are $\sigma_i,i = 0,1$, respectively.
\end{lemma}

Here $\mathsf{TV}(P,Q)$ is the total variation distance between two probability measures $P,Q$ on the measurable space $(\mathcal{Z},\mathcal{A})$. Concretely, we have
\begin{align}
\mathsf{TV}(P,Q) & \triangleq \sup_{A\in \mathcal{A}} | P(A) - Q(A) | \label{eqn.tvdef} \\
& = \frac{1}{2} \int |p-q| d\nu \nonumber \\
& = \frac{1}{2} \| P-Q\|_1, \nonumber
\end{align}
where $p = \frac{dP}{d\nu}, q = \frac{dQ}{d\nu}$, and $\nu$ is a dominating measure so that $P \ll \nu, Q \ll \nu$.

The following lemma was shown in Cai and Low~\cite{Cai--Low2011}:
\begin{lemma}\label{lemma.cailowmeasureconstruction}
For any given even integer $L>0$, there exist two probability measures $\nu_0$ and $\nu_1$ on $[-1,1]$ that satisfy the following conditions:
\begin{enumerate}
\item $\nu_0$ and $\nu_1$ are symmetric around $0$;
\item $\int t^l \nu_1(dt) = \int t^l \nu_0(dt)$, for $l = 0,1,2,\ldots,L$;
\item $\int |t| \nu_1(dt) - \int |t|\nu_0(dt) = 2 E_L[|t|;[-1,1]]$,
\end{enumerate}
where $E_L[|t|;[-1,1]]$ is the distance in the uniform norm on $[-1,1]$ from the absolute value function $|t|$ to the space $\poly_L$. 
\end{lemma}

It is known that $E_L[|t|;[-1,1]] = \beta_* L^{-1}(1+o(1))$, where $\beta_* \approx 0.2802$ is the Bernstein constant~\cite{bernstein1914meilleure}.

The following lemma deals with the approximation theoretic properties of function $\frac{|x-a|-a}{x}$. 

\begin{lemma}\label{lemma.small1overslowerboundapproximation}
For any function $f(x;a) = \frac{|x-a|-a}{x}, x\in [0,1]$, there exists a universal constant $D>0$ such that
\begin{align}
 E_L[f(x;a); [\frac{a}{D},1]] \gtrsim \begin{cases} \frac{1}{L \sqrt{a}} &  \frac{1}{L^2} \leq a \leq \frac{1}{2} \\ 1 & 0<a<\frac{1}{L^2} \end{cases}
\end{align}
where $E_L[f;I]$ denotes the distance in the uniform norm on interval $I$ from the function $f$ to the space $\mathsf{poly}_L$. 
\end{lemma}

Similar to Lemma~\ref{lemma.cailowmeasureconstruction}, the next lemma constructs two measures for the function $f(x;a) = \frac{|x-a|-a}{x}$. The proof is essentially identical to that of Lemma~\ref{lemma.cailowmeasureconstruction}. 
\begin{lemma}\label{lemma.fxameasureconstruction}
For any $0<\eta<1$ and positive integer $L>0$, $f(x;a) = \frac{|x-a|-a}{x},a\in [0,1]$, there exist two probability measures $\nu_1^{\eta, a}, \nu_0^{\eta,a}$ on $[\eta,1]$ such that
\begin{enumerate}
\item $\int t^l \nu_1^{\eta, a}(dt) = \int t^l \nu_0^{\eta, a}(dt)$, for all $l = 0,1,2,\ldots,L$;
\item $\int f(t;a) \nu_1^{\eta, a}(dt) - \int f(t;a)\nu_0^{\eta, a}(dt) = 2 E_L[f(x;a);[\eta,1]]$,
\end{enumerate}
where $E_L[f(x;a);[\eta,1]]$ is the distance in the uniform norm on $[\eta,1]$ from the function $f(x;a)$ to the space $\mathsf{poly}_L$.
\end{lemma}

The next lemma is an extension of~\cite[Lemma 3]{Wu--Yang2014minimax}. 
\begin{lemma}\label{lemma.mixturepoissontvbound}
Suppose $U_0, U_1$ are two random variables supported on $[a-M, a+M]$, where $a\geq M \geq 0$ are constants. Suppose $\mathbb{E}[U_0^j] = E[U_1^j], 0\leq j \leq L$. Denote the marginal distribution of $X$ where $X|\lambda \sim \spo(\lambda), \lambda \sim U_i$ as $F_i$. If $L+1 \geq (2e M)^2 /a$, then 
\begin{align}
\mathsf{TV}(F_0,F_1) & \leq 2 \left( \frac{e M}{\sqrt{a(L+1)}} \right)^{L+1},
\end{align} 
where $\mathsf{TV}(F_0,F_1)$ is the total variation distance defined in~(\ref{eqn.tvdef}). 
\end{lemma}

We consider the set of \emph{approximate} probability vectors 
\begin{align}
\mathcal{M}_S(\epsilon) = \left \{ P: \left| \sum_{i = 1}^S p_i - 1 \right | \leq \epsilon\right \},
\end{align}
with some constant $\epsilon>0$. We further define the minimax risk under the Poisson sampling model with respect to $\mathcal{M}_S(\epsilon)$ with a fixed $Q$ as
\begin{align}
R_P(S,n,Q, \epsilon) & = \inf_{\hat{L}} \sup_{P \in \mathcal{M}_S(\epsilon)} \mathbb{E}_P \left(  \hat{L} - \|P-Q\|_1 \right)^2. 
\end{align}
The following lemma relates $R_P(S,n,Q,\epsilon)$ to $R_P(S,n,Q,0)$. 
\begin{lemma}\label{lemma.approximateprobability}
For any $S,n\in \mathbb{N}_+,  0<\epsilon <1$ and any distribution $Q\in \mathcal{M}_S$, we have
\begin{align}
& R_P(S, n(1-\epsilon)/4, Q, 0 ) \nonumber \\
& \quad \geq \frac{1}{4}R_P(S,n,Q,\epsilon) -\frac{1}{2} e^{-n(1-\epsilon)/8} - \frac{1}{2} \epsilon^2. 
\end{align}
\end{lemma}

Now we are ready to prove our main minimax lower bound. 
\begin{proof}

Fix the distribution $Q \in \mathcal{M}_S$. Without loss of generality we assume that $q_S = \min_{1\leq j\leq S} q_j$. We construct two probability measures $\bm{\mu}_0, \bm{\mu}_1$ on the distribution $P$ that will later be used in Lemma~\ref{lemma.tsybakov}. Concretely, we use an independent prior generation, and set
\begin{align}
\bm{\mu}_0 & = \mu_0^{(q_1)}\otimes \mu_0^{(q_2)} \otimes \ldots \otimes \delta_{1-\gamma} \\ 
\bm{\mu}_1 & = \mu_1^{(q_1)}\otimes \mu_1^{(q_2)} \otimes \ldots \otimes \delta_{1-\gamma}. 
\end{align}
In other words, we assign independent priors $\mu_i^{(q_j)}$ to each symbol $p_j, 1\leq j\leq S-1$, and assign a delta mass at $1-\gamma$ to the symbol $p_S$. The constant $\gamma$ will later be set to 
\begin{align}
\gamma & = \sum_{j: q_j\leq \frac{c\ln n}{n}} \frac{q_j}{D} + \sum_{j: q_j>\frac{c\ln n}{n}} q_j, \label{eqn.gammadefinition}
\end{align}
where $D$ is the universal constant in Lemma~\ref{lemma.small1overslowerboundapproximation}, and $c\in (0,1)$ is a constant. 

Now we construct $\mu_i^{(q)}, i \in \{0,1\}$ for a generic $q\in (0,1)$. We consider two different cases.
\begin{enumerate}
\item $0<q \leq \frac{c\ln n}{n}$, where $c\in (0,1)$ is a constant. We first construct two new probability measures $\tilde{\nu}_i^{\eta,a}, i = 0,1$ from the two probability measures constructed in Lemma~\ref{lemma.fxameasureconstruction}. For $i = 0,1$, the restriction of $\tilde{\nu}_i^{\eta,a}$ is absolutely continuous with respect to $\nu_i$, with the Radon--Nikodym derivative given by
\begin{align}
\frac{d\tilde{\nu}_i^{\eta,a}}{d\nu_i^{\eta,a}} & = \frac{\eta}{t} \leq 1, \quad t\in [\eta,1],
\end{align}
and $\tilde{\nu}_i^{\eta,a}(\{0\}) = 1- \tilde{\nu}_i^{\eta,a}([\eta,1]) \geq 0$. Hence, $\tilde{\nu}_i^{\eta,a},i = 0,1$ are probability measures on $[0,1]$, with the following properties:
\begin{enumerate}
\item $\int t \tilde{\nu}_1^{\eta,a} (dt) = \int t \tilde{\nu}_0^{\eta,a}(dt)  = \eta$; 
\item $\int t^l \tilde{\nu}_1^{\eta,a}(dt)  = \int t^l \tilde{\nu}_0^{\eta,a}(dt) $, for all $l = 2,3,\ldots,L+1$;
\item $\int |x-a| \tilde{\nu}_1^{\eta,a}(dt)  - \int |x-a| \tilde{\nu}_0^{\eta,a}(dt)  = 2\eta E_L[f(x;a);[\eta,1]]$. 
\end{enumerate}
The construction of the Radon--Nikodym derivatives are inspired by Wu and Yang~\cite{Wu--Yang2014minimax}. Define
\begin{align}
L = d_2 \ln n, \eta = \frac{a}{D}, a = \frac{q}{M}, M = \frac{2c\ln n}{n}, 
\end{align}
where $D$ is the universal constant in Lemma~\ref{lemma.small1overslowerboundapproximation} and $d_2>1$ is a constant. It follows from the assumption that $0<a\leq \frac{1}{2}$. Let $g(x) = Mx$ and let $\mu_i^{(q)}$ be the measures on $[0,M]$ defined by $\mu^{(q)}_i(A) = \tilde{\nu}_i^{\eta, a} (g^{-1}(A))$ for $i = 0,1$. It then follows that
\begin{align}
\int t \mu_1^{(q)}(dt) & = \int t \mu^{(q)}_0(dt) = \frac{q}{D}; 
\end{align}
\begin{align}
\int t^l \mu^{(q)}_1(dt) & = \int t^l \mu^{(q)}_0(dt), \text{ for all }l = 2,3,\ldots,L+1;  
\end{align}
\begin{align}
& \int |t-q| \mu^{(q)}_1(dt) - \int |t-q| \mu^{(q)}_0(dt) \nonumber \\
&\quad = 2 \eta M E_L[f(x;a);[\eta,1]] \\
& \quad \gtrsim q \wedge \sqrt{\frac{c q}{d_2^2 n\ln n}}. \label{eqn.biasboundsmallqlower}
\end{align}
\item $q> \frac{c\ln n}{n}$, where $c\in (0,1)$ is a constant. Define function $g(x) = q + \sqrt{\frac{c q\ln n}{n}} x$, where $x\in [-1,1]$. Let $\nu_i, i = 0,1$ be the two measures constructed in Lemma~\ref{lemma.cailowmeasureconstruction}. We define two new measures $\mu^{(q)}_i, i = 0,1$ by $\mu^{(q)}_i(A) = \nu_i (g^{-1}(A))$. Let
\begin{align}
L = d_2 \ln n, d_2 >1. 
\end{align}
It then follows that
\begin{align}
\int t \mu^{(q)}_0(dt) & = \int t \mu^{(q)}_1(dt)  = q;
\end{align}
\begin{align}
\int t^l \mu^{(q)}_0(dt) & = \int t^l \mu^{(q)}_1(dt), \text{ for all }l = 2,3,\ldots,L+1;  
\end{align}
\begin{align}
& \int |t-q| \mu^{(q)}_1(dt) - \int |t-q| \mu^{(q)}_0(dt) \nonumber \\
& \quad = 2 \sqrt{\frac{cq\ln n}{n}} E_L[|t|;[-1,1]] \\
& \quad \gtrsim q \wedge \sqrt{\frac{c q}{d_2^2 n\ln n}}.  \label{eqn.biasboundlargeqlower}
\end{align}
\end{enumerate}

Since we have set $p_S = 1-\gamma$, where $\gamma$ is defined in~(\ref{eqn.gammadefinition}), it is clear that
\begin{align}
\mathbb{E}_{\bm{\mu}_0} \left[ \sum_{j = 1}^S p_j \right] & = \mathbb{E}_{\bm{\mu}_1} \left[ \sum_{j = 1}^S p_j \right]  \\
& = 1. 
\end{align}

Now the construction of the two priors $\bm{\mu_0}$ and $\bm{\mu_1}$ are complete.  In light of Lemma~\ref{lemma.approximateprobability}, it suffices to lower bound $R_P(S,n,Q, \epsilon)$ to give a lower bound to $R_P(S,n,Q,0)$.

Let 
\begin{align}
\epsilon = \frac{\chi}{10}, \chi = \mathbb{E}_{\bm{\mu}_1} \|P-Q\|_1 - \mathbb{E}_{\bm{\mu}_0} \|P-Q\|_1.
\end{align}

We know from~(\ref{eqn.biasboundsmallqlower}) and~(\ref{eqn.biasboundlargeqlower}) that 
\begin{align}
\chi & \gtrsim \sum_{j = 1}^{S-1} q_j \wedge \sqrt{\frac{c q_j}{d_2^2 n\ln n}} \\
& \geq \left( 1- \frac{1}{S} \right) \sum_{j = 1}^{S} q_j \wedge \sqrt{\frac{c q_j}{d_2^2 n\ln n}} \\
& \gtrsim  \sum_{j = 1}^{S} q_j \wedge \sqrt{\frac{c q_j}{d_2^2 n\ln n}},
\end{align}
since we have assumed that $q_S = \min_{1\leq j\leq S}q_j$. 

For $i = 0,1$, introduce the events
\begin{align}
E_i & = \mathcal{M}_S(\epsilon) \bigcap \left \{ P: |\|P-Q\|_1 - \mathbb{E}_{\bm{\mu}_i} \|P-Q\|_1 | \leq \frac{\chi}{4} \right \}. 
\end{align}
It follows from the union bound that
\begin{align}
\bm{\mu}_i[(E_i)^c] & \leq \bm{\mu}_i\left( \left | \sum_{j = 1}^{S} p_j  - 1 \right | > \epsilon \right) \nonumber \\
& \quad + \bm{\mu}_i \left( | \|P-Q\|_1 - \mathbb{E}_{\bm{\mu}_i}[\|P-Q\|_1] | > \frac{\chi}{4} \right). 
\end{align}
Introduce 
\begin{align}
F(Q) & = \sum_{j: q_j\leq \frac{c\ln n}{n}} \left( \frac{2c\ln n}{n} \right)^2 + \sum_{j: q_j > \frac{c\ln n}{n}} \frac{4cq_j \ln n}{n} \\
& \leq \frac{4c^2 S \ln^2 n}{n^2} + \frac{4c \ln n}{n}. 
\end{align}
It follows from the Hoeffing inequality in Lemma~\ref{lem_hoeffding} that
\begin{align}
\bm{\mu}_i[(E_i)^c] & \leq 2\exp\left( -\frac{2\epsilon^2}{F(Q)} \right) + 2 \exp \left( - \frac{\chi^2}{8 F(Q)} \right) \\
& \to 0.
\end{align}

The last step follows from the arguments below. Note that we assumed $c\in (0,1),d_2>1, \sum_{j = 1}^S q_j \wedge \sqrt{\frac{q_j}{n\ln n}} \geq C' \left( \sqrt{\frac{\ln n}{n}} + \frac{\sqrt{S}\ln n}{n} \right)$. We have
\begin{align}
\epsilon & \asymp \chi \\
& \asymp \sum_{j = 1}^S q_j \wedge \sqrt{\frac{cq_j}{d_2^2 n\ln n}} \\
& \geq \frac{\sqrt{c}}{d_2} \sum_{j =1}^S q_j \wedge \sqrt{\frac{q_j}{n\ln n}} \\
& \geq \frac{\sqrt{c}}{d_2} C' \left( \sqrt{\frac{\ln n}{n}} + \frac{\sqrt{S}\ln n}{n} \right) \\
& \gtrsim \frac{\sqrt{c}}{d_2} C' \left( \sqrt{\frac{c \ln n}{n}} + \frac{c\sqrt{S}\ln n}{n} \right) \\
& \gtrsim \frac{\sqrt{c}}{d_2}C' \sqrt{F(Q)}. 
\end{align}
Hence, it suffices to take $C'$ large enough to ensure that $\bm{\mu}_i[(E_i)^c]\to 0, i = 0,1$.


Denote by $\pi_i$ the conditional distribution defined as
\begin{align}
\pi_i(A) & =  \frac{\bm{\mu}_i (E_i \cap A)}{\bm{\mu}_i(E_i)}, i = 0,1.
\end{align}

Now consider $\pi_0, \pi_1$ as two priors and denote the corresponding marginal distributions on the observations $(X_1,X_2,\ldots,X_S)$ as $F_0, F_1$. Note that $X_j \sim \spo(np_j)$. Setting 
\begin{align}
\zeta & = \mathbb{E}_{\bm{\mu}_0}[ \|P-Q\|_1] + \frac{\chi}{2} \\
s & = \frac{\chi}{4},
\end{align}
we have $\beta_0 = \beta_1 = 0$ in Lemma~\ref{lemma.tsybakov}. The total variation distance is then upper bounded as
\begin{align}
\mathsf{TV}(F_0,F_1) & \leq \mathsf{TV}(F_0,G_0) + \mathsf{TV}(G_0,G_1) + \mathsf{TV}(G_1,F_1) \\
& \leq \bm{\mu}_0 [(E_0)^c] + \mathsf{TV}(G_0,G_1) + \bm{\mu}_1[(E_1)^c] \\
& \leq \mathsf{TV}(G_0,G_1) + o(1),
\end{align}
where $G_i$ is the marginal distribution of the observations under priors $\bm{\mu}_i$. It follows from Lemma~\ref{lemma.mixturepoissontvbound} and the fact that $\mathsf{TV}(\otimes_{i = 1}^S P_i, \otimes_{i = 1}^S Q_i) \leq \sum_{i = 1}^S \mathsf{TV}(P_i, Q_i)$ that  
\begin{align}
\mathsf{TV}(G_0, G_1) & \leq \sum_{i = 1}^{S-1} 2 \left (\frac{1}{2} \right )^{d_2 \ln n} \\
& \leq \frac{2S}{2^{d_2 \ln n}} \\
& = \frac{2S}{n^{d_2 \ln 2}} \\
& \to 0,
\end{align}
since we have assumed $\ln n \geq C \ln S$, and we ensure $\mathsf{TV}(G_0, G_1)$ by taking $d_2$ large enough. 

It follows from Lemma~\ref{lemma.tsybakov} and Markov's inequality that
\begin{align}
R_P(S,n,Q,\epsilon) & \geq s^2 \cdot \inf_{\hat{L}} \sup_{P \in \mathcal{M}_S(\epsilon)} \mathbb{P} \left( | \hat{L} - \|P-Q\|_1 | \geq s \right) \\
& \geq \frac{s^2}{2} (1-o(1)) \\
& = \frac{\chi^2}{32}(1-o(1)), 
\end{align} 
which together with Lemma~\ref{lemma.approximateprobability} implies that
\begin{align}
& R_P(S,n(1-\epsilon)/4,Q, 0) \nonumber \\
&\quad \geq \frac{1}{4}R_P(S,n,Q,\epsilon) - \frac{1}{2} e^{-n(1-\epsilon)/8} - \frac{1}{2} \epsilon^2 \\
& \quad \geq \frac{\chi^2}{128}(1-o(1))  - \frac{1}{2} e^{-n(1-\chi/10)/8}  - \frac{\chi^2}{200}\\ 
& \quad \gtrsim \chi^2 \\
& \quad \gtrsim \left( \sum_{i = 1}^S q_i \wedge \sqrt{\frac{q_i}{n\ln n}} \right)^2,
\end{align}
as long as we choose the constants $d_2$ large enough to guarantee that $\chi \leq 5$.

\end{proof}

\subsection{Proof of Theorem~\ref{Thm.opt2}}

We first present the performance of the estimator $\tilde{P}_K^{(1)}(\hat{p}_{i,2},\hat{q}_{i,2})$ when $(p,q) \in \left[ 0, \frac{2c_1 \ln n}{n} \right]^2$. 

\begin{lemma}\label{lemma.2dperformancesquare}
Suppose $(p,q) \in \left[ 0, \frac{2c_1 \ln n}{n} \right]^2$, $(n \hat{p}, n\hat{q}) \sim \spo(np) \times \spo(nq)$. Then, 
\begin{align}
| \mathbb{E} \tilde{P}^{(1)}_K(\hat{p},\hat{q}) - |p-q| | & \lesssim \frac{1}{K}\sqrt{\frac{c_1 \ln n}{n}} (\sqrt{p} +\sqrt{q}) + \frac{1}{K^2} \frac{c_1 \ln n}{n} \\
\mathsf{Var}(\tilde{P}_K^{(1)}(\hat{p},\hat{q})) & \lesssim \frac{B^K c_1c_2^4 \ln^5 n}{n} (p+q),
\end{align}
for some constant $B>0$. The estimator $\tilde{P}_K^{(1)}$ is introduced in~(\ref{eqn.2dpk1}), and $K = c_2 \ln n, c_2 < c_1$. 
\end{lemma}

We then analyze the estimator $\tilde{P}_K^{(2)}(\hat{p}_{i,2},\hat{q}_{i,2}; \hat{p}_{i,1}, \hat{q}_{i,1})$ when $(p,q) \in U, p+q \geq \frac{c_1 \ln n}{2n}$. 

\begin{lemma}\label{lemma.2dperformancepqlarge}
Suppose $(p,q) \in U, p+q \geq \frac{c_1 \ln n}{2n}, x+y \geq \frac{p+q}{2}, x\in [0,1],y\in [0,1]$, where the set $U$ is defined in~(\ref{eq.stripe}). Suppose $(n\hat{p}, n\hat{q}) \sim \spo(np) \times \spo(nq)$. Then,
\begin{align}
\left | \mathbb{E} \tilde{P}_K^{(2)}(\hat{p}, \hat{q} ; x,y) - |p-q| \right | \lesssim \frac{1}{K} \sqrt{\frac{c_1 \ln n}{n}} (\sqrt{x} + \sqrt{y}) \\
\mathsf{Var}(\tilde{P}_K^{(2)}(\hat{p}, \hat{q} ; x,y)) \lesssim \frac{B^K c_1 \ln n}{n}(x+y),
\end{align}
for some constant $B>0$. The estimator $\tilde{P}_K^{(2)}$ is introduced in~(\ref{eqn.2dpk2}), and $K = c_2 \ln n, c_2 < c_1$. 
\end{lemma}

\begin{proof}
Recall the ``good'' events $E_1,E_2,E_3,E_4$ defined in~(\ref{eqn.2derror1}),(\ref{eqn.2derror2}),(\ref{eqn.2derror3}),(\ref{eqn.2derror4}) and introduce $E = E_1 \cap E_2 \cap E_3 \cap E_4$. We have
\begin{align}
& \mathbb{E} \left( \hat{L}^{(2)} - \|P-Q\|_1 \right)^2 \nonumber \\
&\quad = \mathbb{E}\left[ (\hat{L}^{(2)} - \|P-Q\|_1)^2 \mathbbm{1}(E) \right] \nonumber \\
& \qquad + \mathbb{E} \left[  (\hat{L}^{(2)} - \|P-Q\|_1)^2 \mathbbm{1}(E^c) \right] \\
& \quad \leq  \mathbb{E}\left[ ( \tilde{L}_2 - \|P-Q\|_1)^2 \mathbbm{1}(E) \right] + 4 \mathbb{P}(E^c) \\
& \quad \leq \mathbb{E}\left[ ( \tilde{L}_2 - \|P-Q\|_1)^2 \mathbbm{1}(E) \right] + \frac{60 S}{n^\beta},
\end{align}
where we have applied Lemma~\ref{lemma.2dgoodeventswin} and the constant $\beta$ is defined in~(\ref{eqn.2dbetadefinition}). 

Define the random variables
\begin{align}
\mathcal{E}_1 & = \sum_{i \in I_1} \left(  \hat{p}_{i,2} - \hat{q}_{i,2} - |p_i - q_i| \right) \\
\mathcal{E}_2 & = \sum_{i \in I_2} \left(  \hat{q}_{i,2} - \hat{p}_{i,2} - |p_i - q_i| \right) \\
\mathcal{E}_3 & = \sum_{i \in I_3} \left(  \tilde{P}_K^{(1)}(\hat{p}_{i,2},\hat{q}_{i,2}) - |p_i - q_i| \right) \\
\mathcal{E}_4 & = \sum_{i\in I_4} \left(  \tilde{P}_K^{(2)}(\hat{p}_{i,2},\hat{q}_{i,2};\hat{p}_{i,1},\hat{q}_{i,1}) - |p_i - q_i| \right)
\end{align}
where the random index sets $I_1,I_2,I_3,I_4$ are defined as
\begin{align}
I_1 & = \Bigg \{i:  \hat{p}_{i,1} - \hat{q}_{i,1} > \sqrt{\frac{(c_1 + c_3)\ln n}{n}}(\sqrt{\hat{p}_{i,1}} + \sqrt{\hat{q}_{i,1}}) , \nonumber \\
& \qquad p_i \geq q_i \Bigg \} \\
I_2 & = \Bigg \{ i:  \hat{p}_{i,1} - \hat{q}_{i,1} < - \sqrt{\frac{(c_1 + c_3)\ln n}{n}}(\sqrt{\hat{p}_{i,1}} + \sqrt{\hat{q}_{i,1}}) , \nonumber \\
& \qquad p_i \leq q_i \Bigg \}  \\
I_3 & = \left \{ i: \hat{p}_{i,1}+\hat{q}_{i,1} < \frac{c_1 \ln n}{n} , (p_i, q_i)\in  \left[ 0, \frac{2c_1 \ln n}{n} \right]^2 \right \} \\
I_4 & = \Bigg \{ i: (\hat{p}_{i,1},\hat{q}_{i,1} ) \in U_1, \hat{p}_{i,1} + \hat{q}_{i,1} \geq \frac{c_1 \ln n}{n} , \nonumber \\
& \qquad (p_i,q_i) \in U, p_i+q_i\geq \frac{c_1 \ln n}{2n}, \hat{p}_{i,1} +\hat{q}_{i,1} \geq \frac{p_i + q_i}{2} \Bigg \}. 
\end{align}

The index sets $I_1,I_2,I_3,I_4$ are independent of the random variables $\{\hat{p}_{i,2}: 1\leq i\leq S\}$ and $\{\hat{q}_{i,2}: 1\leq i\leq S\}$. It follows from the definition of the $E_i$'s that
\begin{align}
& \left( \tilde{L}_2 - \|P-Q\|_1 \right) \mathbbm{1}(E) \nonumber \\
& \quad = \mathcal{E}_1 \mathbbm{1}(E) + \mathcal{E}_2 \mathbbm{1}(E) + \mathcal{E}_3 \mathbbm{1}(E) + \mathcal{E}_4 \mathbbm{1}(E).
\end{align}
Hence, it follows from the Cauchy--Schwarz inequality that
\begin{align}
\mathbb{E} \left( \hat{L}^{(2)} - \|P-Q\|_1 \right)^2 & \leq 4 \sum_{j = 1}^4 \mathbb{E} \left( \mathcal{E}_j^2  \right) + \frac{60S}{n^\beta} 
\end{align}

It follows from the law of total variance that
\begin{align}
\mathbb{E} \mathcal{E}_1^2 & = \mathbb{E} \left(  \mathsf{Var}(\mathcal{E}_1 | I_1) + (\mathbb{E}[\mathcal{E}_1|I_1])^2 \right) \\
& = \mathbb{E} \mathsf{Var}(\mathcal{E}_1|I_1) \\
& \leq \sum_{i = 1}^S \frac{p_i + q_i}{n} \\
& = \frac{2}{n},
\end{align}
where we have used the fact that $\mathbb{E}[\mathcal{E}_1|I_1] = 0$ with probability one, the independence of $\hat{p}_{i,2}$ and $\hat{q}_{i,2}$, and Lemma~\ref{lemma.poissonmlevariance}. Similarly we have $\mathbb{E} \mathcal{E}_2^2 \leq \frac{2}{n}$. 

Regarding $\mathbb{E} \mathcal{E}_3^2$, it follows from Lemma~\ref{lemma.2dperformancesquare} and the mutual independence of $\{\hat{p}_{i,2}: 1\leq i\leq S\}$ and $\{\hat{q}_{i,2}: 1\leq i\leq S\}$ that
\begin{align}
& \mathbb{E} \mathcal{E}_3^2 \nonumber \\
& \quad \lesssim \sum_{i = 1}^S \frac{B^K c_1 c_2^4 \ln^5 n}{n} (p_i + q_i)  \nonumber \\
& \qquad +  \left( \sum_{i =1}^S \frac{\sqrt{c_1 p_i} + \sqrt{c_1 q_i}}{\sqrt{c_2^2 n\ln n}} + \frac{c_1}{c_2^2 n\ln n} \right)^2 \\
& \quad \lesssim \frac{c_1 c_2^4 \ln^5 n}{n^{1-\epsilon}} + \frac{c_1 S}{c_2^2 n\ln n} \vee \left( \frac{c_1 S}{c_2^2 n \ln n} \right)^2,
\end{align}
where $\epsilon = c_2 \ln B$. 

Regarding $\mathbb{E} \left( \mathcal{E}_4^2  \right)$, it follows from the bias-variance decomposition and Lemma~\ref{lemma.2dperformancepqlarge} that
\begin{align}
& \mathbb{E}[\mathcal{E}_4^2 | \{ (\hat{p}_{i,1},\hat{q}_{i,1}):1\leq i\leq S \}] \nonumber \\
& \quad \lesssim \sum_{i = 1}^S \frac{B^K c_1 \ln n}{n}(\hat{p}_{i,1} + \hat{q}_{i,1}) + \left(  \sum_{i = 1}^S \frac{\sqrt{c_1(\hat{p}_{i,1}+\hat{q}_{i,1})}}{\sqrt{c_2^2 n \ln n}} \right)^2,
\end{align}
where the constant $B$ is the one in Lemma~\ref{lemma.2dperformancepqlarge}. Taking expectations with respect to $ \{ (\hat{p}_{i,1},\hat{q}_{i,1}):1\leq i\leq S \}$, we have
\begin{align}
& \mathbb{E} [ \mathcal{E}_4^2 ] \nonumber \\
& \quad \lesssim \sum_{i = 1}^S \frac{c_1 \ln n}{n^{1-\epsilon}} \mathbb{E} ( \hat{p}_{i,1} + \hat{q}_{i,1} ) + \mathbb{E} \left(  \sum_{i = 1}^S \frac{\sqrt{c_1(\hat{p}_{i,1}+\hat{q}_{i,1})}}{\sqrt{c_2^2 n \ln n}} \right)^2 \\
& \quad \lesssim \frac{c_1 \ln n}{n^{1-\epsilon}} + \sum_{i = 1}^S \sum_{j = 1}^S \mathbb{E} \left( \sqrt{ \frac{ c_1(
\hat{p}_{i,1} + \hat{q}_{i,1})}{c_2^2 n\ln n} } \sqrt{ \frac{c_1(\hat{p}_{j,1} + \hat{q}_{j,1})}{c_2^2 n\ln n}
} \right)  \\
& \quad \lesssim \frac{c_1 \ln n}{n^{1-\epsilon}} + \sum_{i =1}^S \mathbb{E} \left(  \frac{c_1(\hat{p}_{i,1} + \hat{q}_{i,1})}{c_2^2 n\ln n} \right) \nonumber \\
& \qquad + \sum_{1\leq i,j\leq S, i\neq j} \sqrt{ \frac{\mathbb{E} [c_1(\hat{p}_{i,1} + \hat{q}_{i,1})]}{c_2^2 n\ln n} } \sqrt{ \frac{\mathbb{E} [c_1(\hat{p}_{j,1} + \hat{q}_{j,1})]}{c_2^2 n\ln n} } \\
& \quad \leq \frac{c_1 \ln n}{n^{1-\epsilon}} + \frac{c_1}{c_2^2 n\ln n} + \sum_{1\leq i,j\leq S} \frac{c_1(p_i + q_i + p_j + q_j)}{c_2^2 n\ln n} \\
& \quad \lesssim \frac{c_1 \ln n}{n^{1-\epsilon}} + \frac{c_1 S}{c_2^2 n\ln n},
\end{align}
where $\epsilon = c_2 \ln B$. 

Combining everything together, we have
\begin{align}
& \mathbb{E}\left( \hat{L}^{(2)} - \|P-Q\|_1 \right)^2 \nonumber \\
&  \quad \lesssim \frac{(c_1c_2^4 + c_1)\ln^5 n}{n^{1-\epsilon}} + \frac{c_1 S}{c_2^2 n\ln n} \vee \left( \frac{c_1 S}{c_2^2 n\ln n} \right)^2 + \frac{S}{n^\beta},
\end{align}
where $\epsilon = c_2 \ln B$, and the constant $B$ is the larger constant between the one in Lemma~\ref{lemma.2dperformancesquare} and Lemma~\ref{lemma.2dperformancepqlarge}. The constant $\beta$ is in~(\ref{eqn.2dbetadefinition}). 

If $\ln n \lesssim \ln S$, we can take $c_2$ small enough and $c_1,c_3$ large enough to guarantee that $\frac{S}{n^\beta} \lesssim \frac{S}{n\ln n}, \frac{\ln^5 n}{n^{1-\epsilon}} \lesssim \frac{S}{n\ln n}$. Upon noting that $\hat{L}^{(2)} \in [0,2]$, we have 
\begin{align}
\mathbb{E}\left( \hat{L}^{(2)} - \|P-Q\|_1 \right)^2 & \lesssim \frac{S}{n\ln n}. 
\end{align}
\end{proof}

\section{Proofs of main lemmas}\label{sec.proofofmainlemmas}

\subsection{Proof of Lemma~\ref{lemma.individualqcover}}

We first consider the case of $q \leq \frac{c_1 \ln n}{n}$. In this case, 
\begin{align}
\bP(\hat{q} \notin U(q;c_1)) & = \bP \left(\hat{q} > \frac{2c_1 \ln n}{n} \right ) \\
& = \bP(\spo(nq) > 2c_1 \ln n) \\
& \leq \bP(\spo(c_1 \ln n) > 2c_1\ln n) \\
& \leq e^{-\frac{c_1 \ln n}{3}} \\
& = n^{-\frac{c_1}{3}},
\end{align}
where we used Lemma~\ref{lemma.poissontail} in the last step. When $q>\frac{c_1\ln n}{n}$, we have
\begin{align}
& \bP(\hat{q} \notin U(q;c_1)) \nonumber \\
& \quad \leq \bP\left(\hat{q} > q + \sqrt{\frac{c_1 q \ln n}{n}} \right) + \bP\left( \hat{q} < q - \sqrt{\frac{c_1 q \ln n}{n}} \right) \\
& \quad = \bP( \spo(nq)> nq + \sqrt{c_1 q n\ln n} ) \nonumber \\
& \qquad + \bP(\spo(nq) < nq - \sqrt{c_1 qn \ln n}) \\
& \quad \leq e^{ - \frac{c_1 \ln n}{nq} \frac{nq}{3}} + e^{- \frac{c_1\ln n}{nq} \frac{nq}{2}} \\
& \quad \leq \frac{2}{n^{c_1 /3}},
\end{align}
where we applied Lemma~\ref{lemma.poissontail} again. 

\subsection{Proof of Lemma~\ref{lemma.goodeventswin}}

Since
\begin{align}
\bP(E^c) & = \bP(E_1^c \cup E_2^c \cup E_3^c) \\
& \leq \bP(E_1^c) + \bP(E_2^c) + \bP(E_3^c),
\end{align}
it suffices to control $\bP(E_i^c),i = 1,2,3$ separately. We have 
\begin{align}
\bP(E_1^c) & = \mathbb{P} \left( \bigcup_{i = 1}^S \left \{ \hat{p}_{i,1} > U_1(q_i), p_i < q_i \right \} \right) \\
& \leq S \bP \left( \hat{p}_{i,1} > U_1(q_i), p_i < q_i \right) \\
& \leq S \bP \left( \spo(nq_i) > n\cdot U_1(q_i) \right).
\end{align}
Note that if $q_i \leq \frac{c_1 \ln n}{n}$, then it follows from Lemma~\ref{lemma.poissontail} that
\begin{align}
\bP\left( \spo(nq_i) > n\cdot U_1(q_i) \right) & \leq \bP\left( \spo(c_1 \ln n) > (c_1 + c_3 )\ln n \right) \\
& \leq e^{-\frac{c_3^2}{3c_1}\ln n}.
\end{align}
If $q_i > \frac{c_1 \ln n}{n}$, then it follows from Lemma~\ref{lemma.poissontail} that
\begin{align}
\bP\left( \spo(nq_i) > n\cdot U_1(q_i) \right) & \leq \bP\left( \spo(nq_i) > nq_i + \sqrt{c_3 q n\ln n} \right) \\
& \leq e^{-\frac{c_3\ln n}{3}}. 
\end{align}
Hence, 
\begin{align}
\bP(E_1^c) & \leq \frac{S}{n^\beta}. 
\end{align}

Analogously, $\bP( \hat{p}_{i,1} < U_1(q_i), p_i > q_i)=0$ when $q_i \leq \frac{c_1 \ln n}{n}$, and when $q_i > \frac{c_1 \ln n}{n}$, 
\begin{align}
\bP( \hat{p}_{i,1} < U_1(q_i), p_i > q_i) & \leq \bP(\mathsf{Poi}(nq_i) \leq nq_i - \sqrt{c_3 q_i n \ln n}) \\
& \leq e^{-\frac{c_3 \ln n}{2}}.
\end{align}
Hence, 
\begin{align}
\bP(E_2^c) & \leq \frac{S}{n^\beta}. 
\end{align}

As for $\bP(E_3^c)$, when $q_i\leq \frac{c_1 \ln n}{n}$, 
\begin{align}
& \bP(\hat{p}_{i,1} \in U_1(q_i), p_i \notin U(q_i;c_1)) \nonumber \\
&  \quad\leq \bP(\spo(2 c_1 \ln n) \leq (c_1 + c_3) \ln n) \\
& \quad \leq e^{ - \frac{(c_1-c_3)^2}{4c_1} \ln n}.
\end{align}
When $q_i > \frac{c_1 \ln n}{n}$, 
\begin{align}
& \bP(\hat{p}_{i,1} \in U_1(q_i), p_i > U(q_i;c_1)) \nonumber \\
&\quad \leq \bP(\spo(nq_i + \sqrt{c_1 q_i n\ln n}) \leq nq_i + \sqrt{c_3 q_i n\ln n} ) \\
& \quad \leq e^{- \left( \frac{(\sqrt{c_1}-\sqrt{c_3})\sqrt{q_i n\ln n}}{nq_i + \sqrt{c_1 q_i n \ln n}} \right)^2 \frac{1}{2} (nq_i + \sqrt{c_1 q_i n\ln n})} \\
&\quad \leq e^{-\frac{(\sqrt{c_1}-\sqrt{c_3})^2 \ln n}{4}}. 
\end{align}
\begin{align}
& \bP(\hat{p}_{i,1} \in U_1(q_i), p_i < U(q_i;c_1)) \nonumber \\
&  \quad\leq 
\bP(\spo(nq_i - \sqrt{c_1 q_i n\ln n}) \geq nq_i - \sqrt{c_3 q_i n\ln n} ) \\
& \quad \leq e^{-\frac{(\sqrt{c_1}-\sqrt{c_3}) \sqrt{q_i n\ln n}}{3}} \vee e^{-\frac{(\sqrt{c_1}-\sqrt{c_3})^2 n q_i \ln n}{3(nq_i  -\sqrt{c_1 nq_i \ln n})}} \\
&  \quad \leq e^{-\frac{c_1 - \sqrt{c_1c_3}}{3} \ln n} \vee e^{-\frac{(\sqrt{c_1}-\sqrt{c_3})^2}{3} \ln n} \\
& \quad \leq e^{-\frac{(\sqrt{c_1}-\sqrt{c_3})^2}{3} \ln n}. 
\end{align}
Consequently,
\begin{align}
\bP(E_3^c) & \leq \frac{S}{n^\beta}. 
\end{align}

\subsection{Proof of Lemma~\ref{lemma.stripeshape}}

It is clear that the square $\left[ 0, \frac{2c_1 \ln n}{n} \right]^2 \subset U$. To see how we obtained the whole expression of $U$, for any $x > \frac{c_1 \ln n}{n}$, we study the envelope of the parametrized extremal points $\left( x - \sqrt{\frac{c_1 x\ln n }{n}}, x+\sqrt{\frac{c_1 x \ln n }{n}} \right) $, where the other curve $\left( x + \sqrt{\frac{c_1 x\ln n }{n}}, x-\sqrt{\frac{c_1 x \ln n }{n}} \right)$ can be dealt with analogously. 

For $p = x - \sqrt{\frac{c_1 x\ln n }{n}}, q = x + \sqrt{\frac{c_1 x\ln n }{n}}$, we have
\begin{align}
p-q & = -2 \sqrt{\frac{c_1 x\ln n }{n}} \\
p+q & = 2x. 
\end{align}
Hence, 
\begin{align}
(p-q)^2 = \frac{2c_1 \ln n}{n}(p+q). 
\end{align}
We have that for all points $(p,q) \in \cup_{x\in [0,1]} U(x;c_1)\times U(x;c_1)$, 
\begin{align}
|p-q| & \leq \sqrt{\frac{2c_1 \ln n}{n}} \sqrt{p+q} \\
& \leq \sqrt{\frac{2c_1 \ln n}{n}} (\sqrt{p} + \sqrt{q}),
\end{align}
where we used the inequality $\sqrt{p+q} \leq \sqrt{p} + \sqrt{q}$ in the last step. 

\subsection{Proof of Lemma~\ref{lemma.2dgoodeventswin}}

It follows from the union bound that
\begin{align}
\bP(E^c) & = \bP(E_1^c \cup E_2^c \cup E_3^c \cup E_4^c )  \\
& \leq \bP(E_1^c) + \bP(E_2^c) + \bP(E_3^c) + \bP(E_4^c).
\end{align}
Hence, it suffices to analyze each $\bP(E_i^c), i = 1,2,3,4$.  
\begin{enumerate}
\item Analysis of $\bP(E_1^c)$:
\begin{align}
& \bP(E_1^c) \nonumber \\
& = \bP \Bigg( \bigcup_{i = 1}^S \Bigg \{ p_i < q_i, \nonumber \\
& \qquad \qquad  \hat{p}_{i,1} - \hat{q}_{i,1} > \sqrt{\frac{(c_1 + c_3)\ln n}{n}}(\sqrt{\hat{p}_{i,1}} + \sqrt{\hat{q}_{i,1}}) \Bigg \}   \Bigg) \\
& \leq \sum_{i = 1}^S \bP \Bigg( p_i < q_i, \nonumber \\
& \qquad \qquad \hat{p}_{i,1} - \hat{q}_{i,1} > \sqrt{\frac{(c_1 + c_3)\ln n}{n}}(\sqrt{\hat{p}_{i,1}} + \sqrt{\hat{q}_{i,1}}) \Bigg) \\
& = \sum_{i = 1}^S \bP \left( p_i < q_i,  \sqrt{\hat{p}_{i,1}} -\sqrt{\hat{q}_{i,1}}  > \sqrt{\frac{(c_1 + c_3)\ln n}{n}} \right) \\
& \leq \sum_{i = 1}^S \bP \left( p_i = q_i,  \sqrt{\hat{p}_{i,1}} -\sqrt{\hat{q}_{i,1}}  > \sqrt{\frac{(c_1 + c_3)\ln n}{n}} \right). 
\end{align}
It follows from Lemma~\ref{lemma.stripeshape} that the set $U(p_i; (c_1 + c_3)/2) \times U(p_i; (c_1 + c_3)/2) \subset U_1$. Hence, 
\begin{align}
& \bP(E_1^c) \nonumber \\
& \leq \sum_{i = 1}^S \bP( p_i = q_i, \nonumber \\
& \qquad (\hat{p}_{i,1}, \hat{q}_{i,1}) \notin U(p_i; (c_1 + c_3)/2) \times U(p_i; (c_1 + c_3)/2)) \\
& \leq \sum_{i = 1}^S \Bigg( 1- \nonumber \\
& \qquad \bP ( p_i = q_i, \nonumber \\
& \qquad \qquad (\hat{p}_{i,1}, \hat{q}_{i,1}) \in U(p_i; (c_1 + c_3)/2) \times U(p_i; (c_1 + c_3)/2)) \Bigg) \\
& = \sum_{i = 1}^S \Bigg( 1 - \bP(\hat{p}_{i,1} \in U(p_i; (c_1 + c_3)/2) \times \nonumber \\
& \qquad \qquad \bP(q_i = p_i, \hat{q}_{i,1} \in U(p_i; (c_1 + c_3)/2)) \Bigg).
\end{align}
It follows from Lemma~\ref{lemma.individualqcover} that 
\begin{align}
\bP(E_1^c) & \leq S \left( 1 - \left( 1 -\frac{2}{n^{\frac{c_1 + c_3}{6}}} \right)^2 \right) \\
& \leq \frac{4S}{n^{\frac{c_1 + c_3}{6}}}. 
\end{align}
\item Analysis of $\bP(E_2^c)$: following similar steps as in the analysis of $\bP(E_1^c)$, we have $\bP(E_2^c) \leq \frac{4S}{n^{\frac{c_1 + c_3}{6}}}$. 
\item Analysis of $\bP(E_3^c)$:
\begin{align}
& \bP(E_3^c)\nonumber \\
& = \bP \Bigg( \bigcup_{i = 1}^S \Bigg \{ (p_i,q_i) \notin \left[ 0, \frac{2c_1 \ln n}{n} \right]^2, \nonumber \\
& \qquad \qquad \hat{p}_{i,1} + \hat{q}_{i,1} < \frac{c_1 \ln n}{n} \Bigg \} \Bigg) \\
& \leq \bP \left( \bigcup_{i = 1}^S \left \{ p_i + q_i > \frac{2c_1 \ln n}{n}, \hat{p}_{i,1} + \hat{q}_{i,1} < \frac{c_1 \ln n}{n} \right \} \right) \\
& \leq \sum_{i = 1}^S \bP \left( p_i + q_i > \frac{2c_1 \ln n}{n}, \hat{p}_{i,1} + \hat{q}_{i,1} < \frac{c_1 \ln n}{n} \right) \\
& \leq S \bP\left( \spo(2c_1 \ln n) < c_1 \ln n \right) \\
& \leq S e^{- \left( \frac{1}{2} \right)^2 \frac{2c_1 \ln n}{2}} \\
& = \frac{S}{n^{c_1/4}},
\end{align}
where we have used the fact that $n\hat{p}_{i,1} + n\hat{q}_{i,1} \sim \spo(np + nq)$ and Lemma~\ref{lemma.poissontail}. 
\item Analysis of $\bP(E_4^c)$: 
\begin{align}
& \bP(E_4^c) \nonumber \\
& \quad \leq \sum_{i = 1}^S \bP( (p_i,q_i) \notin U, (\hat{p}_{i,1}, \hat{q}_{i,1}) \in U_1 ) \nonumber \\
& \qquad + \sum_{i = 1}^S \bP \left (\hat{p}_{i,1} + \hat{q}_{i,1} > \frac{c_1 \ln n}{n}, p_i + q_i < \frac{c_1 \ln n}{2n} \right ) \nonumber \\
& \quad + \sum_{i = 1}^S \bP \left( \hat{p}_{i,1} + \hat{q}_{i,1} \geq \frac{c_1 \ln n}{n}, \hat{p}_{i,1} + \hat{q}_{i,1} \leq \frac{p_i + q_i}{2} \right).
\end{align}
%
%
We have
\begin{align}
& \quad \sum_{i = 1}^S \bP \left (\hat{p}_{i,1} + \hat{q}_{i,1} > \frac{c_1 \ln n}{n}, p_i + q_i < \frac{c_1 \ln n}{2n} \right ) \nonumber \\
& \leq S \bP \left ( \spo \left( \frac{c_1 \ln n}{2} \right ) > c_1 \ln n  \right ) \\
& \leq \frac{S}{n^{c_1/6}}
\end{align}
and 
\begin{align}
& \quad \sum_{i = 1}^S \bP \left( \hat{p}_{i,1} + \hat{q}_{i,1} \geq \frac{c_1 \ln n}{n}, \hat{p}_{i,1} + \hat{q}_{i,1} \leq \frac{p_i + q_i}{2} \right) \nonumber \\
& \leq \sum_{i = 1}^S \bP \left( p_i + q_i \geq \frac{2c_1 \ln n}{n}, \hat{p}_{i,1} + \hat{q}_{i,1} \leq \frac{p_i + q_i}{2} \right)  \\
& \leq \sum_{i = 1}^S \bP \Bigg( \spo(np_i + nq_i) \leq \frac{n(p_i+q_i)}{2}, \nonumber \\
& \qquad \qquad n(p_i + q_i) \geq 2 c_1 \ln n \Bigg) \\
& \leq S e^{- \frac{1}{4} \frac{2c_1 \ln n}{2}} \\
&  \leq \frac{S}{n^{c_1 / 4}}. 
\end{align}

It suffices to show that there exists some constant $c>0$ such that 
\begin{align}\label{eqn.confidence2drequirement}
\left( \bigcup_{(p,q) \notin U} U(p;c) \times U(q;c) \right) \bigcap U_1 = \emptyset,
\end{align} 
where $U(\cdot;c)$ is defined in~(\ref{eq.uncertain_set}). Indeed, in this case it follows from Lemma~\ref{lemma.individualqcover} that
\begin{align}
& \sum_{i = 1}^S \bP( (p_i,q_i) \notin U, (\hat{p}_{i,1}, \hat{q}_{i,1}) \in U_1 ) \nonumber \\
 & \leq \sum_{i = 1}^S \bP( (\hat{p}_{i,1}, \hat{q}_{i,1}) \notin U(p_i;c) \times U(q_i;c)) \\
& \leq \sum_{i = 1}^S \left( 1 - \bP((\hat{p}_{i,1}, \hat{q}_{i,1}) \in U(p_i;c) \times U(q_i;c)) \right) \\
& = \sum_{i = 1}^S \left( 1 - \bP(\hat{p}_{i,1} \in U(p_i;c)) \bP(\hat{q}_{i,1} \in U(q_i;c)) \right) \\ 
& \leq \sum_{i = 1}^S \left( 1 - \left( 1- \frac{2}{n^{c/3}} \right)^2 \right) \\
& \leq \frac{4S}{n^{c/3}}. 
\end{align}
Now we work to prove~(\ref{eqn.confidence2drequirement}). Without loss of generality we assume $(p,q)$ satisfies $\sqrt{q} -\sqrt{p} \geq \sqrt{\frac{2c_1 \ln n}{n}}$ and the constant $c < c_1$. Under this assumption we have $q \geq \frac{2c_1 \ln n}{n}$. We will show that for any point $(x,y) \in U(p;c) \times U(q;c)$, we have $\sqrt{y} - \sqrt{x} \geq \sqrt{\frac{(c_1 + c_3) \ln n}{n}}$, thereby proving~(\ref{eqn.confidence2drequirement}). 

If $p \leq \frac{c \ln n}{n}$, we have for any $(x,y) \in U(p;c) \times U(q;c)$,
\begin{align}
\sqrt{y} - \sqrt{x} & \geq \sqrt{q - \sqrt{\frac{cq \ln n}{n}}} - \sqrt{\frac{2c \ln n}{n}} \\
& \geq \sqrt{ \frac{2c_1 \ln n}{n} - \sqrt{2c c_1} \frac{\ln n}{n}} - \sqrt{\frac{2c \ln n}{n}},
\end{align}
where in the second step we used the fact that the function $x - \sqrt{ax},a>0$ is monotonically increasing when $x \geq a/4$. Hence, we need to guarantee that
\begin{align}
\sqrt{y} - \sqrt{x} & \geq \sqrt{\frac{\ln n}{n}} \left( \sqrt{2c_1 - \sqrt{2c c_1}} - \sqrt{2c} \right) \\
& \geq \sqrt{\frac{\ln n}{n}} \sqrt{c_1 + c_3},
\end{align} 
which can be reduced to the quadratic inequality:
\begin{align}
\left( \sqrt{\frac{2c}{c_1}} \right)^2 + \left( 1 + 2\sqrt{1 + \frac{c_3}{c_1}} \right)\sqrt{\frac{2c}{c_1}} + \frac{c_3}{c_1} - 1 \leq 0. 
\end{align}
One can easily verify that $c =\frac{(c_1 - c_3)^2}{32 c_1}$ satisfies this inequality since $0<c_3<c_1$. 

Now we consider the case of $p > \frac{c \ln n}{n}$. Then, for any $(x,y) \in U(p;c) \times U(q;c)$, 
\begin{align}
& \sqrt{y} - \sqrt{x} \nonumber \\
& \geq \sqrt{ q - \sqrt{\frac{cq \ln n}{n}}} - \sqrt{p + \sqrt{\frac{cp \ln n}{n}}} \\
& = \frac{q - \sqrt{\frac{cq \ln n}{n}} - p - \sqrt{\frac{cp \ln n}{n}}}{\sqrt{ q - \sqrt{\frac{cq \ln n}{n}}} +\sqrt{p + \sqrt{\frac{cp \ln n}{n}}}} \\
& = \frac{(\sqrt{q} - \sqrt{p})(\sqrt{q} + \sqrt{p}) - \sqrt{\frac{c \ln n}{n}} (\sqrt{p} + \sqrt{q}) }{\sqrt{ q - \sqrt{\frac{cq \ln n}{n}}} +\sqrt{p + \sqrt{\frac{cp \ln n}{n}}}} \\
& \geq  (\sqrt{2c_1}-\sqrt{c})  \sqrt{\frac{\ln n}{n}}  \frac{\sqrt{q} + \sqrt{p}}{\sqrt{ q - \sqrt{\frac{cq \ln n}{n}}} +\sqrt{p + \sqrt{\frac{cp \ln n}{n}}}}.
\end{align}
Further, since $p > \frac{c\ln n}{n}$, 
\begin{align}
& \frac{\sqrt{q} + \sqrt{p}}{\sqrt{ q - \sqrt{\frac{cq \ln n}{n}}} +\sqrt{p + \sqrt{\frac{cp \ln n}{n}}}} \nonumber \\ & \geq \frac{\sqrt{q} + \sqrt{p}}{\sqrt{q} + \sqrt{2p}} \\
& \geq \frac{\sqrt{p} + \sqrt{\frac{2c_1 \ln n}{n}} + \sqrt{p}}{\sqrt{2p} + \sqrt{p} + \sqrt{\frac{2c_1 \ln n}{n}}} \\
& \geq \frac{2}{\sqrt{2}+1},
\end{align}
where we used the fact that $\frac{x+\sqrt{p}}{x+\sqrt{2p}}$ is a monotonically increasing function of $x$ when $x\geq 0$, and the function $\frac{2x + a}{(\sqrt{2}+1)x + a}$ is a monotonically decreasing function of $x$ when $a>0, x>0$. To guarantee that $\sqrt{y} - \sqrt{x}\geq \sqrt{\frac{(c_1 + c_3)\ln n}{n}}$, we need
\begin{align}
\frac{2}{\sqrt{2}+1} \left( \sqrt{2c_1} - \sqrt{c} \right) \geq \sqrt{c_1 + c_3}, 
\end{align}
which is equivalent to 
\begin{align}
c \leq \left( \sqrt{2c_1} - \frac{\sqrt{2}+1}{2}\sqrt{c_1 + c_3} \right)^2,
\end{align}
with the constraint that $\frac{c_3}{c_1} < \frac{8}{(\sqrt{2}+1)^2} -1 \approx 0.373$. 
\end{enumerate}

\subsection{Proof of Lemma~\ref{lemma.zeroapproximationinsufficient}}

We consider two different parameter settings. 
\begin{enumerate}
\item $S \ll n \lesssim S\ln S$: In this case, we construct the distribution $P$ as \footnote{Technically, the distribution $P$ has support no more than $S$. However, a standard continuity argument implies that the same conclusion holds.}
\begin{align}
P = \left( \frac{c \ln n}{n},\frac{c\ln n}{n},\ldots,\frac{c \ln n}{n}, 0,\ldots,0 \right),
\end{align}
where $c>2c_1$ is a constant that will be chosen later, and $Q = P$. Without loss of generality we assume $\frac{n}{c \ln n}$ is an integer. We now argue that for each index $1\leq i\leq \frac{n}{c \ln n}$, 
\begin{align}\label{eqn.outsidebiaslarge}
\left |\bE g(\hat{p}_i,\hat{q}_i)  - |p_i - q_i|\right | \gtrsim \frac{\sqrt{\ln n}}{n}. 
\end{align}

It follows from Lemma~\ref{lemma.poissontail} that $\bP\left( \hat{p}_i \leq \frac{2c_1\ln n}{n} \right) \leq e^{-\frac{1}{2} (1-2c_1/c)^2 c \ln n} = n^{-\beta}$, where $\beta = \frac{c}{2}\left( 1- \frac{2c_1}{c}\right)^2$. Note that $\beta$ can be made arbitrarily large by taking the constant $c$ large. Define $E = \left \{ \hat{p}_i \geq \frac{2c_1\ln n}{n}, \hat{q}_i \geq \frac{2c_1 \ln n}{n} \right \}$. We have
\begin{align}
\bE g(\hat{p}_i, \hat{q}_i) & = \bE \left( g \mathbbm{1}(E) + g \mathbbm{1}(E^c) \right) \\
& = \bE \left( |\hat{p}_i - \hat{q}_i |\mathbbm{1}(E) \right) + \bE \left( g \mathbbm{1}(E^c) \right) \\
& = \bE |\hat{p}_i - \hat{q}_i| - \bE \left( |\hat{p}_i - \hat{q}_i| \mathbbm{1}(E^c) \right) + \bE \left( g \mathbbm{1}(E^c) \right) \\ 
& = \bE |\hat{p}_i - \hat{q}_i| + \bE \left( \left( g - |\hat{p}_i - \hat{q}_i|  \right) \mathbbm{1}(E^c) \right).
\end{align}
Since $|g|\leq B$, we have
\begin{align}
\left | \bE \left( \left( g - |\hat{p}_i - \hat{q}_i|  \right) \mathbbm{1}(E^c) \right) \right | \leq (B+1) \frac{2}{n^\beta}. 
\end{align}
It follows from the triangle inequality that 
\begin{align}
\left | \bE g(\hat{p}_i, \hat{q}_i) \right | \geq  \bE |\hat{p}_i - \hat{q}_i| - \frac{2(B+1)}{n^\beta} 
\end{align}
It follows from the conditional version of Jensen's inequality that $\mathbb{E}|\hat{p}_i -\hat{q}_i| \geq \bE |\hat{p}_i - p_i|$, and by Lemma~\ref{lemma.poissonmlebias} we have 
\begin{align}
\mathbb{E}|\hat{p}_i -\hat{q}_i| & \geq \sqrt{\frac{p_i}{2n}} \\
& = \sqrt{\frac{c\ln n}{2n^2}} \\
& \gtrsim \frac{\sqrt{\ln n}}{n}. 
\end{align}
Since $\frac{\sqrt{\ln n}}{n} \gg \frac{2(B+1)}{n^\beta}$ for $\beta>1$, we conclude that~(\ref{eqn.outsidebiaslarge}) is true. Hence, the total bias of $\hat{L}$ is at least $\left(  \frac{n}{c \ln n} \frac{\sqrt{\ln n}}{n} \right)^2 =  \frac{1}{\ln n} \gg \frac{S}{n\ln n}$ since $S \ll n$. 
\item $n\gg S\ln S$: In this case, we construct $P,Q$ to be uniform distributions with support size $S$. Since $\frac{1}{S} \gg \frac{\ln n}{n}$, it follows from arguments analogous to those above that the squared bias of $\hat{L}$ is at least the order $\left( S \sqrt{\frac{1}{2Sn}} \right)^2 = \frac{S}{2n} \gg \frac{S}{n\ln n}$. 
\end{enumerate}

\subsection{Proof of Lemma~\ref{lemma.pointwise}}

Since $\left | |a| - |b| \right | \leq |a-b|$, it suffices to show that there exists a universal constant $C>0$ such that
\begin{align}
|Q_K(t)| \leq C K t^2
\end{align}
for $|t|\leq 1$. Define $\sqrt{x} = |t|$. Since $Q_K(t)$ is even, it follows that $Q_K(t) = R(t^2)$, where $R \in \poly_{K}$ is a polynomial. The polynomial $R$ satisfies the following:
\begin{align}
R(0) & = 0 \\
\max_{x\in [0,1]} |R(x) - \sqrt{x}| & \lesssim \frac{1}{K}. 
\end{align}

It suffices to show that $|R(x)| \leq C K x$. Let $T(x) \in \poly_K$ denote the best approximation polynomial of the function $\sqrt{x}$ on $[0,1]$ with order no more than $K$. It follows from Lemma~\ref{lemma.ditziantotikgeneral} and Lemma~\ref{lemma.dtcomputationchap34} that $\sup_{x\in [0,1]} |T(x) - \sqrt{x}| \lesssim \frac{1}{K}$. It follows from the triangle inequality that
\begin{align}
\sup_{x\in [0,1]} |R(x) - T(x)| & \leq \sup_{x\in [0,1]} \left( |R(x) - \sqrt{x}| + |T(x) - \sqrt{x}| \right) \\
& \lesssim \frac{1}{K}. 
\end{align}

It follows from the Markov inequality~(Lemma~\ref{lemma.markovinequality}) that $\sup_{x\in [0,1]} |R'(x) - T'(x)| \lesssim K$. Since for any $0\leq x\leq 1$, 
\begin{align}
\left | R(x) \right | & = \left | \int_0^x R'(u)du \right | \\
& = \left | \int_0^x (R'(u) - T'(u))du + \int_0^x T'(u)du \right | \\
& \leq x \sup_{x\in [0,1]} |R'(x) - T'(x)| + \left | \int_0^x T'(u)du \right | \\
& \lesssim K x + \left | \int_0^x T'(u)du \right |,
\end{align}
it suffices to show $\left | \int_0^x T'(u)du \right | \lesssim Kx$. 

It follows from Lemma~\ref{lemma.derivapproximationbound} and Lemma~\ref{lemma.dtcomputationchap34} that 
\begin{align}
\sup_{x\in [0,1]} |\sqrt{x(1-x)} T'(x)| \lesssim 1. 
\end{align}
Hence, it follows from Lemma~\ref{lemma.nsquareshrink} that
\begin{align}
\sup_{x\in [0,1]}|T'(x)| & \lesssim \sup_{x\in [1/K^2, 1-1/K^2]} |T'(x)| \\
& \lesssim \sup_{x\in [1/K^2, 1-1/K^2]} \frac{1}{\sqrt{x(1-x)}} \\
& \lesssim K. 
\end{align}

Hence,
\begin{align}
\left | \int_0^x T'(u)du \right | & \lesssim K x. 
\end{align}
The proof is complete.

\subsection{Proof of Lemma~\ref{lemma.stripelowerbound}}
We prove the lemma by contradiction. Assuming the contrary, then there exist universal constants $c,C>0$ and polynomial $P(x,y) \in \poly_K^2$ of degree $K=c\ln n$ such that
\begin{align}
\sup_{(x,y)\in U'}\frac{|P(x,y)-|x-y||}{\sqrt{x+y}} \le \frac{C}{\sqrt{n\ln n}}
\end{align}
where $U'=\cup_{x\in [\frac{c_1\ln n}{n},t_n]}U(x;c_1)\cup U(x;c_1)$. Now for any $t\in [\frac{c_1\ln n}{n},t_n]$, we have $(t-\frac{1}{2}\sqrt{\frac{c_1t\ln n}{n}},t+\frac{1}{2}\sqrt{\frac{c_1t\ln n}{n}})\in U'$, and plugging in this pair yields
\begin{align}
& \sup_{t\in [\frac{c_1\ln n}{n},t_n]}\frac{|P(t-\frac{1}{2}\sqrt{\frac{c_1t\ln n}{n}},t+\frac{1}{2}\sqrt{\frac{c_1t\ln n}{n}})-\sqrt{\frac{c_1t\ln n}{n}}|}{\sqrt{2t}} \nonumber \\
& \quad \le \frac{C}{\sqrt{n\ln n}}. 
\end{align}
Similarly, for $(t+\frac{1}{2}\sqrt{\frac{c_1t\ln n}{n}},t-\frac{1}{2}\sqrt{\frac{c_1t\ln n}{n}})\in U'$ we also have
\begin{align}
& \sup_{t\in [\frac{c_1\ln n}{n},t_n]}\frac{|P(t+\frac{1}{2}\sqrt{\frac{c_1t\ln n}{n}},t-\frac{1}{2}\sqrt{\frac{c_1t\ln n}{n}})-\sqrt{\frac{c_1t\ln n}{n}}|}{\sqrt{2t}} \nonumber \\
& \quad \le \frac{C}{\sqrt{n\ln n}}.
\end{align}

Now consider
\begin{align}
Q(t) & = \frac{1}{2}\sqrt{\frac{n}{c_1\ln n}}\Bigg(P\Bigg (t-\frac{1}{2}\sqrt{\frac{c_1t\ln n}{n}},t+\frac{1}{2}\sqrt{\frac{c_1t\ln n}{n}}\Bigg ) \nonumber \\
& \quad + P\Bigg (t+\frac{1}{2}\sqrt{\frac{c_1t\ln n}{n}},t-\frac{1}{2}\sqrt{\frac{c_1t\ln n}{n}}\Bigg )\Bigg)
\end{align}
it is easy to see that $Q(t)$ is a polynomial of $t$, and $\deg Q\le 2K$. Moreover, adding the previous two inequalities together, by triangle inequality we obtain
\begin{align}\label{eq.poly_constraint_Q}
\sup_{t\in [\frac{c_1\ln n}{n},t_n]}\frac{|Q(t)-\sqrt{t}|}{\sqrt{t}} \le \sqrt{\frac{2}{c_1}}\cdot\frac{C}{\ln n}\lesssim \frac{1}{K}.
\end{align}

Since $t_n\gg \frac{(\ln n)^3}{n}$, we have $\eta_n\triangleq \frac{c_1\ln n}{nt_n}\ll \frac{1}{K^2}$. Define $R(t)=t_n^{-\frac{1}{2}}Q(t_n\cdot t)$ for $t\in [\eta_n,1]$, \eqref{eq.poly_constraint_Q} becomes
\begin{align}\label{eq.poly_constraint_R}
|R(t) - \sqrt{t}| \lesssim \frac{\sqrt{t}}{K}, \qquad \forall t\in [\eta_n,1].
\end{align}
Moreover, $\deg R\le 2K$. Now let $S$ be the best degree-$2K$ approximating polynomial of $\sqrt{t}$ in the uniform norm on $[\eta_n,1]$, using second-order Ditzian--Totik modulus of smoothness (Lemma \ref{lemma.ditziantotikgeneral}) and $\eta_n\ll \frac{1}{K^2}$ we arrive at
\begin{align}
\sup_{t\in [\eta_n,1]} |S(t) - \sqrt{t}| \lesssim \frac{1}{K}.
\end{align}
Furthermore, following the proof of Lemma \ref{lemma.pointwise} we can prove that
\begin{align}
\sup_{t\in [\eta_n,1]} |S'(t)| \lesssim K.
\end{align}

As a result, by triangle inequality we have
\begin{align}
& \sup_{t\in [\eta_n,1]}|R(t)-S(t)| \nonumber \\
& \quad\le \sup_{t\in [\eta_n,1]}|R(t)-\sqrt{t}| + \sup_{t\in [\eta_n,1]}|S(t)-\sqrt{t}| \\
&\quad \lesssim \frac{1}{K} + \frac{1}{K} \lesssim \frac{1}{K}.
\end{align}
Since $R(t)-S(t)$ is also a polynomial of degree $\le 2K$, by Markov's inequality (Lemma \ref{lemma.markovinequality})
\begin{align}
\sup_{t\in [\eta_n,1]}|R'(t)-S'(t)| & \le \frac{2}{1-\eta_n}\cdot 4K^2\sup_{t\in [\eta_n,1]}|R(t)-S(t)| \\
& \lesssim K
\end{align}
and finally by triangle inequality again
\begin{align}\label{eq.poly_derivative_bound}
\sup_{t\in [\eta_n,1]}|R'(t)| \le \sup_{t\in [\eta_n,1]}|R'(t)-S'(t)| + \sup_{t\in [\eta_n,1]} |S'(t)| \lesssim K.
\end{align}

Now we are about to arrive at the desired contradiction. Choosing $t=\eta_n$ and $t=2\eta_n$ in \eqref{eq.poly_constraint_R}, we have
\begin{align}
\left(1-\frac{D}{K}\right)\sqrt{\eta_n} &\le R(\eta_n) \le \left(1+\frac{D}{K}\right)\sqrt{\eta_n}, \\
\left(1-\frac{D}{K}\right)\sqrt{2\eta_n} &\le R(2\eta_n) \le \left(1+\frac{D}{K}\right)\sqrt{2\eta_n}
\end{align}
with $D>0$ a suitable universal constant appearing in the RHS of \eqref{eq.poly_constraint_R}. As a result, 
\begin{align}
R(2\eta_n) - R(\eta_n) & \ge \left(1-\frac{D}{K}\right)\sqrt{2\eta_n} - \left(1+\frac{D}{K}\right)\sqrt{\eta_n}  \\
& \gtrsim \sqrt{\eta_n}
\end{align}
and by the mean value theorem we conclude that there exists some $\xi \in [t_n,2t_n]$ such that
\begin{align}
R'(\xi) = \frac{R(2\eta_n) - R(\eta_n)}{2\eta_n-\eta_n} \gtrsim \frac{1}{\sqrt{\eta_n}}\gg K
\end{align}
where the last inequality follows from the fact that $\eta_n\ll \frac{1}{K^2}$. However, this inequality is contradicting to our previous result \eqref{eq.poly_derivative_bound}, and thus we are done. \qed

\subsection{Proof of Lemma~\ref{lemma.separateplugin}}
We have
\begin{align}
& \sup_{P,Q\in \mathcal{M}_S} \bE_P \left( \|g(P_n)-Q\|_1 - \|P-Q\|_1 \right)^2 \nonumber \\
& \quad\geq \sup_{P\in \mathcal{M}_S} \bE_P \left( \|g(P_n)-P\|_1\right)^2 \\
& \quad\geq   \left( \sup_{P\in \mathcal{M}_S} \mathbb{E}_P \|g(P_n)-P\|_1 \right)^2 \\
& \quad \gtrsim \frac{S}{n},
\end{align}
where the last step follows from the result of minimax risk for estimating the discrete distribution $P$ under $\ell_1$ loss in~\cite[Cor. 4]{han--jiao--weissman2015minimax}.

\subsection{Proof of Lemma~\ref{lem.opt_1}}
To simplify the notation we denote $\Delta = \frac{c_1 \ln n}{n}$. We split the proof into two cases: $q\leq \Delta$ and $q>\Delta$. 
\begin{enumerate}
\item The case $q\leq \Delta, p\in U(q;c_1) = [0,2\Delta]$. In this case, it follows from~(\ref{eqn.knownqpoly}) that $P_K(x;q)$ is the best approximation polynomial of function $|x-q|$ over $x\in [0,2\Delta]$. Define $y = \frac{x}{2\Delta}$ and introduce function 
\begin{align}
g(y) & = |2y\Delta - q|, y\in [0,1].
\end{align}
Define the best approximation polynomial of $g(y)\in C[0,1]$ with order $K$ as
\begin{align}
H_K(y) & = \argmin_{P\in \poly_K} \max_{y \in [0,1]} |g(y) - P(y)|.
\end{align}
It follows from Lemma~\ref{lemma.ditziantotikgeneral} and~\ref{lemma.dtmodulusknownq} that there exists a universal constant $M>0$ such that
\begin{align}
\sup_{y\in [0,1]}|H_K(y) - g(y)| \leq M \left( q \wedge \frac{\sqrt{q\Delta}}{K} \right).  
\end{align}

Since the approximation performance of $H_K(y)$ is at least as good as that of a constant, and $\max_{y\in [0,1]}|g(y)| \lesssim \Delta$, we know that there exists another universal constant $M_1 >0$ such that
\begin{align}
\sup_{y\in [0,1]}|H_K(y)| & \leq M_1 \Delta. 
\end{align}

Denoting $H_K(y) = \sum_{j = 0}^K a_j y^j$ and using $x = 2\Delta y$, we know
\begin{align}
P_K(x;q) & = \sum_{j = 0}^K a_j (2\Delta)^{-j} x^j \\
\sup_{x\in [0,2\Delta]} \left | P_K(x;q) - |x-q| \right | & \leq M \left( q \wedge \frac{\sqrt{q\Delta}}{K} \right)  \\
& \lesssim q \wedge \frac{1}{K} \sqrt{\frac{q c_1 \ln n}{n}}. 
\end{align}

It follows from Lemma~\ref{lemma.middlevariancebound} that $\prod_{k = 0}^{j-1} \left( \hat{p} - \frac{k}{n}\right)$ is the unique uniformly minimum variance unbiased estimator of $p^j$ when $n\hat{p} \sim \mathsf{Poi}(np)$. Hence,
\begin{align}
\tilde{P}_K(\hat{p};q) & = \sum_{j = 0}^K a_{j} (2\Delta)^{-j} \prod_{k = 0}^{j-1} \left( \hat{p} - \frac{k}{n}\right),
\end{align}
and $\mathbb{E} \tilde{P}_K(\hat{p};q) = P_K(p;q)$. Since $H_K(z^2) = \sum_{j = 0}^{K} a_j z^{2j}$ is a polynomial with degree no more than $2K$ and satisfies
\begin{align}
\sup_{z\in [-1,1]} |H_K(z^2)| \leq M_1 \Delta, 
\end{align}
It follows from Lemma~\ref{lem.polycoeff} that for all $0\leq j\leq K$, 
\begin{align}
|a_j| \leq M_1 \Delta \left(\sqrt{2}+1 \right)^{2K}. 
\end{align}
Now we prove the variance properties of $\tilde{P}_K(\hat{p};q)$. We have
\begin{align}
& \mathsf{Var}(\tilde{P}_K(\hat{p};q)) \nonumber \\
& = \mathsf{Var}\left( \sum_{j = 0}^K a_{j} (2\Delta)^{-j} \prod_{k = 0}^{j-1} \left( \hat{p} - \frac{k}{n}\right) \right) \\
 & \leq \left( \sum_{j = 0}^K   |a_{j}| (2\Delta)^{-j}  \left (\mathsf{Var}\left( \prod_{k = 0}^{j-1} \left( \hat{p} - \frac{k}{n}\right) \right)\right )^{1/2}  \right)^2 \\
& \leq \max_{0\leq j\leq K}|a_j|^2  \left( \sum_{j = 1}^K (2\Delta)^{-j} \left( \frac{2M p}{n} \right)^{j/2}   \right)^2 \\
& = \max_{0\leq j\leq K}|a_j|^2  \left( \sum_{j = 1}^K \left( \frac{2M p}{4\Delta^2 n} \right)^{j/2}  \right)^2 \\
& \leq \max_{0\leq j\leq K}|a_j|^2  \left( \sum_{j = 1}^K \left( \frac{p}{\Delta} \right)^{j/2}  \right)^2 \\
& = \max_{0\leq j\leq K}|a_j|^2  \frac{p}{\Delta} \left( \sum_{j = 0}^{K-1} \left( \sqrt{\frac{p}{\Delta}} \right)^{j}\right)^2,  
\end{align}
where we have applied Lemma~\ref{lemma.middlevariancebound} with $M = \max\{2n\Delta, K\} = 2n\Delta$ since we have assumed $c_2 <c_1, K = c_2 \ln n$. 

Since $p\leq 2\Delta$, we have
\begin{align}
\mathsf{Var}(\tilde{P}_K(\hat{p};q)) & \leq  \max_{0\leq j\leq K}|a_j|^2  \frac{p}{\Delta}   \left(\sum_{j = 0}^{K-1} 2^{j/2} \right)^2 \\
& \lesssim \Delta^2 \frac{p}{\Delta} B^K \\
& \lesssim B^K \frac{c_1 \ln n}{n} (p+q),
\end{align}
where $B>0$ is some universal constant. 

\item The case $q>\Delta$. In this case, it follows from~(\ref{eqn.knownqpoly}) that $P_K(x;q)$ is the best approximation polynomial of function $|x-q|$ over $x\in [q-\sqrt{q\Delta}, q+ \sqrt{q\Delta}]$. Denote the best approximation polynomial of $|y|$ on $[-1,1]$ with order $K$ as 
\begin{align}
R_K(y) = \sum_{j = 0}^K r_j y^j.
\end{align}
Using $x = q + y \sqrt{q\Delta}$, we have
\begin{align}
P_K(x;q) & = \sum_{j = 0}^K r_j (\sqrt{q \Delta})^{-j+1} (x-q)^j.
\end{align}

It is well known that~\cite[Chap. 9, Thm. 3.3]{Devore--Lorentz1993} there exists a universal constant $M_3$ such that
\begin{align}
|R_K(y) - |y|| \leq \frac{M_3}{K}, \forall y \in [-1,1].
\end{align}
Consequently, for $p\in [q-\sqrt{q\Delta}, q+\sqrt{q\Delta}]$,  
\begin{align}
\left | P_K(p;q) - |p-q| \right | \leq \frac{M_3 \sqrt{q\Delta}}{K} \lesssim \frac{1}{K}\sqrt{\frac{q c_1\ln n}{n}}. 
\end{align}

It follows from Lemma~\ref{lemma.middlevariancebound} that $g_{j,q}(\hat{p}), n\hat{p} \sim \mathsf{Poi}(np)$ defined as
\begin{align}
g_{j,q}(\hat{p}) \triangleq \sum_{k = 0}^j {j \choose k} (-q)^{j-k} \prod_{h = 0}^{k-1} \left( \hat{p} - \frac{h}{n} \right)  
\end{align}
is the unique uniformly minimum variance unbiased estimator for $(p-q)^j,j\geq 0,j\in \mathbb{N}$. Hence, 
\begin{align}
\tilde{P}_K(\hat{p};q) & = \sum_{j = 0}^K r_j (\sqrt{q \Delta})^{-j+1} g_{j,q}(\hat{p}). 
\end{align}

It was shown in Cai and Low~\cite[Lemma 2]{Cai--Low2011} that $|r_j| \leq 2^{3K}, 0\leq j \leq K$. We study the variance properties of $\tilde{P}_K(\hat{p};q)$ as follows. 

Define $M_4 \triangleq \max\{K, \frac{n(p - q)^2}{p}\}$. Note that if $p = 0$ the variance of this $\tilde{P}_K(\hat{p};q)$ is zero. We now consider $p \neq 0$. Applying Lemma~\ref{lemma.middlevariancebound} and the fact that the standard deviation of a sum of random variables is upper bounded by the sum of standard deviations of corresponding random variables, we have
\begin{align}
& \mathsf{Var}(\tilde{P}_K(\hat{p};q)) \nonumber \\
& = \mathsf{Var} \left( \sum_{j = 0}^K r_j (\sqrt{q \Delta})^{-j+1} g_{j,q}(\hat{p}) \right) \\
  & \leq \left( \sum_{j = 0}^K |r_j| (\sqrt{q \Delta})^{-j+1} \mathsf{Var}^{1/2}(g_{j,q}(\hat{p})) \right)^2 \\
 & \leq 2^{6K} q \Delta \left( \sum_{j = 0}^K (\sqrt{q \Delta})^{-j}
\left( \frac{2M_4 p}{n} \right)^{j/2}   \right)^2 \\
& \leq  2^{6K} q \Delta \left( \sum_{j = 0}^K \left(  \frac{2M_4 p}{nq\Delta} \right)^{j/2}\right)^2 \\
& \leq  2^{6K} q \Delta  \left( \frac{c^{K+1}-1}{c-1} \right)^2 \\
& \leq \frac{c^2}{(c-1)^2} (8c)^{2K} q\Delta \\
& \lesssim B^K \frac{c_1 \ln n}{n} q,
\end{align}
where $c = \max\{\sqrt{2}, 2\sqrt{c_2/c_1}\}$, and $B>0$ is some universal constant. Recall that $K = c_2 \ln n, \Delta = \frac{c_1 \ln n}{n}$. It suffices to show $\sqrt{\frac{2M_4 p}{nq\Delta}} \leq c$ to complete the proof. Indeed, we have
\begin{align}
\sqrt{\frac{2p}{nq \Delta}\cdot K} & \leq \sqrt{\frac{2K (q + \sqrt{q\Delta})}{nq\Delta}} \\
&  \leq \sqrt{\frac{2K \cdot 2q}{nq\Delta}} \\
& \leq \sqrt{\frac{4K}{n\Delta}}  \\
& = \sqrt{\frac{4c_2}{c_1}} \\
\sqrt{\frac{2p}{nq \Delta} \cdot \frac{n(p - q)^2}{p}} & = \sqrt{\frac{2(p - q)^2}{q\Delta}} \\
& \leq \sqrt{\frac{2q \Delta}{q\Delta}} \\
& = \sqrt{2}.
\end{align}
\end{enumerate}

\subsection{Proof of Lemma~\ref{lemma.small1overslowerboundapproximation}}

It is clear that $\sup_{x\in [0,1]}|f(x;a)|\leq 1$. Introduce 
\begin{align}
f_\eta(x;a) & = \frac{|\eta + (1-\eta) x -a|-a}{\eta + (1-\eta)x} \\
& = \frac{\left| x - \frac{a-\eta}{1-\eta} \right| - \frac{a}{1-\eta} }{\frac{\eta}{1-\eta} + x},
\end{align}
where $\eta = \frac{a}{D}, D>1$. We have $E_L[f(x;a);[\eta,1]] = E_L[f_\eta(x;a);[0,1]]$. Recall the second-order Ditzian--Totik modulus of smoothness given in~(\ref{eqn.2nddtmodulus})
\begin{align}
\omega_\varphi^2(f,t) & = \sup_{0<h\leq t} \sup_{x} |\Delta_{h\varphi}^2 f(x)|,
\end{align}
where $\varphi = \sqrt{x(1-x)}, \Delta_{h\varphi}^2 f(x) = f(x+h\varphi) + f(x-h\varphi)- 2f(x)$. 

We deal with the two cases separately. 
\begin{enumerate}
\item $\frac{1}{L^2}\leq a\leq \frac{1}{2}$: Denote $\delta = \frac{1}{DL} \varphi\left(  \frac{a-\eta}{1-\eta} \right)$. It is easy to verify that if $\frac{1}{L^2}\leq a\leq \frac{1}{2}, D\geq 3$, then $ \frac{1}{1+(DL)^2} \leq   \frac{a-\eta}{1-\eta} \leq \frac{(DL)^2}{1+(DL^2)}$, which ensures that $\frac{a-\eta}{1-\eta} \pm \delta \in [0,1]$. We lower bound $\omega_\varphi^2(f,t)$ for $f_\eta$ as follows:
\begin{align}
\omega_\varphi^2(f_\eta, (DL)^{-1}) & \geq \left |\Delta_{\varphi/(DL)}^2 f_\eta \left( \frac{a-\eta}{1-\eta}  \right ) \right |. 
\end{align}
Since
\begin{align}
\Delta_{\varphi/(DL)}^2 f_\eta \left( \frac{a-\eta}{1-\eta}  \right ) & = \frac{2\delta}{\delta + \frac{a}{1-\eta}},
\end{align}
we have
\begin{align}
\omega_\varphi^2(f_\eta, (DL)^{-1}) & \geq \frac{2\delta}{\delta + \frac{a}{1-\eta}}. 
\end{align}

The relationship between $\omega_\varphi^2(f,\frac{1}{n})$ and $E_n[f;[0,1]]$ was shown in~\cite[Thm. 7.2.4.]{Ditzian--Totik1987} that there exists a universal positive constant $M_2$ such that
\begin{align}
\frac{1}{n^2} \sum_{k = 0}^n (k+1) E_k[f_\eta;[0,1]] \geq M_2 \omega_\varphi^2(f_\eta, \frac{1}{n}). 
\end{align} 
Utilizing the non-increasing property of $E_n[f_\eta;[0,1]]$ with respect to $n$ yields
\begin{align}
& E_L[f_\eta;[0,1]] \nonumber \\
& \geq \frac{1}{DL- L} \sum_{k = L+1}^{DL} E_k[f_\eta;[0,1]] \\
& \gtrsim \frac{1}{(DL)^2} \sum_{k = L+1}^{DL} (k+1) E_k[f_\eta;[0,1]] \\
& \geq M_2 \omega_\varphi^2(f_\eta, \frac{1}{DL}) - \frac{1}{(DL)^2} \sum_{k = 0}^L (k+1) E_k[f_\eta;[0,1]].  \label{eqn.lowerboundLorderfeta}
\end{align}

Now we work out an upper bound on $E_k[f_\eta; [0,1]]$. It follows from Lemma~\ref{lemma.ditziantotikgeneral} that there exists a universal constant $M_1$ such that
\begin{align}
E_k[f_\eta;[0,1]] & \leq M_1 \omega_\varphi^1(f_\eta, \frac{1}{k}),
\end{align}
where $\omega_\varphi^1(f, t) = \sup_{0<h\leq t} \sup_x | \Delta_{h\varphi}^1 f(x) |$, where $\Delta_{h\varphi}^1 f(x) = f(x+h\varphi/2) - f(x-h\varphi/2)$. It follows from straightforward algebra that $\omega_\varphi^1(f_\eta, \frac{1}{k}) \lesssim \frac{1}{k\sqrt{a}}$. Hence,
\begin{align}
\frac{1}{(DL)^2} \sum_{k = 0}^L (k+1) E_k[f_\eta;[0,1]] & \lesssim \frac{1}{D^2 L^2} \sum_{k = 0}^L  \frac{1}{\sqrt{a}} \\
& \lesssim \frac{1}{D^2  L \sqrt{a}}. 
\end{align}

Since $\delta = \frac{1}{DL} \varphi\left(  \frac{a-\eta}{1-\eta} \right), \eta = \frac{a}{D}, \frac{1}{L^2}\leq a\leq \frac{1}{2}$, for $D$ large enough, we know $\delta \gtrsim \frac{\sqrt{a}}{DL}$ and there exist two universal constants $M_3>0, M_4 >0$ such that
\begin{align}
E_L[f_\eta;[0,1]] & \gtrsim M_3 \frac{\frac{\sqrt{a}}{DL}}{\frac{\sqrt{a}}{DL} + a} - M_4 \frac{1}{D^2 L\sqrt{a}} \\
& = \frac{1}{D} \left(  M_3 \frac{\sqrt{a}}{L} \frac{1}{\frac{\sqrt{a}}{DL} + a} - M_4 \frac{1}{D L\sqrt{a}} \right) \\
& \gtrsim \frac{1}{D} \left(  M_3 \frac{\sqrt{a}}{L} \frac{1}{2a} - M_4 \frac{1}{D L\sqrt{a}} \right),
\end{align}
where we used the fact that $\frac{\sqrt{a}}{DL} \leq a$ for $a\geq \frac{1}{L^2}, D\geq 1$. Hence,
\begin{align}
L \sqrt{a} \cdot E_L[f_\eta;[0,1]] & \gtrsim \frac{1}{D} \left( \frac{M_3}{2} - \frac{M_4}{D} \right)  \\
& \geq d_1 >0
\end{align}
when $D$ is large enough. 
\item $0<a<\frac{1}{L^2}$: for $D>1$ we have
\begin{align}
\omega_\varphi^2(f_\eta, (DL)^{-1}) & \geq \left |\Delta_{\epsilon}^2 f_\eta \left( \frac{a-\eta}{1-\eta}  \right ) \right |,
\end{align}
where $\epsilon = \min \left \{ \frac{1}{DL} \varphi\left( \frac{a-\eta}{1-\eta} \right) , \frac{a-\eta}{1-\eta} \right \}$.

Since
\begin{align}
\Delta_{\epsilon}^2 f_\eta \left( \frac{a-\eta}{1-\eta}  \right ) & = 1 + f_\eta\left( \frac{a-\eta}{1-\eta} + \epsilon \right) \\
& \geq 0,
\end{align}
it suffices to lower bound $f_\eta\left( \frac{a-\eta}{1-\eta} + \epsilon \right)$ to lower bound $\omega_\varphi^2(f_\eta, (DL)^{-1})$. Note that the function $f_\eta(\cdot)$ is a non-decreasing function. 

We have $ \frac{1}{DL} \varphi\left( \frac{a-\eta}{1-\eta} \right) \gtrsim \frac{\sqrt{a}}{DL}$ for $D$ large enough, and 
\begin{align}
1 + f_\eta\left( \frac{a-\eta}{1-\eta} + \frac{a-\eta}{1-\eta} \right) & = \frac{2D-2}{2D-1} \\
1 + f_\eta\left( \frac{a-\eta}{1-\eta} + \frac{1}{DL} \varphi\left( \frac{a-\eta}{1-\eta} \right) \right)  & \gtrsim \frac{\sqrt{a}}{\sqrt{a} + DL a} \\
& \gtrsim \frac{1/L}{1/L + D/L} \\
& = \frac{1}{1+D},
\end{align}
where we used the fact that the function $\frac{\sqrt{x}}{\sqrt{x} + DL x}$ is a non-increasing function for $x\geq 0$. 

Hence, we have shown that 
\begin{align}
\omega_\varphi^2(f_\eta, (DL)^{-1}) & \gtrsim \min\left\{ \frac{2D-2}{2D-1}, \frac{1}{D+1} \right \}. 
\end{align}

Following~(\ref{eqn.lowerboundLorderfeta}), we have that
\begin{align}
& E_L[f_\eta;[0,1]] \nonumber \\
& \quad \geq M_2 \omega_\varphi^2(f_\eta, \frac{1}{DL}) - \frac{1}{(DL)^2} \sum_{k = 0}^L (k+1) E_k[f_\eta;[0,1]] \\
& \quad \gtrsim \frac{1}{D} - \frac{1}{(DL)^2} \sum_{k = 0}^L (k+1) \\
& \quad\gtrsim \frac{1}{D} - \frac{1}{D^2} \\
& \quad \geq d_1 >0,
\end{align}
when $D$ is large enough. Here we used the fact that $E_k[f_\eta;[0,1]]\leq 1$. 
\end{enumerate}

\subsection{Proof of Lemma~\ref{lemma.mixturepoissontvbound}}

We have
\begin{align}
\mathsf{TV}(F_0,F_1) & = \frac{1}{2} \sum_{j = 0}^\infty \frac{1}{j!} \left| \mathbb{E}[e^{-U_0} U_0^j-e^{-U_1} U_1^j] \right |. 
\end{align}
For each $j\geq 0, j\in \mathbb{Z}$, we introduce function
\begin{align}
f_j(x) & = e^{-(a+Mx)} (a+Mx)^j,
\end{align}
where $x\in [-1,1]$. We introduce $a+M X_i = U_i, i = 0,1$. It follows from the assumptions that $\mathbb{E}[X_0^j] = \mathbb{E}[X_1^j], 0\leq j\leq L$. 

We write the series expansion of $f_j(x)$ as follows:
\begin{align}
f_j(x) = f_j(0) + \sum_{k  =1}^\infty \frac{f_j^{(k)}(0)}{k!} x^k.
\end{align}
Hence,
\begin{align}
\mathsf{TV}(F_0,F_1) & = \frac{1}{2} \sum_{j = 0}^\infty \frac{1}{j!} \left|  \sum_{k\geq L+1} \frac{f_j^{(k)}(0)}{k!} \mathbb{E}(X_0^k - X_1^k ) \right | \\
& \leq \frac{1}{2} \sum_{j = 0}^\infty \sum_{k\geq L+1} \frac{2}{j!} \left | \frac{f_j^{(k)}(0)}{k!} \right | \\
& =  \frac{1}{2} \sum_{k\geq L+1} \frac{2}{k!} \sum_{j = 0}^\infty \left | \frac{f_j^{(k)}(0)}{j!} \right |, 
\end{align}
where we used the fact that $X_i \in [-1,1], i= 0,1$.

It follows from the Leibniz formula for derivatives of products of functions that 
\begin{align}
& f_j^{(k)}(x)\nonumber \\
& = e^{-a} \sum_{m = 0}^{k} {k \choose m} (e^{-Mx})^{(k-m)} ( (a+Mx)^j )^{(m)} \\
& = e^{-a} \sum_{m = 0}^{k \wedge j} {k \choose m} (-M)^{k-m} e^{-Mx} \frac{j!}{(j-m)!} M^{m} (a+Mx)^{j-m} \\
& = e^{-(a+Mx)} M^{k} \sum_{m = 0}^{k \wedge j} {k \choose m}  (-1)^{k-m} \frac{j!}{(j-m)!}  (a+Mx)^{j-m} \\
& = e^{-(a+Mx)} M^{k} (a+Mx)^{j-k} \times \nonumber \\
& \qquad \left( \sum_{m = 0}^{k \wedge j} {k \choose m}  (-1)^{k-m} \frac{j!}{(j-m)!}  (a+Mx)^{k-m} \right).
\end{align}
Hence,
\begin{align}
\frac{f_j^{(k)}(0)}{j!}  & = e^{-a} \frac{a^j}{j!} \left(\frac{M}{a} \right)^k \sum_{m = 0}^{k \wedge j} {k \choose m}  (-1)^{k-m} \frac{j!}{(j-m)!}  a^{k-m}.
\end{align}
Construct random variable $ Z \sim \spo(a)$. Then, 
\begin{align}
\frac{f_j^{(k)}(0)}{j!}  & = \left( \frac{M}{a} \right)^k \bP(Z = j) \sum_{m = 0}^k {k\choose m} (-a)^{k-m} (j)_m, 
\end{align}
where $(j)_m = j (j-1)\cdots (j-m+1)$. 

Consequently, 
\begin{align}
\sum_{j = 0}^\infty \left | \frac{f_j^{(k)}(0)}{j!} \right | & \leq \left( \frac{M}{a} \right)^k  \mathbb{E} \left|  \sum_{m = 0}^k {k\choose m} (-a)^{k-m} (Z)_m \right | \\
& = \left( \frac{M}{a} \right)^k \mathbb{E}| g_{k, a}(Z)|,
\end{align}
where $g_{k,q}(Z)$ is the estimator introduced in Lemma~\ref{lemma.middlevariancebound} for the case of $n = 1$.

It follows from Lemma~\ref{lemma.middlevariancebound} that
\begin{align}
\sum_{j = 0}^\infty \left | \frac{f_j^{(k)}(0)}{j!} \right | & \leq \left( \frac{M }{a} \right)^k \sqrt{ \mathbb{E} (g_{k, a}(Z))^2 } \\
& \leq  \left( \frac{M }{a} \right)^k  \sqrt{a^k k!} \\
& \leq \left( \frac{M }{a} \right)^k  \sqrt{ a^k k^k} \\
& \leq \left( \frac{M}{a} \sqrt{ak} \right)^k \\
& = \left( M \sqrt{\frac{k}{a}} \right)^k.  
\end{align}

Hence, it follows from $k! \geq \frac{k^k}{e^k}$ that
\begin{align}
\mathsf{TV}(F_0,F_1) & \leq \sum_{k\geq L+1} \frac{1}{k!} \left( M \sqrt{\frac{k}{a}} \right)^k \\
& \leq \sum_{k\geq L+1} \frac{1 }{k^k} \left(  e M \sqrt{\frac{k}{a}} \right)^k \\
& \leq \sum_{k\geq L+1} \left( \frac{eM}{\sqrt{ka}} \right)^k  \\
& \leq \sum_{k\geq L+1} \left( \frac{e M}{\sqrt{a(L+1)}} \right)^k. 
\end{align}
It follows from the assumptions that
\begin{align}
\frac{eM}{\sqrt{a(L+1)}} \leq \frac{1}{2}. 
\end{align}
Consequently, 
\begin{align}
\mathsf{TV}(F_0,F_1) & \leq     \left( \frac{e M}{\sqrt{a(L+1)}} \right)^{L+1} \left( 1 -  \left( \frac{e M}{\sqrt{a(L+1)}} \right) \right)^{-1} \\
& \leq 2 \left( \frac{e M}{\sqrt{a(L+1)}} \right)^{L+1}. 
\end{align}

\subsection{Proof of Lemma~\ref{lemma.approximateprobability}}

We define the minimax risk under the multinomial sampling model for a fixed $Q$ as 
\begin{align}
R(S,n,Q) = \inf_{\hat{L}} \sup_{P \in \mathcal{M}_S} \mathbb{E}_P \left( \hat{L} - \|P-Q\|_1 \right)^2. 
\end{align}

Fix $\delta>0$. Let $\hat{L} = \hat{L}(X_1,X_2,\ldots,X_S)$ be a near-minimax estimator of $\|P-Q\|_1$ under the multinomial model for every sample size $n$, which means that for every sample size $n$,  
\begin{align}
\sup_{P \in \mathcal{M}_S} \mathbb{E} \left( \hat{L} - \|P-Q\|_1 \right)^2 < R(S,n,Q) + \delta. 
\end{align}
Here the random vector $(X_1,X_2,\ldots,X_S)$ follows multinomial distribution parametrized by $n,P$, and the estimator $\hat{L}$ obtains the number of samples $n$ from this random vector. 

Now we consider the Poisson sampling model, where $X_i$'s are mutually independent with marginal distributions $X_i \sim \spo(n p_i)$. Let $n' = \sum_{i = 1}^S X_i \sim \spo(n \sum_{i = 1}^S p_i)$. We use the estimator $\hat{L}(X_1,X_2,\ldots,X_S)$ to estimate $\|P-Q\|_1$ under the Poisson sampling model. For any $P \in \mathcal{M}_S(\epsilon)$ under the Poisson sampling model, we have
\begin{align}
& \mathbb{E}_P \left( \hat{L} - \|P-Q\|_1 \right)^2 \nonumber \\
& \leq \mathbb{E}_P \Bigg(  \hat{L} - L_1\left( \frac{P}{\sum_{i = 1}^S p_i},Q\right ) \nonumber \\
& \qquad \qquad + L_1 \left( \frac{P}{\sum_{i = 1}^S p_i}, Q \right) - \|P-Q\|_1 \Bigg)^2 \\
& \leq 2 \mathbb{E}_P \left( \hat{L} - L_1\left( \frac{P}{\sum_{i = 1}^S p_i},Q\right ) \right)^2  \nonumber \\
& \qquad \qquad + 2 \left( L_1 \left( \frac{P}{\sum_{i = 1}^S p_i}, Q \right) - \|P-Q\|_1\right)^2 \\
& \leq 2 \mathbb{E}_P \left( \hat{L} - L_1\left( \frac{P}{\sum_{i = 1}^S p_i},Q\right ) \right)^2 + 2 \epsilon^2,
\end{align}
where we used the fact that $(a+b)^2  \leq 2a^2 + 2b^2$ for any $a,b\in \mathbb{R}$, and the fact that if $\sum_{i = 1}^S p_i = A$, then
\begin{align}
& \left| L_1 \left( \frac{P}{\sum_{i = 1}^S p_i}, Q \right) - \|P-Q\|_1 \right| \nonumber \\
&  \quad\leq  \sum_{i = 1}^S | |p_i/A - q_i| - |p_i - q_i| | \\
& \quad \leq \sum_{i = 1}^S |p_i/A - p_i| \\
& \quad = \sum_{i = 1}^S \frac{p_i}{A}|1-A| \\
& \quad = |A-1|\\
& \quad \leq \epsilon. 
\end{align}

Then,
\begin{align}
& \mathbb{E}_P \left( \hat{L} - \|P-Q\|_1 \right)^2 \nonumber \\
 & \leq 2\epsilon^2 + \nonumber \\
 & \quad 2 \sum_{m = 0}^\infty \mathbb{E}_P\left[ \left( \hat{L} - \left \| \frac{P}{\sum_{i = 1}^S p_i}-Q\right\|_1 \right)^2 \Bigg | n' = m \right] \mathbb{P}(n'=m)  \\
& \leq 2 \sum_{m = 0}^\infty R(S,m,Q) \mathbb{P}(n' = m) + 2(\delta + \epsilon^2) \\
& \leq 2 \Bigg(  1\cdot \mathbb{P}(n'\leq n(1-\epsilon)/2) + R(S, n(1-\epsilon)/2, Q) \times \nonumber \\
& \qquad \mathbb{P}(n'\geq n(1-\epsilon)/2) \Bigg) + 2(\delta + \epsilon^2)\\
& \leq 2R(S,n(1-\epsilon)/2, Q) + 2\mathbb{P}(n' \leq n(1-\epsilon)/2) + 2(\delta + \epsilon^2) \\
& \leq 2R(S,n(1-\epsilon)/2, Q) + 2\mathbb{P}( \spo(n(1-\epsilon)) \leq n(1-\epsilon)/2) \nonumber \\
& \qquad + 2(\delta + \epsilon^2) \\ 
& \leq 2R(S,n(1-\epsilon)/2, Q) + 2 e^{-n(1-\epsilon)/8} + 2(\delta + \epsilon^2),
\end{align}
where we used the fact that conditioned on $n' = m$, $(X_1,X_2,\ldots,X_S)$ follows multinomial distribution parametrized by $\left( m, \frac{P}{\sum_{i = 1}^S p_i} \right)$, the monotonicity of $R(S,m,Q)$ as a function of $m$, $R(S,m,Q)\leq 1$, and Lemma~\ref{lemma.poissontail}. 

Taking supremum of $\mathbb{E}_P \left( \hat{L} - \|P-Q\|_1 \right)^2$ over $\mathcal{M}_S(\epsilon)$ and using the arbitrariness of $\delta$, we have
\begin{align}
R_P(S,n,Q,\epsilon) & \leq 2R(S, n(1-\epsilon)/2, Q) + 2 e^{-n(1-\epsilon)/8} + 2\epsilon^2,
\end{align}
which is equivalent to
\begin{align}
R(S, n(1-\epsilon)/2, Q) & \geq \frac{1}{2}R_P(S,n,Q,\epsilon) - e^{-n(1-\epsilon)/8} - \epsilon^2. 
\end{align}

It follows from~\cite[Lemma 16]{Jiao--Venkat--Han--Weissman2015minimax} that $R(S,n,Q) \leq 2 R_P(S,n/2,Q,0)$. Hence,
\begin{align}
& R_P(S, n(1-\epsilon)/4, Q, 0 ) \nonumber \\
& \quad \geq \frac{1}{2}R(S, n(1-\epsilon)/2, Q) \\
& \quad \geq \frac{1}{4}R_P(S,n,Q,\epsilon) - \frac{1}{2} e^{-n(1-\epsilon)/8} - \frac{1}{2} \epsilon^2. 
\end{align}

\subsection{Proof of Lemma~\ref{lemma.2dperformancesquare}}

We first analyze the bias. To simplify the notation we denote $\Delta = \frac{c_1 \ln n}{n}$. It follows from the definition of $\tilde{P}_K^{(1)}$ that for $(p,q) \in \left[ 0, 2\Delta\right]^2$, 
\begin{align}
\mathbb{E} \tilde{P}_K^{(1)}(\hat{p},\hat{q}) - |p-q| & = 2\Delta h_{2K} \left( \frac{p}{2\Delta}, \frac{q}{2\Delta} \right)- |p-q|,
\end{align}
where $h_{2K}(x,y) = u_K(x,y) v_K(x,y) -u_K(0,0) v_K(0,0) $, and $u_K(x,y)$ and $v_K(x,y)$ satisfy~(\ref{eqn.defuv}). 

We first argue that there exists a universal constant $M>0$ such that $\sup_{(x,y)\in [0,1]^2} | u_K(x,y) v_K(x,y) - |x-y| | \leq M \left( \frac{\sqrt{x} + \sqrt{y}}{K} + \frac{1}{K^2} \right)$. Indeed, 
\begin{align}
& \left | u_K(x,y) v_K(x,y) - |x-y| \right | \nonumber \\
 & = \Bigg | u_K(x,y) v_K(x,y) - u_K(x,y) |\sqrt{x}-\sqrt{y}| \nonumber \\
 & \quad + u_K(x,y) |\sqrt{x} - \sqrt{y}| - (\sqrt{x} + \sqrt{y}) |\sqrt{x} - \sqrt{y}| \Bigg | \\
& \leq |u_K(x,y)| | v_K(x,y) - |\sqrt{x} - \sqrt{y}| | \nonumber \\
& \quad + |\sqrt{x}-\sqrt{y}| |u_K(x,y) - \sqrt{x} - \sqrt{y}| \\
& \leq |u_K(x,y) - (\sqrt{x}+\sqrt{y})| | v_K(x,y) - |\sqrt{x} - \sqrt{y}| | \nonumber \\
& \quad + (\sqrt{x} + \sqrt{y}) | v_K(x,y) - |\sqrt{x} - \sqrt{y}| |  \nonumber \\
& \quad + |\sqrt{x}-\sqrt{y}| |u_K(x,y) - \sqrt{x} - \sqrt{y}|.
\end{align}
It follows from Lemma~\ref{lemma.ditziantotikgeneral} and Lemma~\ref{lemma.firstorderdtmoduindependentofa} that the best polynomial approximation error of $(\sqrt{x} + \sqrt{y})$ and $|\sqrt{x} - \sqrt{y}|$ over the unit square are both of order $\frac{1}{K}$. Hence,
\begin{align}
\left | u_K(x,y) v_K(x,y) - |x-y| \right | & \leq  M \left( \frac{\sqrt{x} + \sqrt{y}}{K} + \frac{1}{K^2} \right),
\end{align}
which implies that there exists another constant $M>0$ such that
\begin{align}
& \left | u_K(x,y) v_K(x,y)-u_K(0,0) v_K(0,0)  - |x-y| \right | \nonumber \\
& \quad \leq  M \left( \frac{\sqrt{x} + \sqrt{y}}{K} + \frac{1}{K^2} \right),
\end{align}

Denote $x = \frac{p}{2\Delta}, y = \frac{q}{2\Delta}$, we have
\begin{align}
& \left|\mathbb{E} \tilde{P}_K^{(1)}(\hat{p},\hat{q}) - |p-q| \right | \nonumber \\
& = \left| 2\Delta h_{2K} \left( \frac{p}{2\Delta}, \frac{q}{2\Delta} \right)- |p-q| \right | \\
& = 2\Delta \left| h_{2K} \left( \frac{p}{2\Delta}, \frac{q}{2\Delta} \right)- \left| \frac{p}{2\Delta} - \frac{q}{2\Delta} \right| \right | \\
& = 2\Delta | h_{2K}(x,y) - |x-y| | \\
& \leq 2\Delta M\left( \frac{\sqrt{x} + \sqrt{y}}{K} + \frac{1}{K^2} \right) \\
& = 2\Delta M \frac{1}{K} \left( \sqrt{\frac{p}{2\Delta}} + \sqrt{\frac{q}{2\Delta}} + \frac{1}{K} \right) \\
& \lesssim \frac{1}{K} \sqrt{\frac{c_1 \ln n}{n}} (\sqrt{p} + \sqrt{q}) + \frac{1}{K^2} \frac{c_1 \ln n}{n}. 
\end{align}

We now analyze the variance. Express the polynomial $h_{2K}(x,y) \in \poly_{2K}^2$ explicitly as
\begin{align}
h_{2K}(x,y) & = \sum_{0\leq i\leq 2K, 0\leq j \leq 2K, i+j\geq 1} h_{ij} x^i y^j \\
& = \sum_{0\leq i \leq 2K}  \left( \sum_{0\leq j\leq 2K, i+j\geq 1} h_{ij} y^j  \right) x^i.
\end{align}
For any fixed value of $y$, $h_{2K}(x^2, y^2)$ is a polynomial of $x$ with degree no more than $4K$ that is uniformly bounded by a universal constant on $[-1,1]$. It follows from Lemma~\ref{lem.polycoeff} that for any fixed $y\in [-1,1]$, 
\begin{align}
| \sum_{0\leq j \leq 2K} h_{ij} y^{2j} | \leq M (\sqrt{2}+1)^{4K},
\end{align}
which, together with Lemma~\ref{lem.polycoeff}, implies that 
\begin{align}
|h_{ij}| \leq M (\sqrt{2}+1)^{8K}. 
\end{align}
Since $\tilde{P}_K^{(1)}$ is the unbiased estimator of $2\Delta h_{2K}\left( \frac{p}{2\Delta}, \frac{q}{2\Delta} \right)$, we know
\begin{align}
\tilde{P}_K^{(1)}(\hat{p},\hat{q}) & = \sum_{0\leq i,j\leq 2K, i+j\geq 1} h_{ij} (2\Delta)^{1-i-j} g_{i,0}(\hat{p}) g_{j,0}(\hat{q}),
\end{align}
where $g_{j,q}(\hat{p})$ is the unbiased estimator for $(p-q)^j$ introduced in Lemma~\ref{lemma.middlevariancebound}. 

Denote $\| X\|_2 = \sqrt{\mathbb{E}(X - \mathbb{E}X)^2}$ and $M_1  = 2K \vee 2n\Delta$. Using the triangle inequality of the norm $\| \cdot \|_2$ and Lemma~\ref{lemma.middlevariancebound}, we know
\begin{align}
& \| \tilde{P}_K^{(1)}(\hat{p},\hat{q}) \|_2 \nonumber \\
& \leq \sum_{0\leq i,j\leq 2K, i+j\geq 1} |h_{ij}| (2\Delta)^{1-i-j} \| g_{i,0}(\hat{p}) \|_2 \| g_{j,0}(\hat{q}) \|_2 \\
& \leq \sum_{0\leq i,j\leq 2K, i+j\geq 1} M (\sqrt{2}+1)^{8K} (2\Delta) \times \nonumber \\
& \qquad \qquad \left( \frac{1}{2\Delta} \sqrt{\frac{2M_1 p}{n}} \right)^i \left( \frac{1}{2\Delta} \sqrt{\frac{2M_1 q}{n}} \right)^j \\
& \lesssim (\sqrt{2}+1)^{8K} \frac{c_1 \ln n}{n} \sum_{0\leq i,j\leq 2K, i+j\geq 1} \left(\sqrt{\frac{p}{2\Delta}} \right)^i \left( \sqrt{\frac{q}{2\Delta}} \right)^j 
\end{align}
Since for any $x\in [0,1],y\in [0,1]$,
\begin{align}
& \left |\sum_{0\leq i,j\leq 2K, i+j\geq 1} x^i y^j\right | \nonumber \\
& \quad \leq \left| \sum_{j=1}^{2K} y^j \right | + \left| \sum_{i = 1}^{2K} x^i \right | + xy \left |\sum_{0\leq i,j\leq 2K-1} x^i y^j \right |\\
& \quad \leq y(2K) + x(2K) + x y (2K)^2 \\
& \quad \leq 2(2K)^2(x+y) 
\end{align}
we know
\begin{align}
\| \tilde{P}_K^{(1)}(\hat{p},\hat{q}) \|_2 & \lesssim (\sqrt{2}+1)^{8K} \frac{c_1 c_2^2 (\ln n)^3}{n} \left( \sqrt{\frac{p}{2\Delta}} + \sqrt{\frac{q}{2\Delta}} \right) \\
& \lesssim \sqrt{B^K \frac{c_1 c_2^4 \ln^5 n}{n}(p+q)}.
\end{align}
for some constant $B>0$. Hence,
\begin{align}
\mathsf{Var}(\tilde{P}_K^{(1)}(\hat{p},\hat{q}) ) \lesssim B^K \frac{ c_1 c_2^4 (p+q) \ln^5 n }{n}. 
\end{align}

\subsection{Proof of Lemma~\ref{lemma.2dperformancepqlarge}}

We first analyze the bias. It follows from the definition of $\tilde{P}_K^{(2)}$ that 
\begin{align}
\mathbb{E} \left[ \tilde{P}_K^{(2)}(\hat{p},\hat{q}; x,y)  \right ] & = \sum_{j = 0}^K r_j W^{-j+1} (p-q)^j,
\end{align}
where $W = \sqrt{\frac{8c_1 \ln n}{n}} \sqrt{ (x+y) \vee \frac{1}{n}}$. 

Since $(p+q) \in U$, we know 
\begin{align}
|p-q| & \leq \sqrt{\frac{2c_1 \ln n}{n}} (\sqrt{p} + \sqrt{q}) \\
& \leq \sqrt{\frac{2c_1 \ln n}{n}} \sqrt{2} (\sqrt{p+q}) \\
& \leq \sqrt{\frac{2c_1 \ln n}{n}} \sqrt{2} \sqrt{2(x+y)} \\
& \leq W,
\end{align}
where we have used the fact that $\sqrt{p} + \sqrt{q} \leq \sqrt{2(p+q)}$ and the assumption that $p+q \leq 2(x+y)$. 

Hence, it follows from the property that the best degree-$K$ polynomial approximation error of $|t|$ over $[-1,1]$ is $\Theta(\frac{1}{K})$~\cite[Chap. 9, Thm. 3.3]{Devore--Lorentz1993} that 
\begin{align}
\left | \sum_{j = 0}^K r_j W^{-j+1} (p-q)^j - |p-q| \right | & \lesssim \frac{W}{K} \\
& \lesssim \frac{1}{K}\sqrt{\frac{c_1 \ln n}{n}} \sqrt{x+y}. 
\end{align}

Then we analyze the variance. It was shown in Cai and Low~\cite[Lemma 2]{Cai--Low2011} that $|r_j| \leq 2^{3K}, 0\leq j \leq K$. Denote the unbiased estimator of $(p-q)^j$ by $\hat{A}_j(\hat{p},\hat{q})$ and introduce the norm $\| X \|_2 = \sqrt{\mathbb{E}(X - \mathbb{E}[X])^2}$. It follows from the triangle inequality of the norm $\| X \|_2$ and the fact that constants have zero variance that
\begin{align}
\| \tilde{P}_K^{(2)} \|_2 & \leq \sum_{j = 1}^K |r_j| W^{-j+1} \| \hat{A}_j \|_2.
\end{align}
It follows from Lemma~\ref{lemma.middleconstraintvariancework} that 
\begin{align}
\mathbb{E}\hat{A}_j^2 \leq  \left(  2(p-q)^2 \vee \frac{8j (p\vee q)}{n} \right)^j. 
\end{align}
Hence, 
\begin{align}
\| \tilde{P}_K^{(2)} \|_2 & \leq 2^{3K} W \sum_{j =1}^K \left(  \frac{\sqrt{2}|p-q|}{W}   \vee \frac{\sqrt{8j (p\vee q)}}{\sqrt{n} W} \right)^j \\
& = 2^{3K} W \sum_{j = 1}^K C^j,
\end{align}
where 
\begin{align}
C & = \frac{\sqrt{2}|p-q|}{W}   \vee \frac{\sqrt{8j (p\vee q)}}{\sqrt{n} W}  \\
& \leq \frac{\sqrt{2} \sqrt{\frac{2c_1 \ln n}{n}} (\sqrt{p} + \sqrt{q})}{\sqrt{\frac{8c_1 \ln n}{n}} \sqrt{x+y}} \vee \frac{\sqrt{8 K (p+q)}}{\sqrt{n} \sqrt{\frac{8c_1 \ln n}{n}} \sqrt{x+y}} \\
& \leq \frac{\sqrt{p} + \sqrt{q}}{\sqrt{2} \sqrt{x+y}} \vee \sqrt{\frac{c_2}{c_1}} \frac{\sqrt{p+q}}{\sqrt{x+y}} \\
& \leq \sqrt{2} \vee \sqrt{\frac{2c_2}{c_1}} \\
& \leq \sqrt{2}. 
\end{align}
Consequently, 
\begin{align}
\| \tilde{P}_K^{(2)} \|_2 & \leq 2^{3K} \sqrt{\frac{8c_1 \ln n}{n}} \sqrt{x+y} K (\sqrt{2})^K \\
& \lesssim \sqrt{B^K \frac{(x+y)c_1 \ln n}{n}},
\end{align}
where $B$ is a constant.

\section{Proofs of auxiliary lemmas}\label{sec.proofofauxiliarylemmas}

\subsection{Proof of Lemma~\ref{lemma.firstorderdtmoduindependentofa}}

We split the analysis of $|f(x+h\varphi/2) - f(x-h\varphi/2)|$ into two cases:
\begin{enumerate}
\item $x - \frac{h\varphi}{2} \geq a$ or $x + \frac{h\varphi}{2} \leq a$: in this case,
\begin{align}
& |f(x+h\varphi/2) - f(x-h\varphi/2)| \nonumber \\
& = \sqrt{ x + h\varphi/2} - \sqrt{x - h\varphi/2} \\
& = \frac{h\varphi}{\sqrt{x + h\varphi/2} + \sqrt{x - h\varphi/2}} \\
& \leq \frac{h\varphi}{\sqrt{x + h\varphi/2 + x - h\varphi/2}} \\
& = \frac{h\sqrt{1-x}}{\sqrt{2}} \\
& \leq \frac{t}{\sqrt{2}},
\end{align}
where we have used the fact that $\sqrt{x} + \sqrt{y}\geq \sqrt{x+y}$ and $0<h\leq t$. 
\item $x - \frac{h\varphi}{2} < a < x + \frac{h\varphi}{2}$: in this case
\begin{align}
& |f(x+h\varphi/2) - f(x-h\varphi/2)| \nonumber \\
& = \left| \sqrt{x + h\varphi/2} + \sqrt{x - h\varphi/2} - 2\sqrt{a} \right | \\
& \leq \max \{ \sqrt{x + h\varphi/2} - \sqrt{a}, \sqrt{a} - \sqrt{x - h\varphi/2} \} \\
& \leq \sqrt{x + h\varphi/2} - \sqrt{x - h\varphi/2} \\
& \leq \frac{t}{\sqrt{2}}. 
\end{align}
\end{enumerate}

\subsection{Proof of Lemma~\ref{lemma.dtmodulusknownq}}
It follows from taking derivatives that for convex function $f(x)$, the function $f(x-t) -2f(x) + f(x+t)$ is a nondecreasing function of $t$. Since $f(x) = |2x\Delta-q|$ is a convex function, it follows from straightforward algebra that 
\begin{align}
& \omega_{\varphi}^2(f,K^{-1})  \nonumber \\
& = \max\{ \max_{z \in [\frac{1}{1+K^2}, \frac{K^2}{1+K^2}]} A_1(z) , \max_{z < \frac{1}{1+K^2}} A_2(z) , \max_{z > \frac{K^2}{1+K^2}} A_3(z) \},
\end{align}
where 
\begin{align}
A_1(z) & = f\left( z - \frac{\sqrt{z(1-z)}}{K} \right) -2 f(z) + f\left(z + \frac{\sqrt{z(1-z)}}{K} \right) \\
A_2(z) & = f(0) - 2f(z) + f(2z) \\
A_3(z) & = f(2z-1) - 2f(z) + f(1). 
\end{align}

%
%
%
%
%

We break the proof into three parts. 
\begin{enumerate}
\item We first prove that when $\frac{1}{1+K^2} \leq \frac{q}{2\Delta} \leq \frac{K^2}{1+K^2}$, the maximum of achieved by $A_1(z)$ at $z = \frac{q}{2\Delta}$. 

Consider first the case $\frac{1}{1+K^2} \leq z \leq \frac{K^2}{1+K^2}$ and function $A_1(z)$. If $z>\frac{q}{2\Delta}$, without loss of generality we can assume $z - \frac{\sqrt{z(1-z)}}{K} < \frac{q}{2\Delta}$, since otherwise $A_1(z) = 0$. Then,
\begin{align}\label{eqn.secondorderzbig}
A_1(z) & = q - 2\Delta \left( (z - \frac{\sqrt{z(1-z)}}{K} \right) -2 (2z \Delta - q) \nonumber \\
& \quad + 2 \left(z + \frac{\sqrt{z(1-z)}}{K} \right)\Delta - q \\
& = 4\Delta \left( \frac{\sqrt{z(1-z)}}{K} - z  + \frac{q}{2\Delta}\right). 
\end{align}
Taking derivative with respect to $z$, it suffices to show this derivative is non-positive when $\frac{1}{1+K^2} \leq \frac{q}{2\Delta} \leq \frac{K^2}{1+K^2}, z \geq \frac{q}{2\Delta}, z - \frac{\sqrt{z(1-z)}}{K} < \frac{q}{2\Delta}$. We have the derivative expressed as
\begin{align}
& 4\Delta \left( \frac{1-2z}{2K \sqrt{z(1-z)}} - 1\right) \nonumber \\
& \quad = 4\Delta \left( \frac{1-2z-2K \sqrt{z(1-z)}}{2K \sqrt{z(1-z)}} \right)
\end{align}
Since $1-2z-2K \sqrt{z(1-z)}$ is a convex function, it achieves its maximum at the endpoints. When we set $z = \frac{1}{1+K^2}$ and $z = 1$ it is both negative. Similar arguments work for the case of $z < \frac{q}{2\Delta}$. Hence, we conclude that when $\frac{1}{1+K^2} \leq \frac{q}{2\Delta} \leq \frac{K^2}{1+K^2}$, 
\begin{align}
\max_{\frac{1}{1+K^2} \leq z \leq \frac{K^2}{1+K^2}}A_1(z)  & = \frac{2 \sqrt{q(2\Delta - q)}}{K}.
\end{align}

Consider the case $z > \frac{K^2}{1+K^2}$ and the function $A_3(z)$. It suffices to assume $2z-1 \leq \frac{q}{2\Delta
}$ since otherwise $A_3(z) = 0$. In this case
\begin{align}
A_3(z) & = q - (2z-1)2\Delta - 2(2z\Delta - q) + 2\Delta - q \\
& = 4\Delta + 2q - 8z\Delta,
\end{align}
which is a decreasing function in $z$, implying $\max_{z > \frac{K^2}{1+K^2}} A_3(z) \leq \max_{\frac{1}{1+K^2} \leq z \leq \frac{K^2}{1+K^2}}A_1(z)$. Similar arguments work for the $z<\frac{1}{1+K^2}$ and $A_2(z)$ case. 

\item We now prove that when $\frac{q}{2\Delta} \leq \frac{1}{1+K^2}$, the maximum is achieved by $A_2(z)$ at $\frac{q}{2\Delta}$. 

In this case, it suffices to consider $z \leq \frac{1}{1+K^2}$. Indeed, if $z> \frac{1}{1+K^2}$, then the second order difference in the non-zero case is given by (\ref{eqn.secondorderzbig}), which is shown to be a decreasing function when $z > \frac{1}{1+K^2}$. Now consider $z \leq \frac{1}{1+K^2}$. We discuss three cases separately:

\begin{enumerate}
\item $2z \leq \frac{q}{2\Delta}$: in this case, $A_2(z) = 0$.
\item $z \leq  \frac{q}{2\Delta} \leq 2z$: in this case, 
\begin{align}
A_2(z) & = q - 2 (q - 2z\Delta) + 4z\Delta - q \\
& = 8z\Delta - 2q,
\end{align}
which is an increasing function of $z$. It implies that in this regime one should take $z = \frac{q}{2\Delta}$. The resulting $A_2(z)$ is $2q$. 
\item $\frac{q}{2\Delta} \leq z$: in this case, the second order difference is
\begin{align}
q - 2(2z\Delta - q) + 4z\Delta -q = 2q,
\end{align} 
which is independent of $z$. 
\end{enumerate}

Hence, we have shown that for $\frac{q}{2\Delta} \leq \frac{1}{1+K^2}$, the maximum is achieved by $A_2(z)$ and
\begin{align}
\max_{z\leq \frac{1}{1+K^2}}A_2(z) = 2q. 
\end{align}

\item The case of $2\Delta \geq \frac{q}{2\Delta} \geq \frac{K^2}{1+K^2}$ can be dealt with in a fashion similar to the case of $\frac{q}{2\Delta} \leq \frac{1}{1+K^2}$, resulting in
\begin{align}
\max_{z\geq \frac{K^2}{1+K^2}} A_3(z) = 2(2\Delta-q). 
\end{align}
\end{enumerate}

\subsection{Proof of Lemma~\ref{lemma.evenpolyapproximation}}

It suffices to show that for any polynomial $Q \in \poly_{2K}$, 
\begin{align}
\sup_{z\in [-1,1]} |Q(z) - f(z^2)| & \geq \sup_{z\in [-1,1]} |P_K(z^2) - f(z^2)| \\
& = \sup_{z\in [0,1]} |P_K(z^2) - f(z^2)|
\end{align}

Define
\begin{align}
e(z) & = \frac{Q(z) + Q(-z)}{2} \\
o(z) & = \frac{Q(z) - Q(-z)}{2}.
\end{align}
It is clear that $e(z)$ is an even function, $o(z)$ is an odd function, and $Q(z) = e(z) + o(z)$. We have
\begin{align}
& \sup_{z\in [-1,1]} |f(z^2) - Q(z)| \nonumber \\
& = \sup_{z\in [-1,1]} |f(z^2) - e(z) - o(z)| \\
& = \sup_{z\in [0,1]} \max\{ |f(z^2) - e(z) - o(z)|, |f(z^2) - e(z) + o(z)|\} \\
& \geq \sup_{z\in [0,1]} \left( |f(z^2) - e(z) - o(z)|+ |f(z^2) - e(z) + o(z)| \right)/2\\
& \geq \sup_{z\in [0,1]} \left | \frac{(f(z^2) - e(z) - o(z)) + (f(z^2) - e(z) + o(z))}{2} \right| \\
& = \sup_{z\in [0,1]} |f(z^2) -e(z)|, 
\end{align}
where we have used the fact that $\max\{a,b\}\geq \frac{a+b}{2}$ and the convexity of the function $|z|$.  

There exists another polynomial $U_K(z) \in \poly_K$ such that $U_K(z^2) = e(z)$. Hence, for any $Q \in \poly_{2K}$, 
\begin{align}
\sup_{z\in [-1,1]} |f(z^2) - Q(z)| & \geq \sup_{z\in [0,1]} |f(z^2) -U_K(z^2)| \\
& = \sup_{z\in [0,1]} |f(z) -U_K(z)| \\
& \geq \sup_{z\in [0,1]} |f(z) - P_K(z)| \\
& = \sup_{z\in [-1,1]} |f(z^2) - P_K(z^2)| \\
\end{align}
where we used the definition of $P_K(z)$. The proof is complete. 

\subsection{Proof of Lemma~\ref{lemma.middlevariancebound}}

The Charlier polynomial $c_n(x,a),a>0$ is defined as follows:
\begin{align}
c_n(x,a) = \sum_{r=0}^n (-1)^{n-r} {n \choose r} \frac{(x)_r}{a^r},
\end{align}
where $(x)_r = x \cdot (x-1) \cdot \cdots \cdot (x-r+1)$ is the falling factorial. It satisfies the following generating function relation~\cite{Peccati--Taqqu2011some}:
\begin{align}\label{eqn.charliergenerating}
\sum_{n = 0}^\infty \frac{c_n(x,a)}{n!} t^n = e^{-t} \left( 1+ \frac{t}{a} \right)^x, x\in \mathbb{N}.
\end{align}
Substituting $t$ by $at$, we have
\begin{align}\label{eqn.charliergeneral}
\sum_{n = 0}^\infty \frac{a^n c_n(x,a)}{n!} t^n = e^{-at} \left( 1+ t \right)^x, x\in \mathbb{N}.
\end{align}

Note that we have
\begin{align}
a^n c_n(x,a) = \sum_{r=0}^n (-1)^{n-r} {n \choose r} a^{n-r}  (x)_r,
\end{align}
which is well defined even for $a = 0$. If $a = 0$, then $a^n c_n(x,a)$ may be defined as
\begin{align}
a^n c_n(x,a) \triangleq (x)_n. 
\end{align}
We note that relation (\ref{eqn.charliergeneral}) is true also when $a = 0$. Indeed, the case $a = 0$ reduces to the relation:
\begin{align}
\sum_{n = 0}^\infty \frac{(x)_n}{n!} t^n = (1+t)^x, x\in \mathbb{N}. 
\end{align}

Assuming $Y \sim \mathsf{Poi}(\lambda)$, replacing $x$ with random variable $Y$ in (\ref{eqn.charliergeneral}) and taking expectation on both sides, we have
\begin{align}
\sum_{n = 0}^\infty \frac{\mathbb{E}a^n c_n(Y,a)}{n!} t^n & = e^{-at} \mathbb{E}\left( 1+ t \right)^Y  \\
& = e^{t (\lambda-a)} \\
& = \sum_{n = 0}^\infty  \frac{\left( \lambda-a\right)^n}{n!} t^n
\end{align}

Note that $\mathbb{E}a^n c_n(Y,a)$ does not depend on $t$. Hence we know
\begin{align}
\mathbb{E}a^n c_n(Y,a) =  \left(\lambda-a\right)^n. 
\end{align}

Thus, if $nX \sim \mathsf{Poi}(np),a = nq$, we have
\begin{align}
\mathbb{E} q^j c_j(nX, nq) = (p-q)^j, j\geq 0. 
\end{align}

Expanding $q^j c_j(nX,nq)$ implies that it is equal to $g_{j,q}(X)$ defined in Lemma~\ref{lemma.middlevariancebound}. The estimator $g_{j,q}(X)$ being the unique uniformly minimum variance unbiased estimator of $(p-q)^j$ follows from the Lehmann--Scheffe Theorem~\cite[Chap. 2, Thm. 1.11]{Lehmann--Casella1998theory} and the complete sufficiency of $X$ in model $nX \sim \mathsf{Poi}(np)$(\cite[Chap. 1, Thm. 6.22]{Lehmann--Casella1998theory}). 

Now we proceed to bound the second moment of $g_{j,q}(X)$. It follows from (\ref{eqn.charliergeneral}) that for any $a+b\geq 0$, 
\begin{align}
\sum_{n = 0}^\infty \frac{(a+b)^n c_n(x,a+b)}{n!} t^n = e^{-(a+b)t} \left( 1+ t \right)^x, x\in \mathbb{N},
\end{align}
which implies that
\begin{align}
& \sum_{n = 0}^\infty \frac{(a+b)^n c_n(x,a+b)}{n!} t^n \nonumber \\
& = e^{-bt} \sum_{n = 0}^\infty \frac{a^n c_n(x,a)}{n!} t^n \\
& = \left(\sum_{j = 0}^\infty \frac{(-bt)^j}{j!}  \right) \left( \sum_{n = 0}^\infty \frac{a^n c_n(x,a)}{n!} t^n \right) \\
& = \sum_{j = 0}^\infty \sum_{n = 0}^\infty \frac{a^n c_n(x,a)(-b)^j}{n!j!}t^{n+j}
\end{align}
It follows from coefficient matching that
\begin{align}
\frac{(a+b)^n c_n(x,a+b)}{n!} & = \sum_{k = 0}^n \frac{a^k c_k(x,a) (-b)^{n-k}}{k!(n-k)!} \nonumber \\
& = \sum_{k = 0}^n {n \choose k}   \frac{a^k c_k(x,a) (-b)^{n-k}}{n!},
\end{align}
which simplifies to 
\begin{align}
(a+b)^j c_j(x,a+b) = \sum_{k = 0}^j {j \choose k} (-b)^{j-k} a^k c_k(x,a). 
\end{align}

Now assume $nX \sim \mathsf{Poi}(np)$. Taking $a+b = nq, a = np$ and dividing both sides by $n^j$, we have
\begin{align}
q^j c_j(nX,nq) & = \frac{1}{n^j} \sum_{k = 0}^j {j \choose k} (n(p-q))^{j-k} (np)^k c_k(nX,np) \\
& = \sum_{k = 0}^j {j \choose k} (p-q)^{j-k} p^k c_k(nX,np).
\end{align}

The Charlier polynomials are orthogonal with respect to the Poisson measure. Concretely, for $Y \sim \mathsf{Poi}(\lambda)$~\cite{Peccati--Taqqu2011some}, 
\begin{align}
\mathbb{E} c_n(Y,\lambda)c_m(Y,\lambda) = \frac{n!}{\lambda^n} \delta_{mn}.
\end{align}
For $nX \sim \mathsf{Poi}(np)$, we have
\begin{align}
\mathbb{E} \left( p^j c_j(nX, np) \right)^2 = \frac{p^j j!}{n^j},
\end{align}
which is also true for $p = 0$. 

Applying the orthogonal property of Charlier polynomials and assuming $p>0$,  we have
\begin{align}
& \mathbb{E} \left( q^j c_j(nX,nq) \right)^2 \nonumber \\
& = \mathbb{E} \sum_{k = 0}^j {j \choose k}^2 (p-q)^{2(j-k)} \mathbb{E} \left( p^k c_k(nX, np) \right)^2 \\
& = \sum_{k = 0}^j {j \choose k}^2 (p-q)^{2(j-k)} \frac{p^k k!}{n^k} \\
& = j! \sum_{k = 0}^j {j \choose k} (p-q)^{2(j-k)} \left( \frac{p}{n}\right)^k \frac{1}{(j-k)!} \\
& = j! \sum_{k = 0}^j {j \choose k} (p-q)^{2k} \left( \frac{p}{n} \right)^{j-k} \frac{1}{k!} \\
& = j! \left( \frac{p}{n} \right)^j \sum_{k = 0}^j {j \choose k} \left[ \frac{n(p-q)^2}{p} \right]^k \frac{1}{k!} \quad \text{Assuming }p>0 \\
& = j! \left( \frac{p}{n} \right)^j L_j\left( - \frac{n(p-q)^2}{p} \right),
\end{align}
where $L_m(x)$ stands for the Laguerre polynomial with order $m$, which is defined as:
\begin{align}
L_m(x) = \sum_{k = 0}^m {m \choose k} \frac{(-x)^k}{k!}
\end{align}

If we further assume $M \geq \max\left \{ \frac{n(p-q)^2}{p}, j  \right \}$, we have
\begin{align}
\mathbb{E} \left( q^j c_j(nX,nq) \right)^2 & \leq  j! \left( \frac{p}{n} \right)^j \sum_{k = 0}^j {j \choose k} \frac{M^k}{k!}  \\
& \leq j! \left( \frac{p}{n} \right)^j \sum_{k = 0}^j {j \choose k} \frac{M^j}{j!}  \\ 
& =\left( \frac{2Mp}{n}\right)^j.
\end{align}

\subsection{Proof of Lemma~\ref{lemma.middleconstraintvariancework}}

It follows from the fact that $\mathbb{E} \prod_{i = 0}^{k-1} (\hat{p}-\frac{i}{n}) = p^k$ for $n\hat{p} \sim \spo(nq)$~\cite[Ex. 2.8]{Withers1987} that $\hat{A}_j$ is unbiased for estimating $(p-q)^j$. It follows from the Lehmann--Scheffe Theorem~\cite[Chap. 2, Thm. 1.11]{Lehmann--Casella1998theory} and the complete sufficiency of $(\hat{p},\hat{q})$ (\cite[Chap. 1, Thm. 6.22]{Lehmann--Casella1998theory}) that $\hat{A}_j$ is the unique uniformly minimum variance unbiased estimator for $(p-q)^j$. 

We now work out a different form of $\hat{A}_j$. It follows from the binomial theorem that for any fixed $r>0$, 
\begin{align}
(p-q)^j & = (p-r - (q-r))^j \\
& = \sum_{k = 0}^k {j\choose k} (p-r)^k (-1)^{j-k}(q-r)^{j-k}.
\end{align}
Clearly, the following estimator is also unbiased for estimating $(p-q)^j$: 
\begin{align}
\sum_{k = 0}^j {j\choose k} g_{k,r}(\hat{p})(-1)^{j-k}g_{j-k,r}(\hat{q}),
\end{align}
where $g_{k,r}(\hat{p})$ and $g_{j-k,r}(\hat{q})$ are the unique uniformly minimum variance unbiased estimators for $(p-r)^k$ and $(q-r)^{j-k}$ introduced in Lemma~\ref{lemma.middlevariancebound}, respectively. It follows from the uniqueness of $\hat{A}_j$ that 
\begin{align}
\hat{A}_j & = \sum_{k = 0}^j {j\choose k} g_{k,r}(\hat{p})(-1)^{j-k}g_{j-k,r}(\hat{q}). 
\end{align}

Using $\| X\|_2 = (\mathbb{E}[X^2])^{1/2}$ and the triangle inequality for the norm $\| X \|_2$, we have
\begin{align}
\| \hat{A}_j\|_2 & \leq \sum_{k = 0}^j {j\choose k} \| g_{k,r}(\hat{p}) g_{j-k,r}(\hat{q}) \|_2 \\
& = \sum_{k = 0}^j {j\choose k} \| g_{k,r}(\hat{p}) \|_2 \| g_{j-k,r}(\hat{q}) \|_2,
\end{align} 
where we have used the independence of $\hat{p}$ and $\hat{q}$ in the last step. 

Define $M_1 = \frac{n(p-r)^2}{p} \vee j$, $M_2 = \frac{n(q-r)^2}{q} \vee j$, and set $r = \frac{p+q}{2}$. Define $M = 2(p-q)^2 \vee \frac{8j (p\vee q)}{n}$. It follows from Lemma~\ref{lemma.middlevariancebound} that 
\begin{align}
\| \hat{A}_j\|_2 & \leq \sum_{k = 0}^j {j\choose k} \left( \frac{2 M_1 p}{n} \right)^{k/2} \left( \frac{2M_2 q}{n} \right)^{(j-k)/2} \\
& = \left( \sqrt{\frac{2M_1 p}{n}} + \sqrt{\frac{2M_2q}{n}} \right)^j \\
& = \left( \sqrt{ \frac{(p-q)^2}{2} \vee \frac{2jp}{n} } + \sqrt{ \frac{(p-q)^2}{2} \vee \frac{2jq}{n} } \right)^j \\
& = M^{j/2}. 
\end{align}

\subsection{Proof of Lemma~\ref{lemma.poissoncentralmomenttight}}

Equation~(\ref{eqn.poissoncentralmomentexpansion}) follows from~\cite[Lemma 9.5.5.]{Ditzian--Totik1987}. Now we prove the bound on the magnitude of $|h_{j,s}|$. Note that the moment generating function of $\hat{p}-p$ is given by
\begin{align}
\mathbb{E}[\exp(z(\hat{p}-p))] = e^{-zp} e^{np(e^{z/n}-1)}
\end{align}
Written as formal power series of $z$, the previous identity becomes
\begin{align}
\sum_{s=0}^\infty \frac{\mathbb{E}(\hat{p}-p)^s}{s!}z^s = \left(\sum_{i=0}^\infty \frac{(-p)^i}{i!}z^i\right)\left[\sum_{k=0}^\infty \frac{n^k p^k}{k!} \left(\sum_{l=1}^\infty \frac{1}{l!}(\frac{z}{n})^l\right)^{k}\right].
\end{align}

Hence, by comparing the coefficient of $n^{j-s}z^s$ at both sides, we obtain
\begin{align}
& \frac{h_{j,s} p^j}{s!} \nonumber \\
& = \sum_{i=0}^j \frac{(-p)^i}{i!} \left( \frac{p^{j-i}}{(j-i)!} \sum_{a_1+\cdots+a_{j-i}=s-i, a_1,\cdots,a_{j-i}\ge 1} \prod_{l=1}^{j-i} \frac{1}{a_l!}\right).
\end{align}
Moreover, 
\begin{align}
\sum_{a_1+\cdots+a_{j-i}=s-i, a_1,\cdots,a_{j-i}\ge 1} \prod_{l=1}^{j-i} \frac{1}{a_l!} &\le \sum_{a_1+\cdots+a_{j-i}=s-i} \prod_{l=1}^{j-i} \frac{1}{a_l!} \\
&\le \sum_{a_1+\cdots+a_{j}=s} \prod_{l=1}^{j} \frac{1}{a_l!} = \frac{j^s}{s!}. 
\end{align}

Then,
\begin{align}
|h_{j,s}| & \leq \sum_{i = 0}^j  \frac{1}{i!} \frac{1}{(j-i)!} j^s \\
& = \frac{j^s}{ j!} \sum_{i = 0}^j {j \choose i} \\
& = \frac{2^j j^s}{j!}. 
\end{align}

Since $1\leq j \leq s$, we have
\begin{align}
\frac{2^j j^s}{j!} & \leq 2^s \frac{j^s}{j!} \\
& \leq 2^s \frac{j^s}{\sqrt{2\pi j} (j/e)^j} \\
& \leq 2^s \frac{j^s e^j}{j^j}. 
\end{align}

Now we consider the maximization problem $\max_{x\geq 0} \frac{x^s e^x}{x^x}$. It follows from taking derivatives that this function attains it unique maximum at point $x^*$ which satisfies the following:
\begin{align}
x^* \ln x^* = s. 
\end{align}
Recall the Lambert $W$ function is defined over $[-1/e,\infty)$ by the equation $W(z)e^{W(z)} = z$, we know that
\begin{align}
x^* = e^{W(s)}. 
\end{align}
The following upper bound on $W(s)$ was proved in~\cite{hoorfar2008inequalities}: for any $s>e$, 
\begin{align}
W(s) & \leq \ln s - \ln \ln s + \frac{e}{e-1} \frac{\ln \ln s}{\ln s} \\
& \leq \ln s - \ln \ln s + \frac{1}{e-1}, 
\end{align}
where we have used the fact that $\max_{x>0} \frac{\ln x}{x} = \frac{1}{e}$. 

Hence, for any $s\geq 3$, 
\begin{align}
|h_{j,s}| & \leq (2e)^s e^{sW(s)} \\
& \leq (2e^{e/(e-1)})^s \left( \frac{s}{\ln s} \right)^s,
\end{align}
which turns out to be also correct for $s = 2$ since $h_{1,2} = 1$. 

\subsection{Proof of Lemma~\ref{lemma.inversepoissonmoment}}

It is clear that when $p\leq \frac{1}{n}$, the statement is true. It suffices to consider the case of $p>\frac{1}{n}$. Introduce function $g_n(p)$ as follows:
\begin{align}
g_n(p) & = \begin{cases} \frac{1}{p^j} & p \geq \frac{1}{n} \\ n^j - j n^{j+1}\left( p - \frac{1}{n} \right) & 0\leq p < \frac{1}{n} \end{cases}.
\end{align}
It is evident that $g_n(p)\leq \frac{1}{p^j}$ and 
\begin{align}
g_n(\hat{p}) - \frac{1}{(\hat{p} \vee \frac{1}{n})^j} & = \begin{cases} 0 & \hat{p}\geq \frac{1}{n} \\ j n^j & \hat{p} = 0  \end{cases}
\end{align}
We have
\begin{align}
\left| \mathbb{E} \frac{1}{(\hat{p} \vee \frac{1}{n})^j} \right | & \leq \left| \mathbb{E} \left[ \frac{1}{(\hat{p} \vee \frac{1}{n})^j} - g_n(\hat{p}) \right ]\right | \nonumber \\
& \quad + \left| \mathbb{E} \left[ g_n(p) - g_n(\hat{p}) \right ] \right | + g_n(p) \\
& \leq  j n^j e^{-np} + \frac{1}{p^j} + \left| \mathbb{E} \left[ g_n(p) - g_n(\hat{p}) \right ] \right |.
\end{align}
Since the function $g_n(p)$ is continuously differentiable on $(0,\infty)$, we have
\begin{align}
& \left| \mathbb{E} \left[ (g_n(p) - g_n(\hat{p}))^2 \right ] \right | \nonumber \\
& = \left| \mathbb{E} \left[ (g_n(p) - g_n(\hat{p}) )^2\mathbbm{1}(\hat{p}\geq p/2) \right ] \right | + \nonumber \\
& \qquad \left| \mathbb{E} \left[ (g_n(p) - g_n(\hat{p}))^2 \mathbbm{1}(\hat{p}\leq  p/2) \right ] \right | \\
& \leq \sup_{\xi \geq p/2}|g_n'(\xi)|^2 \mathbb{E} (p-\hat{p})^2 + \sup_{\xi >0 } |g_n'(\xi)|^2 p^2 \bP(\hat{p} \leq p/2) \\
& \leq \frac{j^2}{(p/2)^{2j+2}} \frac{p}{n} + j^2 n^{2j+2} p^2 e^{-np/8},
\end{align}
where we applied Lemma~\ref{lemma.poissontail} in the last step. Hence,
\begin{align}
\left| \mathbb{E} \frac{1}{(\hat{p} \vee \frac{1}{n})^j} \right | & \leq j n^j e^{-np} + \frac{1}{p^j} + \sqrt{\mathbb{E}(g_n(p) - g_n(\hat{p}))^2} \\
& \leq j n^j e^{-np} + \frac{1}{p^j} + \frac{j}{(p/2)^{j+1}} \sqrt{\frac{p}{n}} + j n^{j+1} p e^{-np/16} \\
& \leq j n^j e^{-np} + \frac{1}{p^j} + \frac{j 2^{j+1}}{p^{j}}  + j n^{j+1} p e^{-np/16},
\end{align}
where in the last step we used the the assumption that $p\geq \frac{1}{n}$. Consequently,
\begin{align}
\left| \mathbb{E} \frac{p^j}{(\hat{p} \vee \frac{1}{n})^j} \right | & \leq j (np)^j e^{-np} + 1 + j 2^{j+1} + j (np)^{j+1} e^{-np/16} \\
& \leq j \left( \frac{j}{e} \right)^j + 1 + j 2^{j+1} + j \left( \frac{16(j+1)}{e} \right)^{j+1},
\end{align}
where we have used the fact that for any $p\geq 0$, 
\begin{align}
(np)^ke^{-cnp} \le \left(\frac{k}{ec}\right)^k.
\end{align}

\subsection{Proof of Lemma~\ref{lemma.poissonmlebias}}

The following upper bound is straightforward: 
\begin{align}
\bE |\hat{q} - q| & \leq \sqrt{\bE |\hat{q}-q|^2} \\
& = \sqrt{\frac{q}{n}}. 
\end{align}

Regarding the other upper bound and the lower bound, we utilize the exact analytic expression~\cite{Diaconis--Zabell1991closed}  for $\mathbb{E}|X-\lambda|$ for $X \sim \spo(\lambda)$. It follows from~\cite{Diaconis--Zabell1991closed} that for random variable $X\sim \mathsf{Poi}(\lambda)$, 
\begin{align}\label{eqn.closedformpoisson}
\mathbb{E}|X-\lambda| = 2\lambda \frac{e^{-\lambda}\lambda^{[\lambda]}}{[\lambda]!},
\end{align}
where $[\lambda]$ denotes the greatest integer less than or equal to $\lambda$. 

When $0<\lambda\leq 1$, we have
\begin{align}
\mathbb{E}|X-\lambda| & = 2\lambda e^{-\lambda},
\end{align}
which implies that if $0<q \leq  \frac{1}{n}$,
\begin{align}
\mathbb{E}|\hat{q} - q| & = 2qe^{-nq} \\
& \in [2q e^{-1}, 2q]. 
\end{align}

Regarding the final lower bound, it suffices to show that for $X \sim \mathsf{Poi}(\lambda), \lambda\geq 1$, we have
\begin{align}
\mathbb{E}|X-\lambda| \geq \sqrt{\frac{\lambda}{2}}. 
\end{align}

Hence, it suffices to show
\begin{align}
2\sqrt{2\lambda} e^{-\lambda} \frac{\lambda^{[\lambda]}}{[\lambda]!} \geq 1
\end{align}
for all $\lambda\geq 1$. It is equivalent to
\begin{align}
2\sqrt{2\lambda} e^{-\lambda} \frac{\lambda^{n}}{n!} \geq 1
\end{align}
for $\lambda \in [n,n+1)$ for all the integers $n\geq 1$. 

Since the function $2\sqrt{2\lambda} e^{-\lambda} \frac{\lambda^{n}}{n!}$ is monotonically increasing for $\lambda \in [n, n+1/2]$, and monotonically decreasing for $\lambda \in [n+1/2, n+1)$, it suffices to consider integers $\lambda$. Hence, it suffices to show for any integer $n\geq 1$, 
\begin{align}
2\sqrt{2n}e^{-n} \frac{n^n}{n!} & \geq 1 \\
2\sqrt{2(n+1)}e^{-n-1} \frac{(n+1)^n}{n!} & \geq 1,
\end{align}
which is equivalent to
\begin{align}\label{eqn.finalreducestirling}
n! & \leq \sqrt{8n} \left( \frac{n}{e}\right)^n. 
\end{align}
It follows from \cite{robbins1955remark} that for any positive integer $n$, 
\begin{align}
n! < \sqrt{2\pi} e^{\frac{1}{12n}} \sqrt{n} \left( \frac{n}{e} \right)^n,
\end{align}
which implies~(\ref{eqn.finalreducestirling}) since $\sqrt{2\pi} e^{\frac{1}{12n}} < \sqrt{8}$ for all positive integers. 

\subsection{Proof of Lemma~\ref{lemma.dtxabsolutevaluefunction}}

We first assume $q\geq p$. Applying the relation
\begin{align}
x  & = (x)_+ - (x)_{-} \\
|x| & = (x)_+ + (x)_{-} 
\end{align}
where $(x)_+ = \max\{x,0\}, (x)_{-} = -\min\{x,0\}$, we have
\begin{align}
\left | \mathbb{E}|\hat{q}-p| - |q-p| \right | & = \left | \mathbb{E}|\hat{q}-p| - (q-p) \right | \\
& = \left | \mathbb{E}(\hat{q}-p) - (q-p) + 2\mathbb{E}(\hat{q}-p)_{-} \right | \\
& = 2\mathbb{E}(\hat{q}-p)_{-}. 
\end{align}
Construct random variable $\hat{p}$ such that $n\hat{p}\sim \mathsf{Poi}(np)$ is on the same probability space as $\hat{q}$, with the relationship $n\hat{q} = n\hat{p} + Z$, where $Z$ is independent of $\hat{p}$ and $Z\sim \mathsf{Poi}(n(q-p))$. Hence, $\hat{q}\geq \hat{p}$ with probability one. We have
\begin{align}
2\mathbb{E}(\hat{q}-p)_{-} & \leq 2\mathbb{E}(\hat{p}-p)_{-} \\
& = \mathbb{E}|\hat{p}-p| - \mathbb{E}(\hat{p} - p) \\
& = \mathbb{E}|\hat{p}-p| \\
& \leq 2 \cdot \min \left\{ p, q, \sqrt{\frac{q}{n}}, \sqrt{\frac{p}{n}} \right\},
\end{align}
where we applied Lemma~\ref{lemma.poissonmlebias} in the last step.  The case of $q\leq p$ can be proved analogously. 

Regarding the lower bound, we have
\begin{align}
\sup_{q \geq 0}\left| \mathbb{E}|\hat{q}-p| - |q-p| \right | & \geq \mathbb{E}|\hat{p} - p| \\
& \geq \frac{1}{\sqrt{2}} \left( p \wedge \sqrt{\frac{p}{n}} \right),
\end{align}
where we lower bound via taking $q = p$ and using Lemma~\ref{lemma.poissonmlebias}. 

\subsection{Proof of Lemma~\ref{lemma.poissonmlevariance}}
For $n\hat{q} \sim \mathsf{Poi}(nq)$, 
\begin{align}
\mathsf{Var}(|\hat{q}-p|) & = \inf_a \mathbb{E} \left( |\hat{q}-p| -a \right)^2 \\
& \leq \mathbb{E} \left( |\hat{q}-p| - |q-p| \right)^2 \\
& \leq \mathbb{E} |\hat{q}-q|^2 \\
& = \frac{q}{n},
\end{align}
where we used the fact that $||a|-|b||\leq |a-b|$. 

\bibliographystyle{IEEEtran}
\bibliography{di}

\begin{IEEEbiographynophoto}{Jiantao Jiao}
(S'13) received the B.Eng. degree with the highest honor in Electronic Engineering from Tsinghua University, Beijing, China in 2012, and a Master's degree in Electrical Engineering from Stanford University in 2014. He is currently working towards the Ph.D. degree in the Department of Electrical Engineering at Stanford University. He is a recipient of the Stanford Graduate Fellowship (SGF). His research interests include information theory and statistical signal processing, with applications in communication, control, computation, networking, data compression, and learning. 
\end{IEEEbiographynophoto}

\begin{IEEEbiographynophoto}{Yanjun Han}
(S'14) received his B.Eng. degree with the highest honor in
Electronic Engineering from Tsinghua University, Beijing, China in 2015, and a Master's degree in Electrical Engineering from Stanford University in 2017. He is currently working towards the Ph.D. degree in the Department of
Electrical Engineering at Stanford University. His research interests include information theory and statistics, with applications in communications, data compression, and learning.
\end{IEEEbiographynophoto}

\begin{IEEEbiographynophoto}{Tsachy Weissman}
(S'99-M'02-SM'07-F'13) graduated summa cum laude with a
B.Sc. in electrical engineering from the Technion in 1997, and earned
his Ph.D. at the same place in 2001. He then worked at Hewlett-Packard
Laboratories with the information theory group until 2003, when he joined
Stanford University, where he is Professor of Electrical
Engineering and incumbent of the
STMicroelectronics chair in the School of Engineering.
He has spent leaves at the Technion, and at ETH Zurich.

Tsachy's research is focused on information theory, statistical signal
processing, the interplay between them, and their applications.

He is recipient of several best paper awards, and prizes for excellence in research.

He served on the editorial board of the \textsc{IEEE Transactions on Information Theory} from Sept. 2010 to Aug. 2013, and currently serves on the editorial board of Foundations and Trends in Communications and Information Theory.
\end{IEEEbiographynophoto}

\end{document}